\documentclass[amsfonts]{article}
\usepackage{marginnote}

\usepackage{mathrsfs}
\usepackage{amsmath} 
\usepackage{amssymb }
\usepackage{amsthm}
\usepackage{url}
\usepackage{graphicx}		% to display figures 
\usepackage{epsfig,psfrag}  % to display figures
\usepackage{enumerate}
\newtheorem{lemma}{Lemma}[section]
\newtheorem{prop}[lemma]{Proposition}
\newtheorem{theorem}{Theorem}
\newtheorem{cor}[lemma]{Corollary}

\theoremstyle{definition}
\newtheorem{definition}[lemma]{Definition}
\newtheorem{rmrk}[lemma]{Remark}
\newtheorem{example}[lemma]{Example}

\newcommand{\Mob}{\mathrm{Mob}}

\newcommand{\PSLtc}{\mathrm{PSL}_2(\C)}

\newcommand{\PSLt}{\mathrm{PSL}_2(\R)}
\renewcommand{\r}{\mathbf{r}}
\renewcommand{\v}{\mathbf{v}}

\newcommand{\nn}{\mathbf{n}}
\newcommand{\RR}{\mathbf{R}}

\newcommand{\CP}{\mathbb{CP}}
\newcommand{\eps}{{\varepsilon}}

\newcommand{\BOM}{\mathbf{\Omega}}

\newcommand{\II}{\mathrm I}

\newcommand{\tG}{\tilde \Gamma}

\newcommand{\bode}{\eqref{eq:bode}}

\renewcommand{\>}{\rangle}
\newcommand{\<}{\langle}

\newcommand{\x}{{\mathbf x}}

\newcommand{\G}{\Gamma}

\renewcommand{\b}{{\bf b}}

\renewcommand{\d}{{\bf d}}

\renewcommand{\Im}{{\rm Im}}
\renewcommand{\Re}{{\rm Re}}

 \newcommand{\CC}{\mathcal{C}}
 
\newcommand{\RP}{\mathbb{RP}}

\newcommand{\st}{\, | \,}

\newcommand{\R}{\mathbb{R}}

\newcommand{\C}{\mathbb{C}} 
\newcommand{\Z}{\mathbb{Z}}

\renewcommand{\H}{\mathbb{H}}

\newcommand{\DS}{\mathscr{D}}
 
\newcommand{\SL}{\mathrm{SL}}

\newcommand{\GL}{\mathrm{GL}}

\newcommand{\SLt}{{\SL_2(\R)}}

\newcommand{\SO}{\mathrm{SO}}

\newcommand{\slt}{\mathfrak{sl}_2(\R)} 
 
\renewcommand{\sl}{\mathfrak{sl}}

\newcommand{\so}{\mathfrak{so}}

\renewcommand{\mod}{{\rm mod\; }}

\newcommand{\n}{\noindent}

\newcommand{\sn}{\smallskip\n} 

\newcommand{\mn}{\medskip\noindent}

\newcommand{\be}{\begin{equation}}
\newcommand{\ee}{\end{equation}}

\usepackage[all]{xy}

\renewcommand{\d}{\mathrm{d}}
\newcommand{\PSL}{\mathrm{ PSL}}
\newcommand{\SU}{\mathrm{SU}}

\usepackage{color}

\usepackage{bm}

\newcommand{\beq}{\eqref{eq:bode}}

\newcommand{\db}{\boldsymbol{\delta}}

\newcommand{\EE}{{\bf E}}

\newcommand{\xb}{{\bf x}}

\newcommand{\xib}{\boldsymbol{\xi}}

\newcommand{\txib}{{\tilde\xib}}

\newcommand{\Hy}{H^n_\ell}
\newcommand{\Iso}{{\rm Iso}}

\newcommand{\Sp}{S^n_\ell}
\renewcommand{\tG}{\widetilde\Gamma}

\newcommand{\phio}{\phi}

\newcommand{\sutwo}{{\mathfrak{su}}_2}

\title {Tire tracks and integrable curve evolution
}

\author{Gil Bor\footnote{
CIMAT, A.P.~402,
Guanjuato, Gto. 36000,
Mexico;
gil@cimat.mx} 
\and
Mark Levi\footnote{
Department of Mathematics,
Penn State,
University Park, PA 16802, USA;
levi@math.psu.edu}
\and
Ron Perline\footnote{
Department of Mathematics, 
Drexel University, 
3141 Chestnut Street, 
Philadelphia, PA 19104, USA;
rperline@math.drexel.edu} 
\and
Sergei Tabachnikov\footnote{
Department of Mathematics,
Penn  State,
University Park, PA 16802, USA;
tabachni@math.psu.edu}} 
\date{\today}

\begin{document}

\maketitle

\begin{abstract}

We study a simple model of bicycle motion: a segment of fixed length in multi-dimensional Euclidean space, moving so that the velocity of the rear end is always aligned with the segment. If the front track is prescribed, the trajectory of the rear wheel is uniquely determined via a certain first order differential equation -- the bicycle equation. The same model, in dimension two, describes another mechanical device, the hatchet planimeter. 

Here is a sampler of our results. We express the linearized flow of the bicycle equation in terms of the geometry of the rear track; in dimension three, for closed front and rear tracks, this is a version of the Berry phase formula. We show that in all dimensions a sufficiently long bicycle also serves as a planimeter: it measures, approximately, the area bivector defined by the closed front track. We prove that the bicycle equation also describes rolling, without slipping and twisting, of hyperbolic space along Euclidean space. We relate the bicycle problem with two completely integrable systems: the AKNS (Ablowitz, Kaup, Newell and Segur) system and the vortex filament equation. We show that ``bicycle correspondence" of space curves (front tracks sharing a common back track) is a special case of a Darboux transformation associated with the AKNS system. We show that the filament hierarchy, encoded as a single generating equation, describes a 3-dimensional bike of imaginary length. We show that a series of examples of ``ambiguous" closed bicycle curves (front tracks admitting self bicycle correspondence), found recently F. Wegner, are buckled rings, or solitons of the planar filament equation. As a case study, we give a detailed analysis of such curves, arising from bicycle correspondence with multiply traversed circles.

\end{abstract}

\newpage

\tableofcontents

\section{Introduction}

This paper concerns a simple model for  bicycle motion. 
An idealized bike is an oriented segment of fixed length that  moves in such a way that the velocity of the rear end is aligned with the segment: the rear bicycle wheel is fixed on its frame, whereas the front wheel can steer. 
The same ``no skid" non-holonomic constraint describes the bicycle motion in $\R^n$ (and, more generally, in any Riemannian manifold; for example,  hyperbolic and elliptic spaces). 

\paragraph {The bicycle model.} The bicycle model has  attracted much attention in recent years, due in part to its unexpected relations with other mathematical problems, old and new. We start with a brief description of these relations and  recent work on this bicycle model. 
\begin{figure}[h!]
\centerline{  \includegraphics[width=.45\textwidth]{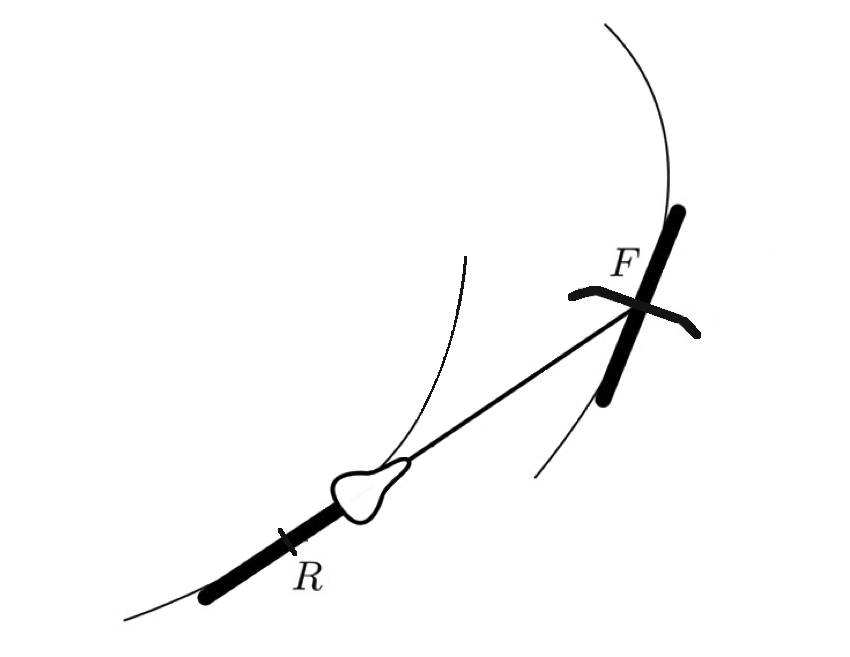}}
\caption{The bicycle front and rear tracks}
\end{figure}

If the front track is prescribed,  the trajectory of the rear wheel is uniquely determined, once the initial orientation of the bicycle is chosen, via a certain first order differential equation, the {\em bicycle equation} (equation \bode\ of Section  \ref{sectmono}). 
In dimension two,  this equation is equivalent to the much studied stationary Schr\"odinger, or Hill, equation $\ddot x + p(t) x =0$, whose potential $p(t)$ depends on the geometry of the front track and the length of the bicycle \cite{Le1,Le2}. 

\paragraph {The bicycle monodromy.} Associated with any given  front track  (closed or not), one   defines the  {\it  bicycle monodromy}, i.e., the map $S^{n-1}\to S^{n-1}$ which assigns to each initial orientation of the bike its final orientation once the front wheel completes its travel.  In dimension two, Foote  \cite{F} observed that this map is a M\"obius transformation; this observation was extended to $\R^n$ in \cite{LT}; we give a new proof  in Theorem \ref{thmf}. 
 
 \paragraph {The hatchet planimeter and Menzin's conjecture.} The bicycle model  in dimension two describes also   a device,  known as the hatchet (or Prytz) planimeter,  for measuring areas of planar domains. The hatchet planimeter consists of  a rod with a hatchet blade fixed at one end and   a pointed pin at the other,  as shown in  Figure \ref{fig:hatchet}. To measure the area  of a planar region, one traces its boundary with the pin; the hatchet slides on the  paper without sideslip, behaving like the rear wheel of a bike.
\begin{figure}[h!]
\centerline{\includegraphics[height=.14\textheight]{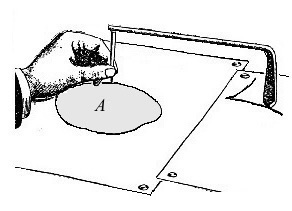} 
\qquad \includegraphics[height=.14\textheight]{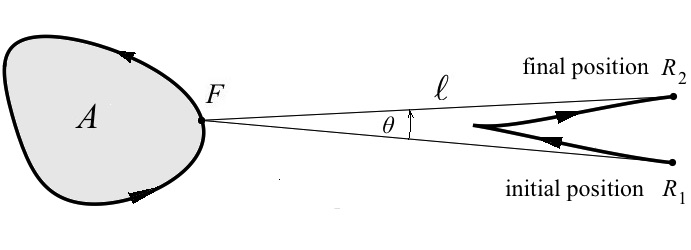}   }
\caption{The hatchet planimeter}\label{fig:hatchet}
\end{figure}

The angle $\theta $ between  the hatchet's initial and final orientations gives an  approximation of the  area $A$  of the region, with an error of order $O(1/\ell)$, 

\be \label{eq:prytzformula}
	A= \ell^2\theta  + O( 1/\ell), 
\ee  
where $ \ell $ is the hatchet's length, see \cite{F,FLT,Hi}. 
 A natural question  is whether this formula is an approximation to some exact result. In Section \ref{sec:berry} we show that indeed $\theta$ is an approximation to the solid angle of a certain cone in $ {\mathbb R}  ^3$.     
 
Planimeters were   popular objects of mathematical study some 100 years ago. In particular,  Menzin (1906) conjectured that if $A > \pi \ell^2$ then the monodromy has a fixed point (that is, for a particular initial orientation of the planimeter   the trajectory of the blade is closed). In other words, the monodromy is a hyperbolic element of the M\"obius group $\PSLt$. 
This conjecture was proved in \cite{LT}; see \cite{FLT,Mac} for expository accounts and \cite{HPZ} for a version of this theorem in spherical and hyperbolic geometries.

\paragraph {Bicycle correspondence.}  A closed rear track   determines {\it  two}  front tracks (one riding forward and the other backward relative to some chosen direction of the rear track). These two front tracks are said to be in the {\it bicycle correspondence}. 
\begin{figure}[h!]
\centerline{\includegraphics[width=.5\textwidth]{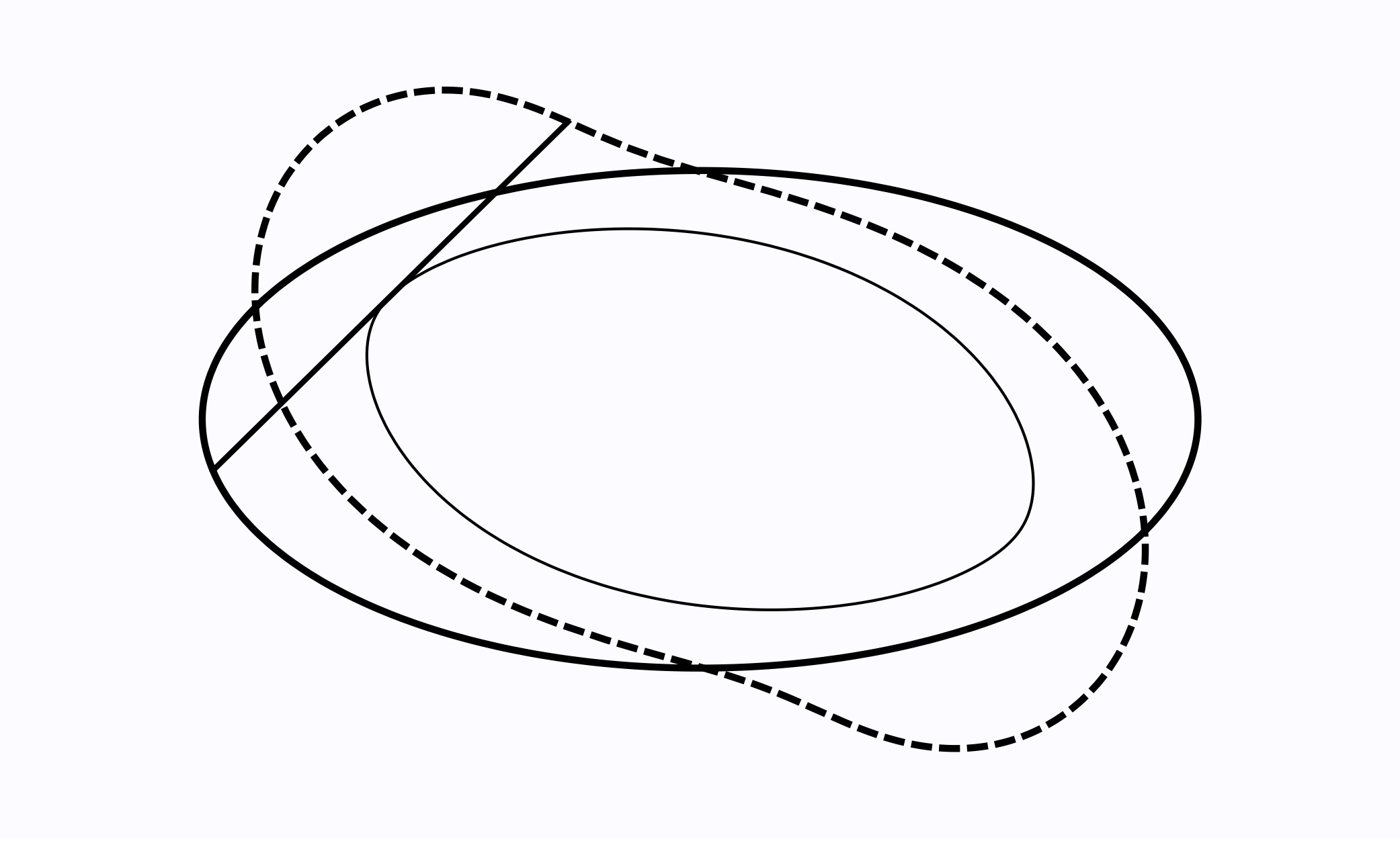}}
\caption{The heavy and dotted curves are in bicycle correspondence; the thin curve is their common back track. }\label{fig:bci}
\end{figure}

Bicycle correspondence of curves has a number of remarkable properties: it satisfies the so-called Bianchi permutability  and it preserves the conjugacy class of the bicycle monodromy (with an arbitrary length of the bicycle, not only the one that defines the bicycle correspondence), see \cite{TT,T2} and Section \ref{ss:bc} below. As a result, bicycle correspondence has infinitely many conserved quantities, starting with the perimeter.

In dimension three, bicycle correspondence is intimately related to the well-studied {\it filament (a.k.a.~binormal, smoke ring, localized induction) equation}, a completely integrable dynamical system on the space of smooth closed curves in $\R^3$, equivalent to the nonlinear Schr\"odinger equation via the Hashimoto transformation \cite{Ha}.
Bicycle correspondence is the Darboux-B\"acklund transformation of the filament equation; it
commutes with the flow of the filament equation and shares with it its integrals and an invariant symplectic structure \cite{T2}. 

\paragraph {Zindler curves.} An interesting problem is whether one can determine the direction of motion given closed rear and front tracks of a bicycle. Usually, this is possible, but sometimes it is not (for example, if the tracks are concentric circles), see \cite{Fi}. The front track in such an ambiguous pair of curves is in  bicycle correspondence with itself; in other words,  two points, $x$ and $y$,  can traverse this curve in such a way that the distance $|xy|$ remains constant  and the velocity of the midpoint of the segment $xy$ is aligned with this segment. Let us call the curves with this property {\it Zindler curves} (see \cite{Zi}). 

Incidentally, Zindler curves provide solutions to another problem, Ulam's problem in  flotation theory (\cite{Sc}, problem 19): which bodies float in equilibrium in all positions?  In the two-dimensional case, the boundary of such a body is a Zindler curve (see \cite{Au,Ru,SK} for  early work\footnote{See \cite{Gu3} for historical information, in particular, about Herman Auerbach (1901--1942).}). Recently, a wealth of results concerning this problem was obtained in \cite{BMO1,BMO2,T1} and in a series of papers by F. Wegner \cite{We1}--\cite{We6}.
\begin{figure}[h!]
\centerline{\includegraphics[width=1\textwidth]{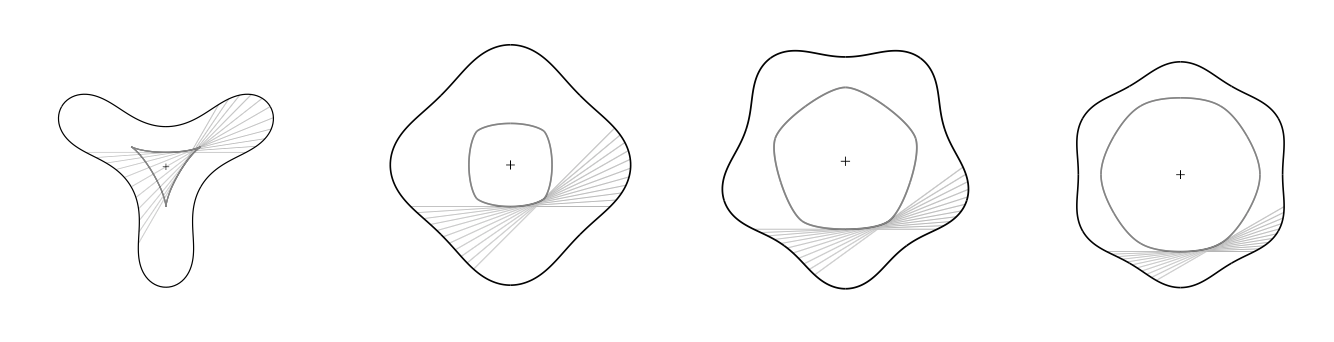}}
\caption{Examples of Zindler curves  from \cite{We6}}
\end{figure}
Wegner constructed a family of non-trivial Zindler curves\footnote{He did not use this terminology.} described explicitly in terms of elliptic functions. He was motivated by a study of the motion of an electron in a magnetic field whose strength  depends quadratically on the distance to the origin. The ``three problems" in \cite{We6} are the ambiguous tire track problem, Ulam's flotation problem, and the motion of an electron. 

A full description of planar Zindler curves, let alone their higher-dimensional version, is still unknown. Let us also mention a discrete version of the bicycle correspondence and, in particular, a polygonal version of Zindler curves \cite{T1,TT}. 

\medskip

\paragraph {Plan of the paper.} 
In Section \ref{sectmono} we discuss various forms of the bicycle differential equation (most of them appeared previously  in the literature), 
%\note{was `litterature'. -M.} 
paying  special attention to the most interesting two- and three-dimensional cases, and give a new proof that the bicycle monodromy is a M\"obius transformation (Theorems \ref{thm:foote2d}-\ref{thmf}). Our goal here is to present a unified, group-theoretic, approach to these foundational matters. 

The geometry of bike tracks in $\R^2 $ is greatly clarified by extension of the problem to $ {\mathbb R}  ^3 $; without such extension some phenomena remain hidden. Theorem \ref{thm:monodromy} (stated for any dimension) is a new result: it describes the derivative of the bicycle monodromy at a fixed point in terms of the geometry of the corresponding closed rear track. In dimension three, one has  the Berry phase formula (Corollary \ref{Berry}): the derivative in question is a complex number whose modulus depends on the signed length of the rear track and whose argument is the Hannay angle, that is, the area on the unit sphere bounded by the tangent Gauss image of the rear track. This  fact is then used to explain geometrically, via Berry's phase, why the planimeter works. A two-dimensional version of the formula for the derivative of the monodromy at the fixed point was obtained in \cite{LT}.

As we mentioned earlier, in the planar case, a sufficiently long bicycle serves as a planimeter. In Theorem  \ref{3Dplan}, we show that a similar fact
holds in higher dimensions: the bicycle measures, approximately, the area {\it  bivector}, determined by the  front track. 

Theorem \ref{thm:roll} of Section \ref{sectmono} gives yet another interpretation of the bicycle equation: 
this equation describes rolling, without slipping and twisting, of the hyperbolic space along Euclidean space, with the front track being the trajectory of the contact point. This interpretation fits naturally with the fact that the bicycle monodromy is a M\"obius transformation,  an isometry of the hyperbolic space.

Section \ref{FilCor} is concerned with the relation of the bicycle problem with the {\em  filament equation}. The  equation defines a flow on the space of smooth closed curves in $\R^3$, a completely integrable Hamiltonian system,  part of an infinite hierarchy of pairwise commuting Hamiltonian vector fields. We start with a detailed description of the notion of bicycle correspondence between curves and give a new proof that this correspondence preserves the conjugacy class of the bicycle monodromy 
(Theorem \ref{thm:bc}). The filament equation shares with the bicycle equation its invariance under  bicycle correspondence, known as the  Darboux, or B\"acklund, transformation, in the context of the filament equation. 

In Section \ref{bikefil}, we encode  the  filament hierarchy   in a single equation with a formal parameter and  show (Corollary \ref{same}) that this equation coincides with the equation of a 3-dimensional bike of {\it  imaginary  length}.

Given a closed front bicycle track, it is intuitively clear that  if the length of the bicycle is infinitesimal, then there exist two closed trajectories of the bicycle, corresponding to the bicycle near-tangent to the front track, pointing either forward or backward. Proposition \ref{intint} provides a rigorous analysis of this phenomenon in dimension 3. As a result, in Theorem \ref{compI}, we obtain an infinite collection of integrals of the bicycle correspondence that, conjecturally, coincide with the known integrals of the filament equation (the Hamiltonians of the commuting hierarchy of vector fileds). 

The classical Bernoulli elastica are extrema of the total squared curvature functional among  curves with  fixed length. 
 {\it Buckled rings} (or {\it pressurized elastica}) are plane curves that are extrema of the total squared curvature functional, subject to  length  and area constraints. In Section \ref{elastica}, we prove that the curves, constructed by Wegner, are buckled rings (Theorem \ref{WegEl}). This provides a connection with the planar filament equation, another completely integrable system, a close relative of the (3-dimensional) filament equation:  buckled rings are solitons of the planar filament equation, that is, evolve under its flow by isometries.

Section \ref{sec:mc} provides a detailed study of a  family of Zindler curves, the ones in   bicycle correspondence with multiply-traversed circles (Theorem \ref{thm:rot}). 
\medskip

The paper is concluded with two  appendices: in appendix  A we describe a relation of  the bicycle equation with yet another integrable system: the AKNS (Ablowitz,  Kaup, Newell, and Segur) system.
We  show (Theorem \ref{thm:new}) that the bicycle correspondence in dimension three can be thought of as a special case of a Darboux transformation associated with  the AKNS system. In appendix B we provide a proof of the main analytical  tool (Proposition \ref{intint}) needed to establish the existence of the integrals of the bicycle  correspondence of Theorem \ref{compI}.

\paragraph{Acknowledgments.} We thank R. Montgomery, J. Langer, L. Hern\'andez,  and F. Wegner for inspiring discussions. GB and RP are grateful to the Department of Mathematics of Penn State for its hospitality. GB was supported by Conacyt grant 222870. RP was supported by the Shapiro Visitor Program. ML and ST were supported by NSF grants DMS-1412542  and DMS-1510055, respectively.

\section{The bicycle equation and  its monodromy} \label{sectmono}
\subsection{The bicycle equation}We consider a smoothly parametrized curve $\Gamma(t)$ in $\R^n$ (the ``front track"), and 
 a real number $\ell>0$ (the ``bicycle length"); a rear track  $\gamma$ is, by the definition, any parametrized curve $\gamma(t)$ in $\R^n$ that satisfies
\begin{align}\label{eq:rig}
&\|\Gamma(t)-\gamma(t)\|=\ell,\\
&\gamma(t)-\Gamma(t) \mbox{ is tangent to $\gamma$ at $\gamma(t)$.} 
\label{eq:ns}\end{align}

\begin{figure}[h!]\centerline{\includegraphics[width=0.4\textwidth]{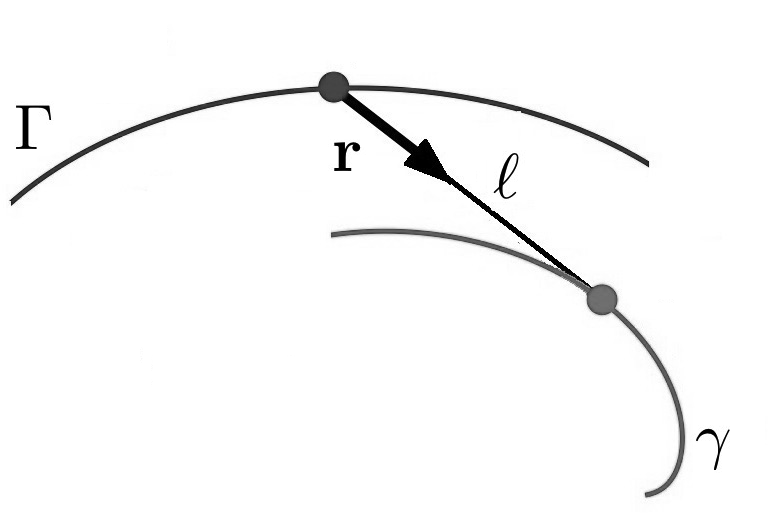}}
\caption{The bicycling ``no skid"  condition   }\label{fig:beq}
\end{figure}

    To keep track of the direction of the rear wheel relative to the front wheel, we introduce the unit direction vector 
$\r(t)\in S^{n-1}$ (see Figure \ref{fig:beq}), thus rewriting condition \eqref{eq:rig}, expressing the bicycle ``rigidity" condition,   as  $\gamma(t)=\Gamma(t)+\ell\r(t)$. Condition \eqref{eq:ns}, expressing the rear wheel  ``no-skid" condition,  is then equivalent to  an ordinary differential equation  for $\r(t)$ which we now state.

\begin{prop}\label{prop:bode}
%
%\note{Made a small change here, moved "where" to after the formula to text-ML. }
%
Let $\Gamma(t), \r(t)$ be parameterized curves in $\R^n, S^{n-1}$, respectively, $\ell>0$, and   $\gamma(t)=\Gamma(t)+\ell\r(t)$. Then  the ``no-skid" condition \eqref{eq:ns}  is equivalent to 
\be\label{eq:bode}
\ell\dot\r=-\v+(\v\cdot\r)\r, 
\ee
where $ \v=\dot\Gamma  $ and where $ \cdot $ denotes the scalar product. 
\end{prop}
Equation \eqref{eq:bode} is the {\em $\ell$-bicycle equation} in $\R^n$, defined for every  parametrized front track  $\Gamma(t)$ and bicycle length $\ell$.

  \begin{proof}  
Let us decompose $\v=\dot\Gamma$ as  $\v=\v^\|+\v^\perp$, 
 where $\v^\|,\v^\perp$ are the orthogonal projections of $\v$ onto 
 $\R\r$, $\r^\perp$, respectively, as in Figure \ref{fig:bepf}. Then conditions \eqref{eq:rig}-\eqref{eq:ns} are equivalent to $\dot\gamma=\v^\|.$ From $\gamma=\Gamma+\ell\r$ follows 
 $\dot\gamma=\v+\ell\dot \r$, hence $\dot\gamma=\v^\|$ is equivalent to  $0=\v^\perp+\ell\dot\r.$ 
  Now $\v^\perp=\v-\v^\|=\v-(\v\cdot\r)\r$,  from which equation \bode\ follows. 
  \end{proof}
  
%\note{Expanded a bit this proof, and added a figure, as requested by referee 2. -G.}
\begin{figure}[h!]
\centerline{\includegraphics[width=0.5\textwidth]{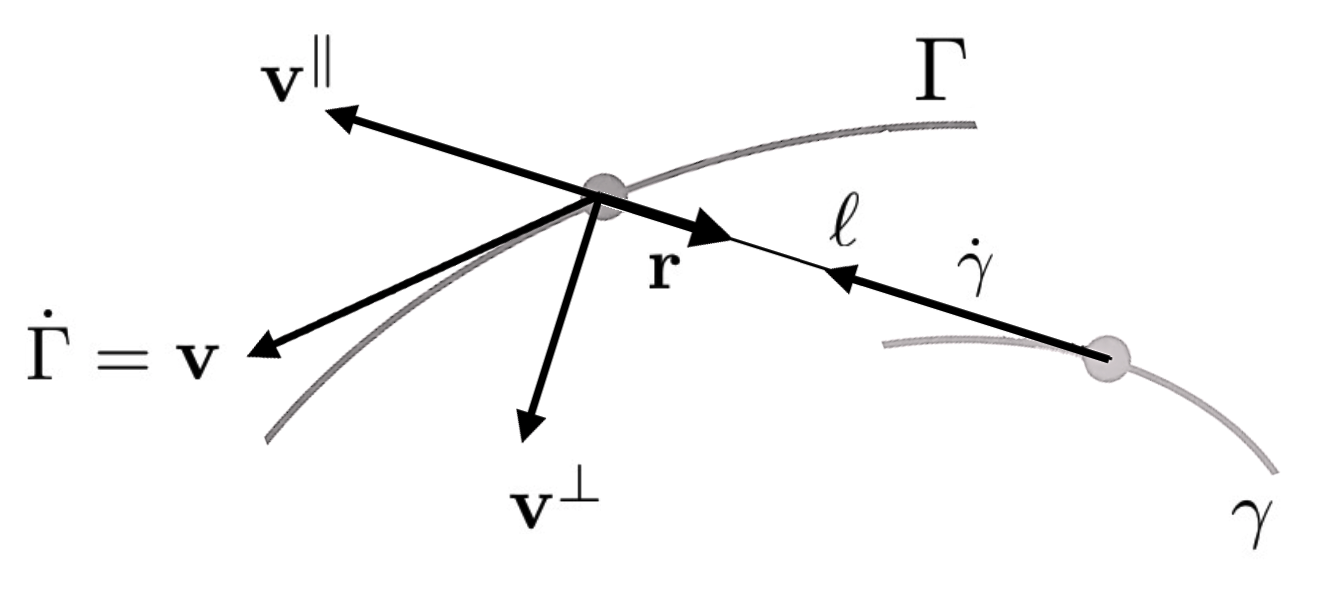}}
\caption{The proof of Proposition \ref{prop:bode}}\label{fig:bepf}
\end{figure}

\begin{rmrk} 
 Equation \eqref{eq:bode} was derived above for $\r(t)\in S^{n-1}$ and indeed it leaves invariant the condition $\|\r\|=1$, as can be easily checked. But it makes sense also for arbitrary $\r(t)\in \R^n$, for which it has also an interesting mechanical interpretation, at least in the $\|\r\|<1$ case  (see Section \ref{sec:rolling} below).
\end{rmrk}

\begin{rmrk} 
 Even if  $\Gamma$ is a  regularly immersed curve, i.e., $\dot\Gamma$ does not vanish, $\dot \gamma$ may vanish. 
From equation \beq, we see that $\dot\gamma=\v+\ell\dot\r$ vanishes  precisely when $\v\cdot \r=0,$ that is, when the bicycle is perpendicular to the front wheel track $\Gamma$. 

In the planar case,  the resulting singularities of $\gamma$ are generically semi-cubical cusps (see \cite{LT}, Section 2, for more information). 

The conceptual explanation of the singularities is as follows.\footnote{This explanation can be safely skipped at first reading.} 
%\note{added a footnote to satisfy a complaint of referee 2, that this remark is a 'tons of bricks'  for the non-expert. -G.}
The configuration space of oriented segments of length $\ell$ in $\R^n$ is the spherization of the tangent bundle $ST\R^n$, and the non-holonomic ``no-skid" constraint  defines a completely non-integrable $n$-dimensional distribution ${\DS}$ therein. The motion of the bicycle is a smooth curve in $ST\R^n$  tangent to the distribution $\DS$ (i.e., a horizontal curve relative to the distribution). 

\begin{figure}[h!]
\centerline{\includegraphics[width=0.3\textwidth]{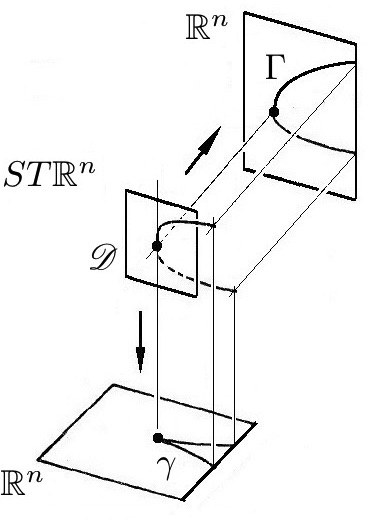}}
\caption{The two projections $ST\R^n\to\R^n$}\label{fig:cusp}
\end{figure}
The two projections $ST\R^n \to \R^n$, to the front and  rear ends of the segment, yield the front and rear bicycle tracks, see Figure \ref{fig:cusp}. The former projection is transverse to $\DS$, therefore the front track is a smooth curve, but the kernel of the latter projection is contained in $\DS$, and hence the rear track may have singularities; this happens when the horizontal curve is tangent to this kernel.  
\end{rmrk}  
\subsection{The bicycle monodromy}
Given a parameterized curve $\Gamma(t)$ in $\R^n$, consider the family of unit spheres centered at points of $\Gamma$, and identify these spheres with each other by parallel translation.\footnote{Such an identification is assumed throughout the paper.} Fix  a point $\Gamma(t_0)$ on the curve $\Gamma$. Then, according to Proposition \ref{prop:bode}, conditions \eqref{eq:rig} and \eqref{eq:ns} define, for each $t$ (for which $\Gamma(t)$ is defined) and $\ell>0$, a  diffeomorphism 
$$
M_\ell^t:S^{n-1}\to S^{n-1},
$$
called {\em the bicycle monodromy}, that maps 
$\r_0$ to $\r(t)$,  where $\r(t)$ is the solution to equation \bode\ satisfying the initial condition $\r(t_0)=\r_0$. In other words, $
M_\ell^t$ is the flow of the differential equation (\ref{eq:bode}).

%\note{added the parenthesis, to answer a comment of the referee 2. -G.}Note that  for a {closed} curve $\Gamma$ the  conjugacy class of the $\ell$-monodromy  (within the diffeomorphism group of $S^{n-1}$) does not depend on the choice of the initial point. 

\begin{example}
%\note{added 'where $\theta=\theta(t)$', after a comment of referee 2. -G.}
 Let $\Gamma$ be the $x$-axis in $\R^2$, parameterized by  $\Gamma(t)=(t,0)$. Substitute $\r=(\cos\theta, \sin\theta)$ in equation \eqref{eq:bode}, where $\theta=\theta(t)$,  and obtain $\ell\dot\theta=\sin\theta$. Another substitution $p=\tan(\theta/2)$  linearizes this equation, yielding  $\ell \dot p= p,$
with solution $p(t)=p_0e^{ t/\ell}.$ 
The resulting  rear track $\gamma$ is the classical {\em tractrix}, and we can use the solution $p(t)$ to give it  an explicit  parametrization (see, e.g.,  \cite{F} for details).
\begin{figure}[h!]
\centerline{\includegraphics[width=.8\textwidth]{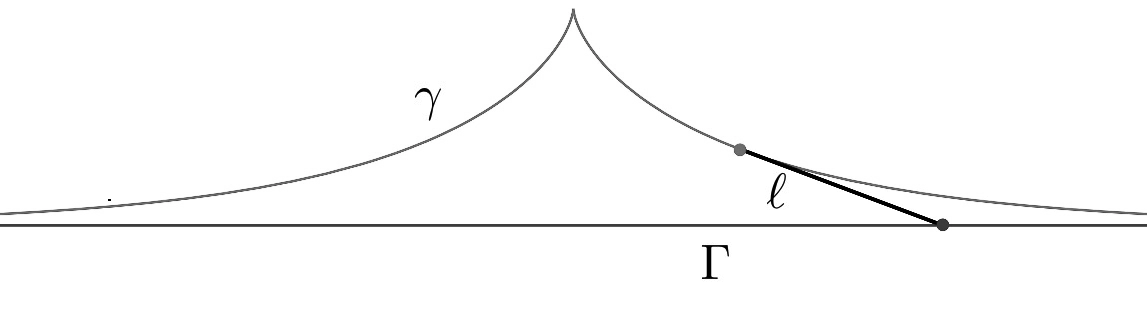}}
\caption{The tractrix}
\end{figure}
\end{example}

\begin{example}
 Let $\Gamma$ be the unit circle in $\R^2$, parameterized by  $\Gamma(t)=(\cos t,\sin t)$. As in the previous example, substitute $\r=(\cos\theta, \sin\theta)$ in equation \eqref{eq:bode},  giving
$\ell\dot\theta=-\cos(\theta-t).$ Changing to $\phi:=\theta-t$ gives 
$\dot\phi=-1-(\cos\phi)/\ell.
$ Changing again to  $p:= \tan(\phi/ 2),$ gives 
$$ \dot p=-{1\over 2\ell}\left[p^2(\ell-1)+\ell+1\right].$$
 This is a constant coefficient Riccati equation that can be solved explicitly in elementary functions (see Section \ref{sec:mc}  below for details). 

\begin{figure}[h!]
 \centerline{\includegraphics[width=.4\textwidth]{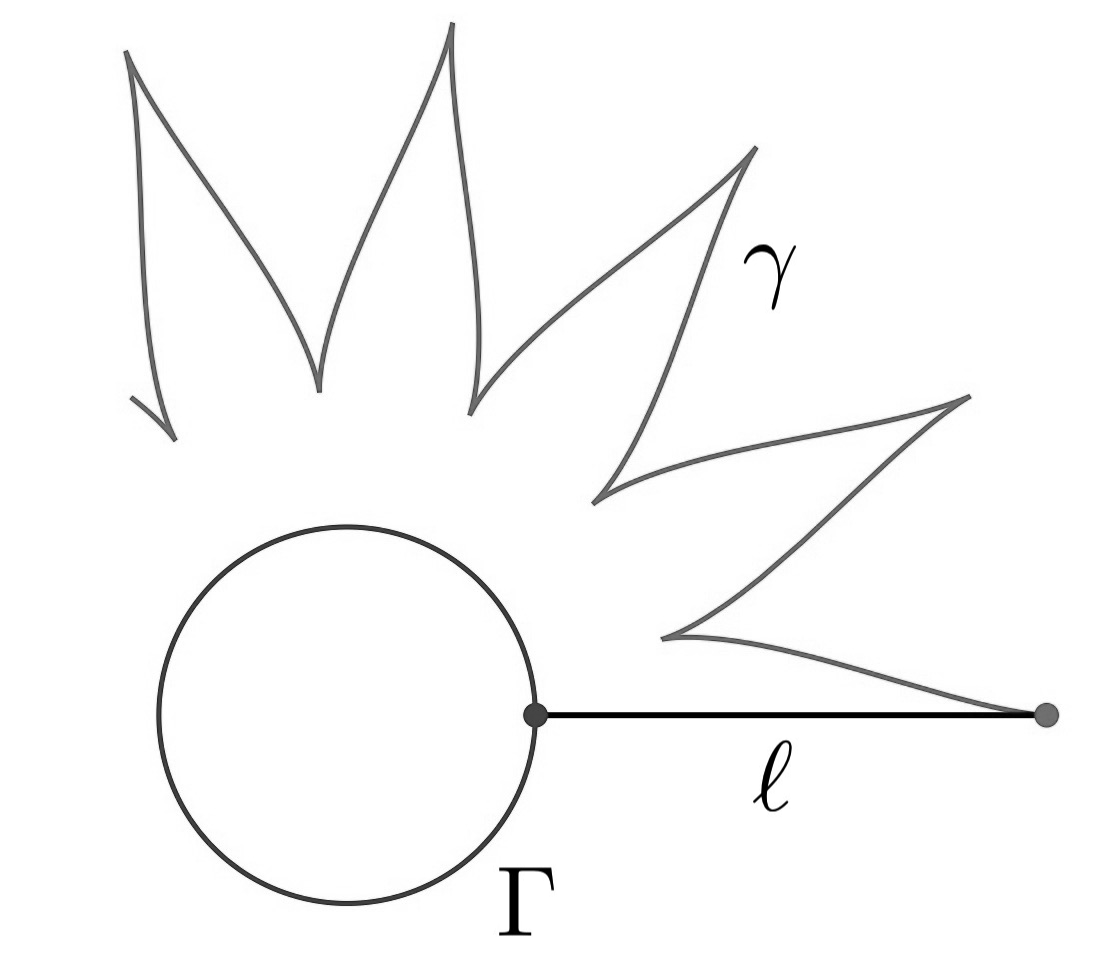}}
 \caption{The circular tractrix}
 \end{figure}
 \end{example}

\subsection{Bicycling in $\R^2$}
There are a number of reformulations of  equation \eqref{eq:bode} for $n=2$ found in the literature \cite{Finn,F,FLT,LT,T1}. We  collect them   in this subsection. 

First, we use  an angle coordinate $\theta$ on $S^1$, i.e., substitute $\r=(\cos\theta, \sin\theta)$ in   equation \eqref{eq:bode}, obtaining, 
\be\label{eq:trigbike}
\ell\dot\theta=v_1\sin\theta-v_2\cos\theta, \qquad \dot\G=(v_1, v_2) .
\ee
Now the 
  {\em projective coordinate} $p=\tan(\theta/2)$, i.e., the slope of a vector with  the argument $ \theta /2 $, satisfies the
   {\em Riccati equation}
\be\label{eq:ric1}\dot p={1\over 2\ell}\left(-v_2+2v_1p+v_2p^2\right), \qquad \dot\G=(v_1, v_2).
\ee

A  consequence of equation \eqref{eq:ric1} is the following theorem of Foote \cite{F}.

 \begin{theorem}\label{thm:foote2d}
The flow of equation \eqref{eq:trigbike}  is the projection  to $S^1$ of the flow of the linear system 
 \be\label{eq:2d}
{\dot x\choose \dot y}=-{1\over 2\ell}\left(\begin{array}{lr}v_1&v_2\\ v_2&-v_1\end{array}\right){ x\choose  y}
\ee
 via the double covering map (using complex notation)  $z=x+iy\mapsto \r= z^2/|z|^2$
 or, more explicitly, 
 $$(x, y)\mapsto\r= \left({x^2-y^2\over x^2+y^2}, {2xy\over x^2+y^2}\right).$$
Thus the bicycle monodromy for $n=2$ is given by elements of the M\"obius group $\PSLt$  of fractional linear transformations  $p\mapsto (ap+b)/(cp+d)$. 
\end{theorem}
 \begin{proof}
  It is well-known that the flow of a Riccati equation consists of  M\"obius transformations (see, e.g., \cite{I}, p.~24). Let us review the argument. Consider the linear  system  
\be\label{eq:lin}
{\dot x\choose \dot y}=\left(\begin{array}{lr}a & b\\ c& -a\end{array}\right){x\choose y}.
\ee
One can check easily that a solution $(x(t), y(t))$ of this system projects to a solution $p(t)=y(t)/x(t)$ of the equation
\be\label{eq:ricgen}
\dot p=c-2ap-bp^2.
\ee
Thus the flow of the Riccati equation \eqref{eq:ricgen} is the projectivization of the flow  of the  linear system \eqref{eq:lin}.
Applying this procedure to equation \eqref{eq:2d}, we obtain equation \eqref{eq:ric1}, and thus   (\ref{eq:trigbike}). 
\end{proof}

The next reformulation of equation \eqref{eq:bode} is obtained by switching to a {\em moving frame} along $\Gamma$  (the Frenet-Serret frame). 
To this end, assume first that $\Gamma$ is parameterized by arclength, so that $\v=\dot\Gamma$ is a unit tangent vector along $\Gamma$. Complete $\v$ to  a positively oriented  orthonormal frame  $\{\v, \nn\}$ along $\Gamma$. Then $\dot\v=\kappa\nn$, where $\kappa$ is the curvature function along  $\Gamma$. Now we use an angle coordinate $\Theta$ for $\r$ in the moving frame $\{\v, \nn\}$, i.e.,  let
$\r=e^{i\Theta}\v= (\cos\Theta)\v + (\sin\Theta)\nn.$ 
(Note: the angle $\Theta$  is $\pi$ minus the ``steering angle" $\alpha$ of \cite{LT}.)

\begin{figure}[h!]
\centerline{\includegraphics[width=.36\textwidth]{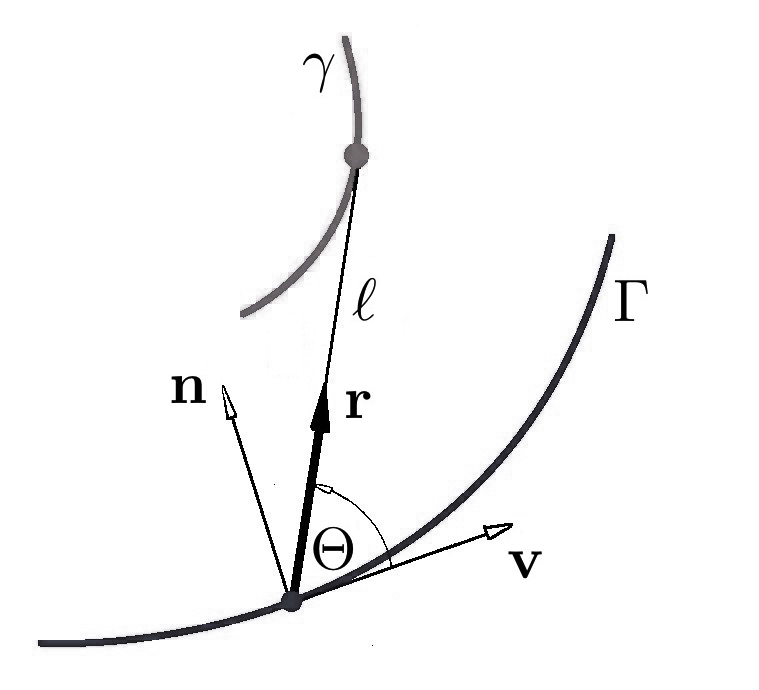}}
\caption{The Frenet-Serret frame along $\Gamma$}
\end{figure}

\begin{prop} $\r= e^{i\Theta}\v$ satisfies equation \eqref{eq:bode} for $n=2$ if and only if $\Theta(t)$ satisfies
$$\dot\Theta={\sin\Theta\over \ell}-\kappa. 
$$
Using the projective coordinate $P=\tan(\Theta/2)$, the last equation is equivalent to 
\be \label{eq:ric2}\dot P={P\over \ell}-{\kappa\over 2}(1+P^2),
\ee
which   is the projectivization $P=Y/X$ of the linear  system 
$$
{\dot X\choose \dot Y}={1\over 2}\left(\begin{matrix}-1/\ell&\kappa \\ -\kappa&1/\ell \end{matrix}\right){ X\choose  Y}.
$$
\end{prop}

\mn The proof is a direct calculation, and we omit it. 

\subsection{Bicycling in $\R^3$} \label{B3D}

Similar to the $n=2$ case, equation \eqref{eq:bode}  for $n=3$ can be reformulated in a variety of ways. To begin with, we rewrite  equation \eqref{eq:bode}  for $n=3$  using the vector product in $\R^3$.
%\note{I erased here the `equivalent' equation $(\v + \ell \dot\r) \times \r = 0$, which is in fact not equivalent, as  referee 2 pointed out. Actually, we never used it. -G.}
\begin{lemma}Equation \beq, for $n=3$, is equivalent to 
\be\label{eq:bode3d} 
\dot\r={1\over \ell}(\v\times \r)\times\r, \quad \v=\dot\Gamma.
\ee
\end{lemma}
We omit the simple verification. 

\mn 

Next we rewrite  equation \eqref{eq:bode3d}   as a {\em complex}  Riccati equation, i.e., as the projectivization of a 2-dimensional complex linear system. 
 \begin{theorem}\label{thm:foote3d}
The flow of equation \eqref{eq:bode3d}  is the projection  to $S^2$ of the flow of the complex linear system 
\be\label{eq:beq1}{\dot z_1\choose \dot z_2}=-{1\over 2\ell}
\left(
\begin{array}{cc}
 v_1 & v_2-i v_3 \\
v_2+i v_3 & -v_1
\end{array}\right)
{z_1\choose z_2}, \quad \dot\Gamma=(v_1 ,v_2 ,v_3),
\ee
 via the complex  Hopf fibration  $\C^2\setminus  0 \to S^2$, 
$${z_1\choose z_2}\mapsto \r=\left({|z_1|^2-|z_2|^2\over |z_1|^2+|z_2|^2}, {2\bar z_1 z_2\over  |z_1|^2+|z_2|^2}\right)\in\R\oplus\C=\R^3.$$
Using the complex coordinate $z=z_2/z_1=(r_2+ir_3)/(1+r_1)$ on $S^2\simeq\CP^1$, the linear system \eqref{eq:beq1} 
projects to the complex Riccati equation 
\be\label{eq:cric}\dot z={1\over 2\ell}\left(-q+2v_1z+\bar q z^2\right), \quad \dot \Gamma=\v=(v_1, v_2, v_3), \quad q=v_2+iv_3.\ee
\end{theorem}

It follows that the bicycle monodromy in $\R^3$   is given by elements of the complex M\"obius group $\PSLtc$.

The proof is by direct calculation which we omit.  In the next subsection we give a more conceptual (group theoretic) explanation  of  Theorems \ref{thm:foote2d} and  \ref{thm:foote3d}.

\begin{rmrk} 
Note that the Riccati equation \eqref{eq:cric} reduces to equation \eqref{eq:ric1} for $q=v_2$, that is, for a planar curve $\Gamma$ with  $v_3=0$. 
\end{rmrk}
 
We now derive a ``moving-frame" version of equation \eqref{eq:cric}.  
Assume $\Gamma$ is parameterized by arc length, so that  
$\v=\dot\Gamma$ is a unit vector, and  complete $\v$ to the Frenet-Serret  frame $(\v, \nn, \b)$ along 
$\Gamma$, satisfying the equations
 $$\dot\v=\kappa\nn, \quad \dot \nn=-\kappa \v+\tau \b,\quad \dot\b=-\tau \nn,$$
where  $\kappa, \tau$ are  the {curvature} and {torsion} of  $\Gamma$. 

\begin{prop} \label{comp3DR}
Let  $\r=R_1\v+R_2\nn+R_3\b$ be a unit vector field along an arc length parameterized  curve $\Gamma$ in $\R^3$.  Then $\r(t)$ satisfies equation \eqref{eq:bode3d} if and only if 
$\RR=(R_1, R_2, R_3)$ satisfies
\be\label{eq:R}\dot\RR=\left[{1\over \ell}\EE_1\times \RR - \BOM\right]\times \RR,
\ee
where $\BOM=\tau \EE_1+\kappa\EE_3$ (the Darboux vector of $\Gamma$ in the Frenet frame) and 
$$ \EE_1=\left(\begin{matrix}1\\ 0\\ 0\end{matrix}\right),\quad \EE_3=\left(\begin{matrix}0\\ 0\\ 1\end{matrix}\right).$$
Using the complex coordinate $Z=(R_2+iR_3)/(1+R_1)$ on the $\RR$-sphere
(stereographic projection from $-\EE_1$ onto  the $(R_2, R_3)$-plane),
we obtain the complex Riccati equation
\be\label{eq:ric3}
\dot Z= \left({1\over \ell}-i\tau\right)Z-{\kappa\over 2}(1+Z^2),
\ee
the projectivization $Z=Z_2/Z_1$ of the linear  system 
\be\label{eq:sys3}
{\dot Z_1\choose \dot Z_2}=
{1\over 2}\left(\begin{matrix}-1/\ell+i\tau&\kappa \\ -\kappa&1/\ell -i\tau\end{matrix}\right){ Z_1\choose Z_2}.
\ee
\end{prop}
\n The proof is again a direct calculation that we omit. 

Note that the  bracketed term in equation   (\ref{eq:R})  is  the angular velocity of the bike expressed in the Frenet frame.  
Note also that equation \eqref{eq:ric3} reduces to equation \eqref{eq:ric2} for a planar curve ($\tau=0$). 

\subsection{Reformulation for general $n$ using the M\"obius group}\label{sec:reform}
 
 In this section we present another way to interpret the bicycle flow   (\ref{eq:bode}).  To illustrate the  idea for  $ n = 2 $ (the higher dimensional case works almost verbatim), the circle $ \| {\bf r} \| = 1 $  is embedded in Minkowski's 3-space $\R^3 $ as shown in Figure \ref{fig:cone}; namely, as the intersection of the cone $ x_1 ^2 + x_2 ^2 - x_3 ^2 = 0 $ and the horizontal plane $ x_3=1 $. Each generating ray of the cone is uniquely determined by a unit  vector $ {\bf r} $, as shown in Figure \ref{fig:cone}. We then consider  {\it  linear } flows in $\R^3$ preserving the Lorentz quadratic form  $ x_1 ^2 + x_2 ^2 - x_3 ^2,$ so that   the cone is invariant under any such Lorentz--orthogonal flow.  
 %
%   \note{added `at time $t$' and  `where $\v=\dot\Gamma(t)$', to clarify a confusion of referee 2 here. Also,  modified a bit figure \ref{fig:cone} and its caption, after a remark of  referee 2.-G.}
 %  
    We show that the bicycle flow on ${\bf r}$ at time $t$ corresponds to a particular linear Lorentz--orthogonal flow, namely, to a flow with two eigendirections lying in the vertical plane through the origin containing $\pm\v$, where $\v=\dot\Gamma(t)$.

\begin{figure}[h!]
\centerline{\includegraphics[width=0.4\textwidth]{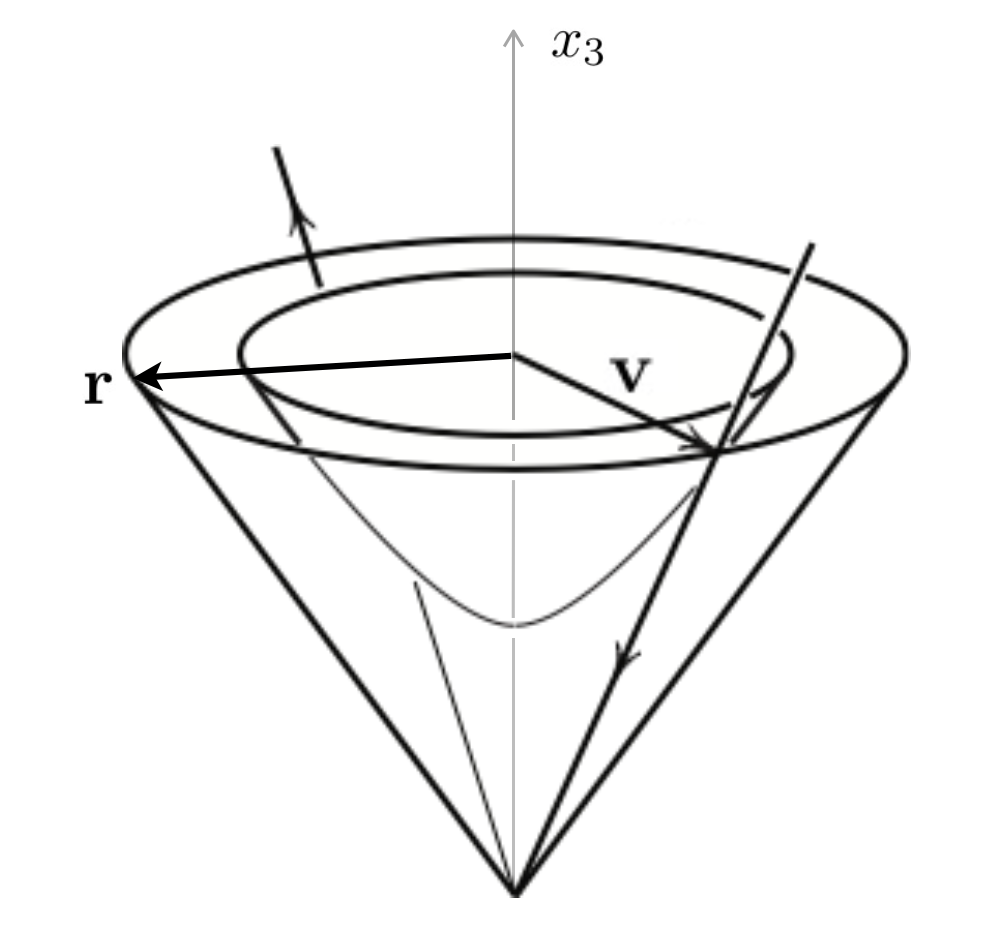}} 
\caption{The null cone in $\R^{2,1}$. The arrows along the two cone generators show the direction of the flow along the eigendirections of $A$ in Lemma \ref{lemma:A} below.}\label{fig:cone}
\end{figure}

   The same  construction shows how the bicycle flow extends from the circle $\|\r\| =1 $ to a flow of the disk $\|\r\|< 1 $ by hyperbolic isometries.

 We now proceed with the formal discussion for general $n$.

 \mn

 Let $\R^{n,1}$ be ${n+1}$-dimensional space equipped with the  quadratic form 
 $$\<\xb,\xb\>:=(x_1)^2+\ldots+(x_n)^2-(x_{n+1})^2, \quad \xb=(x_1,\ldots, x_{n+1}).$$ 
 Let $\SO^+_{n,1}\subset\GL(\R^{n,1})$ be the orientation and time-orientation preserving linear isometries of $\R^{n,1}$ (the identity component of the Lorentz-orthogonal  group $\mathrm{O}_{n,1}$). Its Lie algebra $\so_{n,1}$ consists of  $(n+1)\times (n+1)$ matrices, written in block form as 
$$
\left(\begin{matrix}
B&\v\\
\v^t&0
\end{matrix}\right),\quad \v=(v_1, \ldots , v_n)^t\in\R^n, \quad B\in \so_n \,\,(\hbox{i.e., }  B^t=-B) .$$
 Let 
 $$\R^{n,1}_+=\{\x\in\R^{n,1}\st x_{n+1}>0\}$$  and 
\be\label{proj}\pi:\R^{n,1}_+\to \R^n, \quad 
\xb=(x, x_{n+1})\mapsto {x\over x_{n+1}}, \quad x=(x_1, \ldots, x_n).
\ee
For each $\ell>0$, let 
\be\label{hyp}H^n_\ell=\{\x\in \R^{n,1}_+\st \<\x, \x\>=-\ell^2\}\ee
and 
$$\CC=\{\x\in \R^{n,1}_+\st \<\x, \x\>=0\}.$$
Equip $\R^{n,1}$ with the  flat pseudo-Riemannian  metric induced by $\<\cdot, \cdot \>$, 
$$g=\<\d\x, \d\x\>=(\d x_1)^2+\ldots+(\d x_n)^2-(\d x_{n+1})^2.$$

Using this notation, we collect in the next proposition  some standard facts about the geometry of the  $\SO^+_{n,1}$-action on $\R^{n,1}$, see, e.g., \cite{BP}. 
\begin{prop}\label{prop:action} For all $n\geq 2$,  
\begin{enumerate}
\item $H^n_\ell$ and $\CC$  are  the $\SO^+_{n,1}$-orbits  of $(\ell,0, \ldots, 0)^t$ and 
$(1,1,0, \ldots, 0)^t$, respectively.

\item The flat pseudo-Riemannian metric  $g=\<\d\x, \d\x\>$  on $\R^{n,1}$ restricts on $H^n_\ell$ to a Riemannian metric of constant negative sectional curvature $-1/\ell^2$, on which  $\SO^+_{n,1}$ acts transitively as its group of orientation preserving isometries. ($H^n_1$ is  the ``hyperboloid model" of the hyperbolic $n$-space.)

\item For each $\ell>0$, the restriction of $\pi$ to  $H^n_\ell\subset \R^{n,1}_+$ is a diffeomorphism onto the unit ball $B^n=\{\r\in \R^n\st \|\r\|<1\}$. The induced metric on $B^n$ is
\be \label{met}\d s^2_\ell={\ell^2\over 1-\|\r\|^2}\left({(\r\cdot \d\r)^2\over 1-\|\r\|^2}+\|\d\r\|^2\right).
\ee
 ($B^n$, equipped with this metric  for  $\ell=1$, is   the Klein-Belrami or projective model of hyperbolic $n$-space.) 
\item The restriction of $g$ to $\CC $ is degenerate (for all $\x\in\CC $ the line $\R\x\subset T_\x\CC $ is orthogonal to   $T_\x \CC $), descending to a conformal Riemannian metric on its spherization  $S^{n-1}=\CC /\R^+$, isomorphic to the standard conformal structure on $S^{n-1}$ (see next item).

 \item \label{item:conf} The image of $\CC $ under $\pi$ is $S^{n-1}=\partial B^n.$ The metric $g$, restricted to $\CC $,  descends via $\pi$ to the standard (``round") conformal metric on $S^{n-1}$. The action of $\SO^+_{n,1}$ on $\CC $ descends to   $S^{n-1}$, preserving the conformal structure. The group $\SO^+_{n,1}$ acting in the described way  is  called the M\"obius group $\Mob(S^{n-1})$. 
 \item $\Mob(S^n)$, for $n\geq 2$,  is the full group of orientation preserving conformal transformations of $S^n$. $\Mob(S^1)$  is the projective group $\PSLt$. 

\end{enumerate}
 \end{prop}

 The following lemma is  borrowed from \cite{LT}. We reproduce its proof here for the convenience of the reader. 
 
\begin{lemma}\label{lemma:A}
%
%\note{Modified a bit the formulation of the Lemma and its proof to answer a comment of referee 2.  -G.}
For each $\v\in\R^n$, consider the linear vector field $v$ on $\R^{n,1}$, $v(\x)=A\x$, where 
$$A=-\left(\begin{matrix}
0_n&\v\\
\v^t&0
\end{matrix}\right)\in \so_{n,1}.$$
Then,  under the projection $\pi:\R^{n,1}_+\to \R^n$ of formula \eqref{proj},  $v$ maps to the vector field on $\R^n$ defined  by the right hand side of equation \eqref{eq:bode}.

\end{lemma}

\begin{proof} 
%
%\note{Referee 2 says this proof can be done  ``much faster'' by a symmetry argument. I doubt it. -G.

%I am not sure either. - S.}

Let $\x=(x,x_{n+1})\in\R^{n,1}_+$ and $\delta\x=(\delta x,\delta x_{n+1} )$ a tangent vector at $\x,$ where $x=(x_1, \ldots, x_n)$.  Then, by formula \eqref{proj}, the derivative of $\pi$ at  $\x$ is     $\delta \x\mapsto (x_{n+1}\delta x-x\delta x_{n+1} )/(x_{n+1})^2.$  Next, let $\delta\x=v(\x)=A\x.$ Then $\delta\x=(\delta x,\delta x_{n+1} )$, where  $ \delta x=-x_{n+1}\v$ and $\delta x_{n+1}=-x\cdot \v.$  It folllows that  the  image of $v(\x)$ under $\d\pi$ is 
$$\frac{x_{n+1}(- x_{n+1}\v)-x (-x\cdot \v)}{x_{n+1}^2}=-\v+(\v\cdot\r)\r.$$
\end{proof}

As a consequence of the previous proposition and  lemma, we obtain the following  theorem, proved  for $n=2$ in \cite{F}, and for general $n$ in \cite{LT}.

\begin{theorem}\label{thmf} 
%
%\note{Referee 2 sketches an `easier proof' of this theorem. I dont really understand his proof and doubt that it is easier. -G.
%I think, this referee's remark is a continuation of the previous one (reduction to 2-dim case). I'd leave things as they are here. -S.}
The flow of equation \eqref{eq:bode} (for all $\r$)  is the projection
via $\pi:\R^{n,1}_+\to \R^n$ (defined in  equation \eqref{proj})
 of the flow of the  linear system in $\R^{n,1}$ with $\so_{n,1}$ coefficient matrix 
\be\label{eq:mono}
\dot\x=-{1\over\ell}\left(\begin{matrix}
0_n&\v\\
\v^t&0
\end{matrix}\right)\x,  \quad \v=\dot\Gamma. 
\ee
It follows that 

\mn (1) The bicycle monodromy $M_\ell^t: S^{n-1}\to S^{n-1}$  is a M\"obius transformation, well-defined for all $t$ for which $\Gamma(t)$ is defined.  

\mn (2) The flow of equation \eqref{eq:bode} preserves the unit open ball $B^n=\{\|\r\|<1\}$, on which it acts by isometries of the hyperbolic metric of  equation \eqref{met}. 

\end{theorem}

\begin{rmrk} \label{rmrk:roll}
In Section \ref{sec:rolling} below we interpret the hyperbolic isometries of item (2) of  Theorem \ref{thmf}  above as ``rolling without slipping and twisting" of $\Hy$ along $\Gamma\subset \R^n$. The flow of  equation \eqref{eq:bode}  also preserves the  complement of the closed unit ball $\{\|\r\|>1\}$ on which it acts  by isometries of a (curved) Lorenzian metric.  We do not pursue here this aspect of the bicycle monodromy, but it would be interesting to find  a  mechanical-geometric interpretation of this flow. 
\end{rmrk}

\begin{rmrk}
The bicycle equation \eqref{eq:bode} can be also reformulated in the language of bundles and connections, 
which some readers might find useful (a similar interpretation for $n=2$ appeared in \cite{F}). This formulation leads  to a straighforward generalization of the bicycling equation on any Riemannian manifold. We sketch here this formulation. 

Consider  the $\so_{n,1}$-valued 1-form on $\R^n$ 
\be \label{eq:con}
\theta={1\over\ell}\left(\begin{matrix}

0&\cdots&0&\d x_1\\
\vdots&\vdots&\vdots&\vdots\\
0&\cdots&0&\d x_n\\
\d x_1&\cdots&\d x_n&0
\end{matrix}\right).
\ee
We view $\theta$ as the 1-form of an $\SO^+_{n,1}$-connection on  the trivial principal bundle 
$\R^n\times\SO_{n,1}^+\to\R^n$. 
For any  space $F$  on which $\SO_{n,1}^+$ acts,  $\theta$ defines the covariant derivative $\mathrm D=\d+\theta$ 
of sections $f:\R^n\to F$  of the associated bundle $\R^n\times F\to \R^n$. Namely, $\mathrm Df=\d f+\theta\cdot f$, and 
a section is parallel if $\mathrm Df=0$. 

Equation \eqref{eq:mono} is then the equation for paralell transport in the vector bundle associated to the standard represetation of $\SO_{n,1}^+$  on $\R^{n,1}$. The projectivization of this representation 
has 3 orbits: the projectivized null cone $S^{n-1}$, its interior $B^n$, and the exterior $\RP^n\setminus\overline{B^n}$. On each of these orbits $\SO^+_{n,1}$ acts as the automorphism group of a different structure: isometries of a hyperbolic metric on $B^n$, M\"obius  transformation of $S^{n-1}$, and isometries of a Lorentzian  metric on $\RP^n\setminus\overline{B^n}$. The parallel transport in the associated bundle $\R^n\times \RP^n\to \R^n$ is given by the bicycle equation \eqref{eq:bode}, where $\r$ is used as an affine coordinate. 

For an arbitrary Riemannian manifold $M$ the $\ell$-bicycling equation defines  an $\SO^+_{n,1}$-connection on  its unit tangent sphere bundle.  In general, this connection is non-flat, unless $M$ has a metric of constant curvature $-1/\ell^2$, i.e., is hyperbolic, in which case the connection defines an interesting foliation of the unit tangent bundle of $M$. See Example 3.6.17 on p.~165 of \cite{LTh}.  
\end{rmrk}
 
\subsection{The special isomorphisms $\so_{2,1}\simeq\slt $,  $\so_{3,1}\simeq \sl_2(\C)$}
For $n=2,3$ there are  ``special isomorphisms" which enable us to  replace equation \eqref{eq:mono} with a more compact   linear system  with $2\times 2$ real or complex  matrices instead of $3\times 3$ or $4\times 4$  real  matrices (respectively). 

\subsubsection{$n=2$}  
Let $\slt$ be the Lie algebra of $\SLt$, i.e., the set of traceless 
real $ 2 \times 2 $ matrices $A$.  We equip $\slt$ with the quadratic form $-\det(A)$; this   form has signature $(2,1)$,  suggesting a relation 
with $\R^{2,1} $; indeed, one has the following.

\begin{prop} The map $\R^{2,1}\to\slt$, $$\x=(x_1,x_2,x_3)\mapsto A=\left(\begin{matrix}
-x_2&x_1+x_3\\
x_1-x_3&x_2
\end{matrix}\right),$$
is an isometry, mapping the quadratic form $\<\x, \x\>=(x_1)^2+(x_2)^2-(x_3)^2$ to $-\det(A)$.    The  conjugation action of  $\SLt$  on $\slt$, $A\mapsto gAg^{-1}$, preserves the quadratic form $-\det(A)$. The resulting homomorphism  $\SLt\to \SO^+_{2,1}$ is surjective  with kernel  $\{I, -I\}$. The corresponding isomorphism of Lie algebras $\slt\simeq \so_{2,1}$ is 
\be\label{eq:2d_iso} \left(\begin{matrix}
v_1 &v_2\\
v_3&-v_1
\end{matrix}\right)\mapsto 
\left(
\begin{array}{ccc}
 0 & v_3-v_2 & 2v_1 \\
  v_2-v_3 & 0 & v_2+v_3 \\
2 v_1 & v_2+v_3 & 0 \\
\end{array}
\right).\ee

\end{prop}

\begin{proof} A direct calculation which we omit.
 \end{proof}

In particular, applying the inverse of the isomorphism \eqref{eq:2d_iso}  to the coefficient matrix  of the system \eqref{eq:mono} for $n=2$, 
$$-{1\over \ell}\left(
\begin{array}{ccc}
 0 & 0  & v_1 \\
0  & 0 & v_2  \\
v_1  &v_2   & 0 \\
\end{array}
\right)\mapsto 
 -{1\over 2\ell}\left(\begin{matrix}
v_1 &v_2\\
v_2&-v_1
\end{matrix}\right),$$
we obtain the system \eqref{eq:2d} (which we obtained previously by different means): 

$${\dot x\choose \dot y}=-{1\over 2\ell}\left(\begin{matrix}v_1&v_2\\v_2&-v_1\end{matrix}\right){ x\choose  y}. 
$$

 \newcommand{\cH}{\mathcal{H}}
 
\subsubsection{$n=3$ }

Let $\cH$ be the space of  $2\times 2$ complex Hermitian  matrices, $A=\bar A^t$, equipped with the (real) quadratic form $ -\det(A)$  of signature $(3,1)$. 
\begin{prop} 
The map $\R^{3,1}\to\cH$, 
$$\x=(x_1, x_2, x_3,x_4)\mapsto A=\left(\begin{array}{cc}
-x_1+x_4&-x_2+ix_3\\
-x_2-ix_3 &x_1+x_4
\end{array}\right), 
$$
is an isometry, 
mapping the quadratic form $\<\x,\x\>=(x_1)^2+(x_2)^2+(x_3)^2-(x_4)^2$ to $-\det(A)$.  The linear action of  $\SL_2(\C)$ on $\cH$, 
$A\mapsto gA\bar g^t$, preserves the quadratic form   $-\det(A)$. The resulting homomorphism  $\SL_2(\C)\to \SO^+_{3,1}$ is surjective  with kernel  $\{I,-I \}$.
The associated isomorphism of Lie algebras 
$\sl_2(\C)\simeq \so_{3,1}$ 
is
\be\label{eq:iso3} \left(\begin{array}{lr}
a&b\\
c&-a
\end{array}\right)\mapsto\left(
\begin{array}{cccc}
0 		&b_1-c_1		&-b_2-c_2		&-2a_1		\\
- b_1+c_1	&0 			& 2 a_2  		& - b_1- c_1	\\
b_2+ c_2&-2a_2 	 	& 0 			& b_2-c_2  \\
-2a_1&- b_1-c_1		&b_2- c_2	 	& 0
\end{array}
\right),\ee
where $ a=a_1+ia_2, b=b_1+ib_2, c=c_1+ic_2.$
\end{prop}

\begin{proof} Another computation that we omit. 
\end{proof}

Applying the inverse of the isomorphism \eqref{eq:iso3}  to the coefficient matrix  of the system \eqref{eq:mono} for $n=3$, 
$$
-{1\over \ell}\left(
\begin{array}{cccc}
 0 & 0 & 0&v_1 \\
 0& 0 &  0 &v_2 \\
 0& 0 &  0 &v_3 \\
v_1 &v_2 & v_3 &0 
\end{array}
\right)
\mapsto -{1\over 2 \ell}
\left(
\begin{array}{cc}
 v_1 & v_2-i v_3 \\
v_2+i v_3 & -v_1
\end{array}\right),$$
we obtain the 2-dimensional complex linear system 
$${\dot z_1\choose \dot z_2}=-{1\over 2\ell}
\left(
\begin{array}{cc}
 v_1 & v_2-i v_3 \\
v_2+i v_3 & -v_1
\end{array}\right)
{z_1\choose z_2}, \quad \dot\Gamma=(v_1 ,v_2 ,v_3),
$$
which we also obtained previously in equation \eqref{eq:beq1}  by different means. 

\begin{rmrk}
Although not used  in this article, two other  special isomorphisms are  $\so_{5,1}\simeq \sl_2(\H)$ and $\so_{9,1}\simeq \sl_2(\mathbb{O})$ (see, e.g., \cite{V}). Hence the bicycle equation in $\R^5$ and $\R^9$ can also be ``linearized" by   2-dimensional quaterionic and octonionic linear systems (respectively), with the corresponding Riccati equations. The quaternionic system is 
$${\dot h_1\choose \dot h_2}=-{1\over 2\ell}
\left(
\begin{array}{cc}
 v_1 & \bar q \\
 q & -v_1
\end{array}\right)
{h_1\choose h_2},$$
where
$$\dot\Gamma=(v_1 ,\ldots,v_5), \quad q=v_2+iv_3+jv_4+kv_5\in\H,$$
with the quaternionic  
Riccati equation for $h=h_2h_1^{-1}$ 
$$\dot h={1\over 2\ell}(-q+ 2v_1h+h\bar q h).$$ 
\end{rmrk}

\subsection{A Berry phase formula for  the bicycle monodromy }
\label{sec:berry}

Let $\Gamma$ be a parameterized front curve in  $\R^n$ and $M_\ell:S^{n-1}\to S^{n-1}$ the  associated $\ell$-bicycle monodromy  between two points $\Gamma(t_0),\Gamma(t_1)$ on $\Gamma$, with $t_0<t_1$. 

\begin{theorem}\label{thm:monodromy} For every $n\geq 2$, $\ell>0$ and $\r_0\in S^{n-1}$, the derivative 
$M'_\ell(\r_0):T_{\r_0}S^{n-1}\to T_{\r_1}S^{n-1}$ is given by 
$$M'_\ell(\r_0)=e^{-L_\gamma/\ell}P, 
$$
where 
\begin{itemize}
\item $\r_1=\r(t_1)$ and $\r(t)$ is the solution to equation \bode\ with $\r(t_0)=\r_0$,  
\item $\gamma(t)=\Gamma(t)+\ell \r(t)$ is the corresponding rear track, 
\item $L_\gamma=-\int_{t_0}^{t_1}\r\cdot\v\, \d t$, $\v=\dot\G$,  is the (signed) length of $\gamma$, 
\item $P\in \mathrm{Iso}(T_{\r_0}S^{n-1}, T_{\r_1}S^{n-1})$ is the parallel transport  in $TS^{n-1}$ (with respect to the Levi-Civita connection) along the curve  $\r(t)$, $t_0\leq t\leq t_1$.
\end{itemize}
\end{theorem}
 
\begin{rmrk}
The sign of the  length element $-\r\cdot\v\,\d t$ of the rear track  is adjusted to coincide with the geometric intuition of forward riding of $\gamma$ being counted as positive length and backward riding as negative. The sign is reversed at a cusp. 
\end{rmrk}

 \begin{proof}[{\it Proof of Theorem \ref{thm:monodromy}}] 
 %
% \note{In response to a comment of referee 2, I shortened the proof a bit by eliminating a lemma about general connections. -G. }
 %
 Note first that  for every $\xib_0\in T_{\r_0}S^{n-1}$, one has $M'_\ell(\r_0)\xib_0=\xib(t_1)$, where $\xib(t)\in T_{\r(t)}S^{n-1}$ is the solution to the linearization of equation \bode\ along $\r(t)$, satisfying $\xib(t_0)=\xib_0$; namely, 
$$
\ell\dot\xib=(\v\cdot\xib)\r+ (\v\cdot\r)\xib, \quad\v=\dot\Gamma.
$$
It follows that $(\v\cdot \r)\xib/\ell$ is the orthogonal projection of $\dot\xib$ on $T_\r S^{n-1}$. But this is precisely the definition of the covariant derivative along a submanifold in $\R^n$. That is, 
\be\label{eq:cov}\nabla_{\dot \r}\xib={\r\cdot\v\over\ell} \xib,
\ee
where  $\nabla$ is the Levi-Civita connection on $S^{n-1}$.

Next, let $u:=(\r\cdot\v)/\ell$, 
$f(t):=\exp\left(\int_{t_0}^t u(s)\d s\right)$
and 
$\hat\xib(t)$ the  parallel transport of $\xib(t_0)$  along $\r(t)$; that is, $\hat\xib(t_0)=\xib(t_0)$, and $\nabla_{\dot \r}\hat\xib=0$.  The theorem then amounts to $\xib=f \hat\xib.$ 
To show this, it is enough to show that both $\xib, f\hat\xib$ have the same value at $t=t_0$ and satisfy the same first order differential equation. By equation \eqref{eq:cov}, $\xib$ satisfies $\nabla_{\dot \r}\xib=u\xib$. Now $f$, by its definition, satisfies $f(t_0)=1$ and $\dot f=uf$, hence  $(f\hat\xib)(t_0)=\xib(t_0)$ and  
$$\nabla_{\dot \r}(f\hat\xib)=\dot f\hat\xib+f\nabla_{\dot \r}\hat\xib=\dot f\hat\xib=u(f\hat\xib).$$ Therefore $\xib=f\hat\xib$.
\end{proof}

\begin{rmrk} 
For $n\geq 3$, Theorem \ref{thm:monodromy} implies   that $M_\ell$ is a conformal transformation (since parallel transport is an isometry), so it gives  an alternative proof of  Theorem \ref{thmf}, item 1. 
\end{rmrk}

The most  interesting case of Theorem \ref{thm:monodromy} is that of a closed curve  in $\R^3$, where $M_\ell\in\PSL_2(\C)$. Generically, $M_\ell$ has two fixed points in $S^2$ and is thus conjugate to the M\"obius transformation $z\mapsto \lambda z,$ whose fixed points are $0,\infty$, with $M'(0)=\lambda,M'(\infty)=1/\lambda.$ The {\em conjugacy class} of $M$ is thus given by the derivatives $\lambda^{\pm 1}$ at the fixed points.

\begin{cor}[Berry phase formula] \label{Berry}
Let $\Gamma$ be a closed  curve in $\R^3$, $\r_0\in S^2$  a fixed point of the $\ell$-bicycle monordomy of $\Gamma$ (with respect to some initial point $\Gamma(t_0)$)  and $\gamma$ the corresponding closed rear track. Then   
\begin{equation} \label{eq:berry}
M'_\ell(\r_0)=e^{-(L_\gamma/\ell)+i\Omega},
\end{equation}
where $L_\gamma$ is the signed length of $\gamma$ and $\Omega$ is the area in $S^2$ enclosed by the spherical curve $\r(t).$
\end{cor}
 \begin{proof}   By the Gauss-Bonnet theorem, parallel transport around  a  closed  curve  in $S^2$  is a rotation by an angle equal to the area  of the spherical region bounded by the curve.   
 \end{proof}
 
\begin{rmrk}

\mn 
\begin{enumerate}[(i)]
\item The last corollary and its proof still hold  when    the spherical curve $\r(t)$ is not simple,  provided  $\Omega$ is defined as the   algebraic (or signed) area of the spherical region bounded by $\r(t)$; see, e.g., \cite{Ar}.
\item  For $n=2$,  in the case of hyperbolic monodromy, with $\gamma$ being one of the two periodic rear tracks, the formula reduces to $M'_\ell(\r_0)=e^{-L_\gamma/\ell}$,  as in  Theorem 3.6 of  \cite{LT}. This formula determines the conjugacy class of the bicycle monodromy when it is a hyperbolic element of $\PSLt$. The elliptic case, on the other hand,  becomes clear only once we embed the bike in $ {\mathbb R}  ^3 $, as explained below in  Section \ref{bird}.  
 \item For a  generic $ \Gamma $ in $ {\mathbb R}  ^3$,  the $M_\ell$--iterates of all points on the sphere, except for the unstable fixed point, approach the stable fixed point. This means that all spatial motions  of the bike, save the unstable periodic one, approach the stable periodic motion.  
The case of planar $ \Gamma $ (in $ {\mathbb R}  ^3  $) is special: $M_\ell$ commutes with reflections in the plane, and in the elliptic case $M$ is conjugate to a rigid rotation of $S^2 $. All the bike motions in $ {\mathbb R}  ^3 $ are then periodic or quasiperiodic with two frequencies. 

\item In the planar elliptic case embedded in $ {\mathbb R}  ^3 $,  the length of each of the two  periodic rear tracks is zero, as follows from   equation (\ref{eq:berry}). 
\item Here is  a heuristic explanation  for the appearance of Berry  phase $\Omega$ in  formula \eqref{eq:berry}.   Figure \ref{fig:heur} shows two infinitesimally close bikes, with rear wheels at $R,R_1$, sharing the same front trajectory $ \Gamma$ at $F$, with $R$ tracing a closed back track  $\gamma$ and $R_1$ a nearby   (not  necessarily closed) back track. Consider the unit    vector $\hat\xib =RR_1/|R R_1| $. The key observation is that {\em  the angular velocity of $\hat \xib$ around the axis $ RF $ is zero.}   

\begin{figure}[h!]
 \centerline{\includegraphics[width=.7\textwidth]{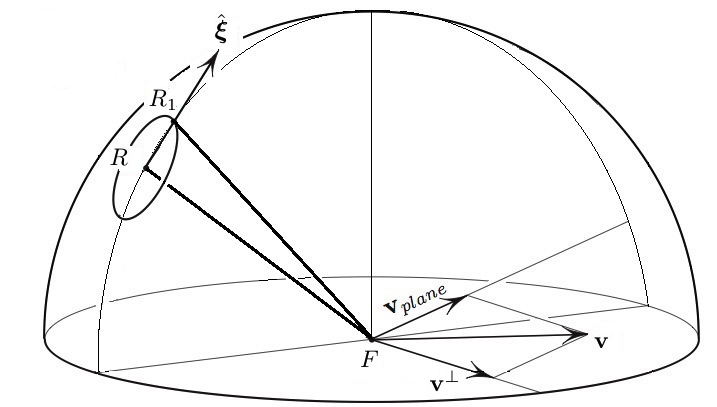}}
\caption{A heuristic explanation  for the appearance of Berry's phase}\label{fig:heur}
\end{figure}

To justify this, let us decompose $\v=\dot F$ as $\v=\v_{plane}+\v^\perp$, where $\v_{plane}$ is the orthogonal projection of $\v$ unto the $FRR_1$--plane and $\v^\perp$ the perpendicular component, as shown in Figure \ref{fig:heur}.  Let us consider separately the effects of $ {\bf v}_{plane} $
and ${\bf v} ^\perp $. 

First, the motion of $R$ and $R_1$ due to $ {\bf v}_{plane} $ occurs in the plane $ RFR_1 $, and thus $\hat\xib\parallel  RR_1 $ 
  does not rotate about any axis in that plane, let alone about $ FR $.  
Second,  the component $ {\bf v} ^\perp $  does not even contribute to the velocities  of  $R$ and $R_1$, and therefore $ {\bf v} ^\perp $ has no 
effect on the motion of $\hat \xib $.  

This zero angular velocity statement  is equivalent to saying that   {\em $\hat \xib $ undergoes parallel transport  on the sphere   centered at $F$ along the curve traced on it by $R$}. It follows, by the Gauss--Bonnet theorem, that the vector $ \hat\xib $ will end up rotated by an angle, equal to  the solid angle $\Omega$ (or spherical area) bounded by  the  closed path traced by $R$, as viewed by an observer moving with $F$. 
\end{enumerate}
 \end{rmrk}

 \subsection{Bicycle as a planimeter in $\R^n$}
 
 As we mentioned in the introduction, the planar bicycle can serve as a planimeter.  
 In this section, we examine the higher dimensional version of this phenomenon.

 Let $\Gamma$ be a closed curve in $\R^n$, the bicycle front track,  of length $L$.
% \note{replaced 'perimeter length' with length; referee 2 didnt like it. Shall we change it throughout the paper? -G.
% Done. S.}
Let the bicycle length be $\ell=1/\eps$; as before, $\v= \dot \Gamma$ and $\r$ is a unit vector along the bicycle segment. The  bicycle equation 
\bode\ with an initial condition is 
\begin{equation}  
   \left\{ \begin{array}{l} 
   \dot\r=\varepsilon ( -\v+(\v\cdot\r)\r) ,  
 \\[3pt] 
   \r(0, \varepsilon ) = \r_0. \end{array} \right.   
 \label{eq:cauchy}
\end{equation}  
 
 The following theorem generalizes the hatchet planimeter formula \eqref{eq:prytzformula}.
 The area bounded by a closed plane curve $\Gamma(t)$ is given by the integral 
 $$
 \frac{1}{2} \int \det(\Gamma,\dot\Gamma)\ \d t.
 $$
%  \note{ Added text as per referee 2 request. -S. 
  
%\mn  Changed $\Gamma'$ to $\dot\Gamma$, and $dt$ to $\d t$. -G. }
For a curve $\Gamma$  in $\R^n$, an analog of the area is the area bivector 
$$
 \frac{1}{2} \int \Gamma \wedge\dot\Gamma\ \d t,
 $$
 which contains the information about the areas bounded by the projections of the curve on all coordinate 2-planes, see Remark \ref{areas} below. This  bivector can be interpreted as a skew--symmetric linear operator.

\newcommand{\A}{{\mathcal A}}

\begin{theorem} \label{3Dplan}

The  bicycle vector $ \r $, i.e., the solution of   (\ref{eq:cauchy}), undergoes a net rigid rotation, up to an $ O( \varepsilon^3 ) $--error; more precisely,  

\begin{equation} 
	\r(L)=\r_0+\varepsilon ^2 {\A} \r_0+ O(\varepsilon^3), 
	\label{eq:3Dplan}
\end{equation} 
where $\A:\R^n\to\R^n$ is the skew--symmetric ``area operator" of $\Gamma$, given by 
\[
	\A \r_0= \int_{0}^{L}( \Gamma \cdot \r_0)\dot\Gamma \,\d t,  \  \  \hbox{for  }  \  \  \r_0\in {\mathbb R}  ^n.
\]
For $n = 3 $,  
\begin{equation} 
	\A \r_0=\int_{0}^{L}(\Gamma \cdot \r_0)\dot \Gamma\,\d t={\widehat \A}\times \r_0,  
	\label{eq:3D}
\end{equation}   
where 
$$ {\widehat \A} = {1\over 2}\int_0^L ( \Gamma\times\dot\Gamma)\ \d t$$ is the area vector of $ \Gamma $. Thus in $ {\mathbb R}  ^3 $, modulo an 
$ O ( \varepsilon ^3 ) $--error, the initial bike direction $ \r_0 $ is rotated around the direction $ {\widehat \A} $ through the angle $ \| {\widehat \A}\| $, equal to the signed area of the projection of $\Gamma$ on a plane perpendicular to $ {\widehat \A}$.
\end{theorem}

\begin{proof}
The  solution $ \r (t, \varepsilon ) $ of the Cauchy problem   (\ref{eq:cauchy})  is analytic in $\varepsilon$ since so is the right--hand side,  and thus can be expanded in a Taylor series in $ \varepsilon$, starting with   
\begin{equation} 
	\r(t, \varepsilon ) = \r(t,0)+ \r_1(t) \varepsilon +  \r_2 (t) \varepsilon ^2 + O( \varepsilon^3) ,
	\label{eq:incr}
\end{equation}   
where $\r_1(t)=\partial_ \varepsilon \r(t, 0) $,  $\r_2(t) =\frac{1}{2} \partial ^2_ \varepsilon \r(t, 0)$.   To find $\r(t,0), \r_1(t), \r_2(t)$, we first set   $\varepsilon = 0 $ in  (\ref{eq:cauchy})  to find $ \r(t,0)\equiv \r_0 $. 
Differentiating    (\ref{eq:cauchy})  by $\varepsilon$ two times and setting  $\varepsilon=0 $ after each differentiation,  we get  
\begin{equation} 
	\dot  \r_1 =- {\bf v} + ({\bf v} \cdot \r_0) \r_0 
	\label{eq:xi}
\end{equation}   
 and 
  \begin{equation} 
	 \dot\r_2=  ({\bf v} \cdot \r_1)\r_0 + ({\bf v} \cdot \r_0)\r_1.   
	\label{eq:eta}
\end{equation}   
  From   (\ref{eq:xi}),  
  \begin{equation} 
	  \r_1   = - \Gamma + (\Gamma\cdot \r_0)\r_0  + {\bf c},   
	    \label{eq:xi1}
\end{equation} 
and in particular $ \r_1(L)=\r_1(0) $. 
Substituting     (\ref{eq:xi1})   into   (\ref{eq:eta})  and integrating, we find, after simplification, the sole surviving term (other terms drop out as the   derivatives of periodic functions): 
\[
	 \r_2  \biggl|_{t=0}^L =
	- \int_{0}^{L}(\dot \Gamma \cdot \r_0)\Gamma \,\d t=  \int_{0}^{L}( \Gamma \cdot \r_0)\dot\Gamma\,\d t,        
\]    
the last step using integration by parts. 

For $n=3$, integrating  the  identity 
$$(\Gamma\times\dot\Gamma)\times \r_0=(\r_0\cdot\Gamma)\dot\Gamma-(\r_0\cdot\dot\Gamma)\Gamma$$
  over a period and using integration by parts, one obtains the stated formula.
 \end{proof}

\begin{rmrk}
Theorem  \ref{3Dplan}  reveals  an interesting behaviour of the bicycle equation in $\R^3$ for large bike length. On the one hand,  the bicycle vector field   on $ S^2 $ given by   equation (\ref{eq:cauchy})  is ``maximally hyperbolic", in the sense that the two instantaneous equilibria at $\r=\pm\v/\|\v\|$ are  antipodal nodes on the $\r$--sphere, one stable (at $-\v/\|\v\|$) the other unstable (at $\v/\|\v\|$). On the other hand, the 
 time $L$ map $ \r_0\mapsto \r(L) $ of this vector field, i.e., the monodromy map, is ``maximally elliptic" in the sense that it  is $ O( \varepsilon ^3 )$--close to a rigid rotation, with  antipodal pair of elliptic fixed points. 
\end{rmrk}

 \begin{rmrk} \label{areas}
The entries $ \A_{ij} = - \A_{ji} $ are the signed areas of the projections of $ \Gamma $  onto the $ ij$th planes, so $ \A$ is the area bivector of the curve $ \Gamma $. For  $ n=2$,  
\[
	\A= \left( \begin{array}{cc}0 & -A_\Gamma\\ A_\Gamma& 0 \end{array} \right) , 
\]  
where $ A_\Gamma $ is the signed area of $ \Gamma $, reproducing the  planimeter formula  \eqref{eq:prytzformula}. 
\end{rmrk}
%\note{Remark added as per referee 2 complaint. -S.}
\begin{rmrk}
The classical literature on the hatchet planimeter contains a series, in the negative powers of the length of the planimeter $\ell$, of its turning angle $\theta$, see, e.g., \cite{Hi}. This makes it possible to estimate the error in measuring the area effectively. It should be possible to obtain a similar power expansion in the multi-dimensional setting; we do not dwell on it here.
\end{rmrk}

\subsection{A bird's eye view  of the hatchet planimeter} \label{bird}
\newcommand{\diam}{\mathrm{diam}}
The results of the last two subsections lead to a new intuitive understanding of  the hatchet planimeter formula \eqref{eq:prytzformula}, which we describe  in this subsection. 
Loosely speaking, the angle  by which the planimeter rotates is approximated by  the solid angle of a certain cone, as explained next.

We consider a planar curve $\Gamma\subset \R^2 $ of  small diameter 
$ \diam (\Gamma)=O( \varepsilon )$, area $A=O(\eps^2)$,  and fixed bicycle length $\ell=1$ -- this assumption is equivalent to taking a long bike of size $\ell=1/\varepsilon$ for a fixed $\Gamma$.

With $\Gamma$ so scaled, the hatchet planimeter formula \eqref{eq:prytzformula}, expressing the area $A$ bounded by $\Gamma$ in terms of the rotation angle $\theta$ of the planimeter, becomes 
\be\label{eq:plani}
A=\theta+O(\eps^3).
\ee
%%%

We now  ``lift" our point of view above   $\Gamma$ by 
considering  the plane $\R^2$ containing  $\Gamma$ as the horizonal plane $z=0$ in $\R^3$. The monodromy of $\Gamma$ then becomes a M\"obius transformation $M:S^2\to S^2$, commuting with the reflection about the horizontal plane and   conjugate to a rigid rotation about a vertical axis, with a pair of fixed points, symmetrically situated on opposite sides of the horizontal equator  $S^1\subset S^2$. Consider the closed rear track $\gamma$ corresponding to one of those  fixed point; see  Figure~\ref{fig:berry}(a). 
%%%%%
\begin{figure} [h!]\centerline{\includegraphics[width=\textwidth]{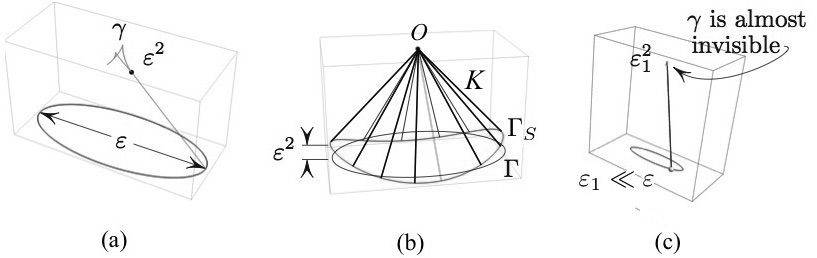}}
	\caption{(a) $\diam(\Gamma) = O ( \varepsilon)$, $\diam(\gamma) = O ( \varepsilon ^2)$. (b) $ \Gamma $ is $O ( \varepsilon^2 )$--close to  $ \Gamma_S $. (c)  As $\eps$ descreases to $\varepsilon_1\ll\eps$,    $\gamma $ becomes indistinguishable from a point, and thus $\Gamma_S $ looks  indistinguishable from $ \Gamma $, illustrating   \eqref{eq:OG}.}
	\label{fig:berry}
\end{figure}
%%%%%%%%%

Let us parallel transport every bike segment $ RF$  tangent to $\gamma$  by moving the tangency point $ R $ to a chosen point  $O\in \gamma$. The translated segments form a cone $K$ with vertex at $O$ and the translated endpoints $F$ form a curve $\Gamma_S$ on the unit sphere $S^2 $ centered at $O$, enclosing a spherical area $\Omega$, equal to the solid angle  subtended by $K$ at $O$; see  Figure~\ref{fig:berry}(b). 

The Berry's phase formula (Corollary \ref{Berry}) states   that $M$ (the monodromy  of $\Gamma$) is conjugate to a rigid rotation of $S^2$ by the angle $\Omega$ (note that $L_\gamma=0$ in equation \eqref{eq:berry} for a planar $\Gamma$ with an elliptic $M$). Furthermore, Theorem \ref{3Dplan} states that $M$ is $O(\eps^3)$-close to a rigid rotation. Combining the last two statements, we conclude that the rotation angle $\theta$ of the planimeter in formula \eqref{eq:plani} is $O(\eps^3)$-close to the solid angle $\Omega$. In other words, formula 
\eqref{eq:plani} is equivalent to the statement 
 \be\label{eq:OG} A=\Omega+O(\eps^3).
 \ee

Formula \eqref{eq:OG} suggests a   ``bird's eye view" interpretation of the planimater formula \eqref{eq:plani}: a bird standing at a point $O$,  at height 1 above  a planar curve $\Gamma$ of diameter  $O(\eps)$,  estimates its area $A$, with   $O(\eps^3)$-accuracy, by   the solid angle $\Omega$  subtended at $O$ by the $O(\eps^2)$--close curve $\Gamma_S$.

To justify the $O(\eps^2)$--closeness of $\Gamma$ and $\Gamma_S$, we argue a follows. First, we observe that 
$\diam (\gamma) = O( \varepsilon^2 )$
as Figure \ref{fig:berry}(a) illustrates; indeed,  the tangent segments $ RF$ to $\gamma$ form angles $ \pi /2 +O( \varepsilon ) $ with the plane of $\Gamma $,  and thus 
$$ 
	 \dot \gamma = ({\bf v} \cdot{\bf r} ) {\bf r}   = \|{\bf v} \| \cos ( \pi/2 + O(\varepsilon ) ) = O( \varepsilon ^2 ),
	 \label{eq:smallgamma}
$$
 since $\|{\bf v}\| =\|\dot\Gamma \| = O( \varepsilon ) $. 
When constructing  $ \Gamma_S $ we therefore moved each point $F\in \Gamma$    by $ O( \varepsilon ^2 )  $, which shows that 
$ \Gamma_S $ and $ \Gamma $ are $ O ( \varepsilon ^2 ) $--close to each other. 

This now implies   \eqref{eq:OG} as follows: first,  projecting $ \Gamma_S$ onto the horizontal plane, the area of the resulting planar region is   $O ( \varepsilon ^3 ) $--close to the spherical area $\Omega $ enclosed by $\Gamma_S$; since the projected curve is $ O( \varepsilon ^2) $--close to $ \Gamma $, and both have length $ O( \varepsilon ) $, their areas differ by  $O( \varepsilon ^3 ) $.  This explains  equation (\ref{eq:OG}).

\subsection{Bicycling  and hyperbolic  rolling}\label{sec:rolling}
The main purpose of this section is to elaborate on the equivalence mentioned before (Remark \ref{rmrk:roll}):  the bicycle equation 
 \beq, for $\|\r\|<1$, also describes the   rolling without sliding and twisting of the hyperbolic $n$--ball on the Euclidean   $n$-space $\R^n$.   A precise statement is given in Theorem \ref{thm:roll} below. We precede this statement by a discussion of rolling of the Euclidean sphere. 

 We feel that this material is not common knowledge and not easy to gather  from the literature, so we begin with an elementary exposition of rolling a ball on the plane, before moving on to  the rolling of hyperbolic $n$-space on  $\R^n$. For a more abstract ``intrinsic" treatment of rolling we recommend \cite{BH} (section 4.4), as well as \cite{B}.
For $n=2$, the material here is closely related to the ``stargazing" interpretation
of  the bicycling equation \beq\ as it appears in Section 3 of  \cite{FLT}.

\mn 

Consider a rubber ball  lying on top of the   flat rough surface  of a table, so that the ball can roll on the table, but not slide; that is, as the ball moves, at each moment the point of the ball in contact with the table has zero velocity. We paint a straight line segment  $\Gamma$ 
on the table, position the ball at one end of $\Gamma$ and  roll it  along   until it reaches the other end. As we do so,  the paint, which is still wet,   marks a  curve $\tG$ on the surface of the ball. 

Since we are rolling without sliding, $\tG$ and $\Gamma$  have  the same length. Furthermore, if we are careful not to spin the ball about  the vertical axis through its contact point with the table as  it rolls, $\tG$   is in fact a geodesic segment  (an arc of a great circle). Note also that as a result of the rolling, the ball is translated along $\Gamma$ and rotated about the  horizontal axis passing through the center of the ball and perpendicular to  the  direction of $\Gamma$. Note also that due to the no-spin condition,  a parallel field of vectors along $\Gamma$ leaves a ``track" of corresponding vectors on $ \tG $ which  form a parallel field with respect to parallel transport on the sphere. 
%\note{added on painted parallel vectors - M

%Dont see any. -G}

\begin{figure}[h!]
 \centerline{\includegraphics[width=.5\textwidth]{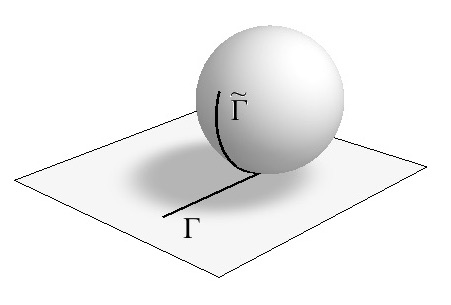}}
\caption{A ball rolling without slipping and twisting}
\end{figure}

We can of course roll the ball along a more general  curve 
$\Gamma$ drawn on the table,  in which case the no-slide and no-spin conditions  (also called  ``no-slip" and ``no-twist") imply that  the curve  $\tG$ traced on the ball has the same length  and the same  geodesic curvature as $\Gamma$ at the corresponding points. The change of orientation of the ball as a result of the rolling  is an element $g\in\SO_3$, called the {\em rolling monodromy} of $\Gamma$.

For example, if $\Gamma$ is a circle of radius $R$, then $\tG$ is an arc of a circle of latitude on the sphere, of length $2\pi R$ and geodesic curvature $1/R$, from which one can easily determine $\tG$,   as  well as $g$ (given the radius of the rolling ball). 

\mn 

Let us now formulate the above more precisely and generally. Consider the $n$-sphere of radius $\ell$, 
$$\Sp=\{(x_1, \ldots, x_{n+1})\in\R^{n+1}|(x_1)^2+\ldots+(x_{n+1})^2=\ell^2\},$$
rolling along a smoothly parametrized curve $\Gamma(t)$ in $\R^n=\{x_{n+1}=0\}\subset\R^{n+1}$.    
The rolling motion  is given by a time-dependent family of rigid motions 
$\varphi(t):\R^{n+1}\to\R^{n+1}$, so that $\varphi(t)(\Sp)$  
is positioned in the upper half space $\{x_{n+1}\geq 0\}$, tangent to $\R^n$ at 
$\Gamma(t)$. 
Then  $\varphi(t)$ can be written as 
$$\varphi(t)(\x)=g(t)\x+\Gamma(t)+\ell e_{n+1}, \quad g(0)=\II,$$
where  $g(t)\in\SO_{n+1}=\Iso^+(\Sp)$ is the ``rolling monodromy", describing the rotation of the moving  sphere at time $t$ with respect to its initial position at $t=0$. 
Let 
$$\tG(t)=\varphi(t)^{-1}(\Gamma(t))=-\ell[g(t)]^{-1} e_{n+1}\in\Sp,$$
 the ``body" 
curve of contact points.  The (space) derivative of $\varphi(t)$  is a linear isometry $T_{\tG(t)}\Sp\to T_{\Gamma(t)}\R^n$, given by $g(t)$, satisfying the following two rolling conditions:
\begin{enumerate}[(1)]

\item  No-slip: $g(t)\dot\tG(t)=\dot\Gamma(t)$;

\item  No-twist: If $\txib$  is a vector field  tangent to $\Sp$ and parallel along $\tG$, then $g(t)\txib(t)$ is parallel along $\Gamma.$ 

\end{enumerate}
\begin{rmrk}It can be easily shown that the no-slip condition  (1) is equivalent to the vanishing of  the Killing field $v(t):=\dot\varphi(t)[\varphi(t)]^{-1}$ at $\Gamma(t)$ (``the velocity of the contact point of the rolling  body with $\R^n$  is equal to zero"). Thus, for $n=2$,  $v(t)$ is the velocity vector field of rotations about an instantaneous axis passing through the contact  point $\Gamma(t)$ (the ``angular velocity" axis). Furthermore, for $n=2$, the no-twist condition (2) is equivalent  to 
the instantaneous rotation axis lying  in  $\R^2$
and perpendicular to $ \dot\Gamma $. It is also equivalent to the equality of the geodesic curvatures  of $\Gamma$ and $\tG$ at the corresponding points. 
%\note{added "and perpendicular to $ \dot\Gamma $" in the remark, removed ", so the angular velocity vector has no $e_3$ component" - M}
\end{rmrk}

\begin{prop}\label{prop:rolls} The monodromy $g(t)\in\SO_{n+1}$ of rolling $\Sp$ along a parametrized curve $\Gamma(t)$ in $\R^n$ satisfies
\be\label{eq:sroll}
\dot g={1\over \ell}\left(\begin{matrix}
0_n&\v\\
-\v^t&0
\end{matrix}\right)g, \quad g(0)=\II, \quad \v=\dot\Gamma.
\ee
\end{prop}

\begin{proof} For the sake of brevity we omit  the  explicit $t$-dependence, writing $g=g(t)$,   etc.

 Let $A=\dot g g^{-1}\in\so_{n+1}$. The statement is then that the no-slip and no-twist conditions are equivalent to (1): $Ae_{n+1}=\dot\Gamma/\ell$ and (2):  if $\xib\perp e_{n+1}$ then $A\xib\equiv 0\; (\mod e_{n+1})$ (i.e., $A\xib$ is a multiple of $e_{n+1}$). We now prove (1) and (2). 

\sn (1) $g\tG=-\ell e_{n+1}$ implies $0=\dot g\tG+g\dot\tG=-\ell Ae_{n+1}+g\dot\tG.$ Hence the no-slip condition, $g\dot\tG=\dot\Gamma$,  is equivalent to $Ae_{n+1}=\dot\Gamma/\ell$.  

\sn (2) 
Suppose  $\txib$ is parallel along $\tG$. 
Then $\dot\txib\equiv 0\; (\mod \tG)$, hence $g\dot\txib\equiv 0\; (\mod e_{n+1})$.   Let $\xib=g\txib$. Then  
$$\dot \xib =\dot g\txib+g\dot\txib=A\xib+g\dot\txib\equiv A\xib \; (\mod e_{n+1}).$$
 But $\xib\perp e_{n+1}$, hence $\dot\xib\perp e_{n+1}$ and  $\dot \xib =A\xib$. It follows that if $A$ has the form given in formula \eqref{eq:sroll}, then $A\xib=0$, hence $\dot\xib=0$, i.e., $\xib$ is parallel. 

Conversely, given a vector $\xib_0\perp e_{n+1}$ tangent to $\R^n$ at  $\Gamma(t_0)$, we let 
$\txib_0=[g(t_0)]^{-1}\xib_0$ and  extend it to a parallel vector field $\txib$  along $\tG$. Assuming the no-twist condition,  $\xib:=g\txib$ is parallel. As before, it implies that $0=\dot\xib\equiv A\xib \; (\mod e_{n+1}).$ In particular, $A(t_0)\xib_0\equiv 0\; (\mod e_{n+1}),$ as claimed. 
\end{proof}

Next, recall from Section \ref{sec:reform} the hyperboloid model for hyperbolic $n$-space of curvature $-1/\ell^2$, 
$$\Hy=\{\x\in \R^{n,1}|\<\xb,\xb\>=-\ell^2, x_{n+1}>0\}.$$ 
Given a curve $\Gamma$ in $\R^n=\{x_{n+1}=0\}\subset\R^{n,1}$, a rolling of $\Hy$ along $\Gamma$ consists of a $t$-dependent family of rigid motions $\varphi(t):\R^{n,1}\to \R^{n,1}$ (orientation preserving isometries), so that $\varphi(t)(\Hy)$  
is positioned in the upper half space $\{x_{n+1}\geq 0\}$, tangent to $\R^n$ at 
$\Gamma(t)$. 
Such  $\varphi(t)$ can be written as 
$$\varphi(t)(\xb)=g(t)\xb+\Gamma(t)-\ell e_{n+1}, \quad g(0)=\II,$$
where  $g(t)\in\SO_{n,1}=\Iso^+(\Hy)$ is the ``rolling monodromy", describing the rotation of the moving  hyperbolic $n$-space at time $t$ with respect to its initial position at $t=0$. Furthermore, $g(t)$ is required to satisfy the same no-slip and no-twist conditions that were given in the case of rolling $\Sp$. 

\begin{theorem} \label{thm:roll} 
The monodromy $g(t)\in\SO_{n,1}$  of rolling $\Hy$ along a parametrized curve $\Gamma(t)$ in $\R^n$ satisfies
\be\label{eq:hroll}\dot g=-{1\over \ell}\left(\begin{matrix}
0_n&\v\\
\v^t&0
\end{matrix}\right)g, \quad g(0)=\II, \quad \v=\dot\Gamma.
\ee
Thus,  $g(t)$ coincides with the bicycling $\ell$-monodromy of $\Gamma$ (see Theorem \ref{thmf}). 

\end{theorem}
The proof  is almost identical to the above proof of Proposition \ref{prop:rolls} and is omitted. 

\begin{rmrk}  Embedding   $\R^n$ and $\Hy$ in $\R^{n,1}$  facilitates   intuition and  calculations but is  not essential, since the no-slip and no-twist conditions are intrinsic. These conditions  thus apply to the rolling of $\Hy$ along an arbitrary  Riemannian $n$-manifold $M$, defining a principal $\SO_{n,1}$-connection on $M$,  whose associated parallel transport can be interpreted as  either the monodromy of  rolling $\Hy$ along $M$, or  the monodromy  of $\ell$-bicycling  on  $M$.

\end{rmrk}

\section{Bicycle correspondence, the filament equation and integrable systems} \label{FilCor}

\subsection{Bicycle  correspondence}\label{ss:bc}

We start by recalling the definition of the bicycle correspondence.

%
%\note{Referee 2 did'nt like the phrasing of condition (ii). I changed it a bit but I am not sure I understand him. -G.}
%

\begin{definition}\label{def:bc}Let $\ell>0$. Two smoothly parameterized curves $\Gamma_1, \Gamma_2$ in $\R^n$ are in $\ell$-bicycle correspondence    if, for all $t$,
\begin{enumerate}[(i)]
\item  the connecting segment  $\Gamma_1(t)\Gamma_2(t)$ has a fixed length $\ell$, and
\item  the midpoint curve $(\Gamma_1(t)+\Gamma_2(t))/2$ is  tangent to  the connecting segment. 
\end{enumerate}
\end{definition}
 
See Figure \ref{fig:bci} of the Introduction.  Condition (ii) can be expressed by  the following formula
\be\label{bc1}(\Gamma_1(t)-\Gamma_2(t))\wedge(\dot \Gamma_1(t)+\dot \Gamma_2(t))=0
.\ee
Here is a useful reformulation. 

\begin{lemma}\label{lemma:reflection} Two parameterized curves $\Gamma_1, \Gamma_2$ in $\R^n$ are in bicycle correspondence (for some $\ell$) if and only if,  for all $t$, the vector $\dot\Gamma_2(t)$ is the reflection of $\dot\Gamma_1(t)$ about the connecting line segment $\Gamma_1(t) \Gamma_2(t)$, followed by parallel translation from $\Gamma_1(t)$ to $\Gamma_2(t)$, 
\be\dot\Gamma_2(t)=-\dot \Gamma_1(t)+2\left(\dot \Gamma_1(t)\cdot \r(t)\right)\r(t), \quad \mbox{where }\r(t)={\Gamma_1(t)-\Gamma_2(t)\over \|\Gamma_1(t)-\Gamma_2(t)\|}.
\ee
\end{lemma}

\begin{figure}[h!]
\centerline{\includegraphics[width=.6\textwidth]{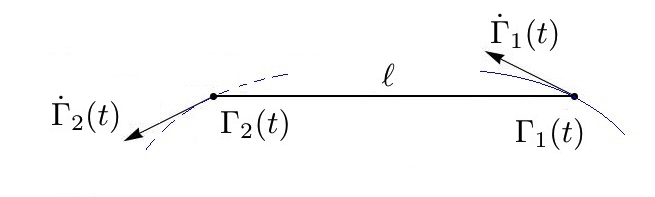}}
\caption{$\ell$-bicycle correspondence}\label{fig:gr}
\end{figure}

\begin{proof} Condition (i)  of Definition \ref{def:bc} is equivalent to the equality of the orthogonal projections of  $\dot \Gamma_1(t)$ and $\dot \Gamma_2(t)$ onto $\r(t)$. Condition (ii), by formula \eqref{bc1}, is then that the orthogonal components sum up to 0. 
\end{proof}

%\note{\small Referee 2 wants to change this to ``{\bf The} bicycle correspondence\ldots". I disagree. Compare to ``Riemannian isometry is arc length preserving". Or ``Inversion with respect to a circle is a conformal map." -G.
%How about asking Ron, the only native speaker among us? - S.}
\begin{cor}\label{cor:arc}Bicycle correspondence is arc-length preserving. 
\end{cor}
\begin{proof}This follows from Lemma \ref{lemma:reflection}  and the fact that reflection and parallel translation are isometries. 
\end{proof}

The main result of this section is that bicycle correspondence preserves the  bicycle monodromy (Theorem \ref{thm:bc}). 
This result is not new: in \cite{TT}, it is established for the discrete version of the bicycle correspondence, defined for polygons in $\R^n$, and in the smooth case, it follows by taking limit. Here we give a different proof whose  
idea is to conjugate the bicycle monodromies  along the corresponding curves  using ``Darboux Butterflies", which we now introduce. 

\begin{definition}\label{def:dar} A Darboux Butterfly in $\R^n$ is the result of ``folding" a parallelogram about one of its diagonals; more precisely, it is  an ordered quadruple $ABCD$ of  4 distinct points  in $\R^n$, such that $D$ is the reflection of the point $A-B+C$ about the line $AC$.

\end{definition}

\begin{figure}%[h!]
\centerline{\includegraphics[width=\textwidth]{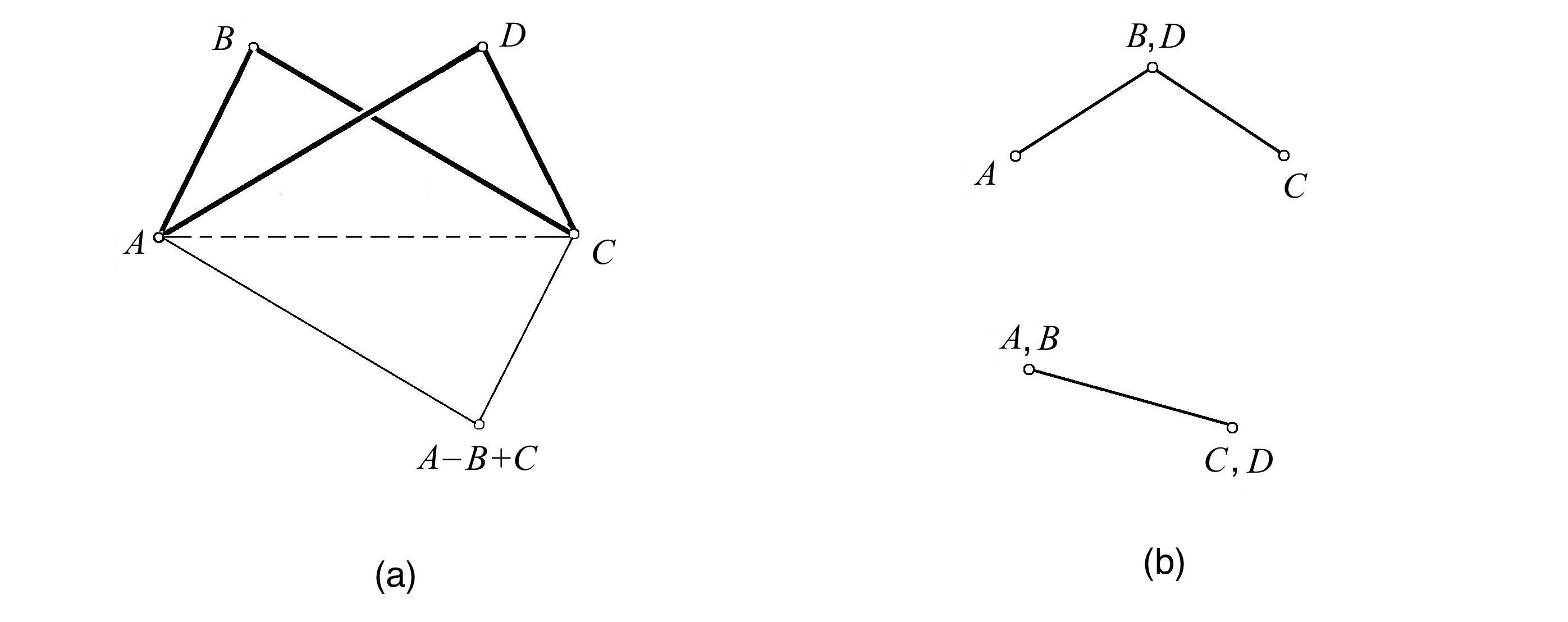}}
\caption{(a)  The definition of a Darboux butterfly;  (b) Degenerate butterflies}
\end{figure}

\begin{rmrk}\label{rmrk:deg}The  above definition applies also   to ``degenerate" butterflies $ABCD$, where one or two pairs of  points coincide, as long as $A\neq C$ or  $B\neq D$, so one can apply the definition, or the equivalent one:  $A$ is the reflection of $B-C+D$ about $BD$. 
\end{rmrk}

Here are some immediate consequences of  Definition \ref{def:dar}:
\begin{lemma}\label{lemma:immed}

\begin{enumerate}[(i)]
\item  A Darboux Butterfly is a planar quadrilateral.

%\note{rephrased condition (ii) to answer a comment of Referee 2. -G.}
\item   The butterfly property   is invariant under  cyclic permutation and order reversing of its vertices. Namely, if $ABCD$ is a Darboux butterfly, then so are $BCDA$ and  $DCBA$.

\item   Any triple of points  $ABC$ in $\R^n$ with $A\neq C$ can be completed uniquely to a Darboux Butterfly $ABCD$ (possibly degenerate; see Remark \ref{rmrk:deg}). 

\end{enumerate}

\end{lemma}

%\note{I modified slightly the phrasing of the definition and notation for `glide reflection', since referee 2 seemed to be confused by this definition and our slick proof of the Butterfly lemma. I am not sure it is better. Please check! -G.}
Another property of Darboux Butterflies is the following infinitesimal version of the ``Butterfly Lemma" of \cite{TT}. To formulate it, we first define for a given  segment $UV$ in $\R^n$  the {\em glide reflection} $G_{UV}:\R^n\to\R^n$ as the composition of the reflection about  the line through $UV$, followed by parallel translation through the vector $V-U$. For example, in Figure \ref{fig:gr}, $\dot\Gamma_2(t)$ is the image of $\dot\Gamma_1(t)$ under  $G_{\Gamma_1(t)\Gamma_2(t)}$.  

\begin{lemma}\label{lemma:bf}
For any Darboux Butterfly $ABCD$, one has: 
$$G_{DA} \circ G_{CD} \circ G_{BC} \circ G_{AB} = Id.$$
\end{lemma}

\begin{proof} The linear part of the isometry in question  is the composition of four reflections. 
Decompose $\R^n$ into the direct sum of the plane of the butterfly, translated to the origin, and its orthogonal complement. In the orthogonal complement each reflection acts by $-1$, hence the composition of the four reflections acts trivially. 

In the plane of the butterfly, the product of the  reflections about two successive edges is a rotation by twice the angle between the edges; the product of reflections about the next pair of successive edges is then a rotation by the same angle in opposite direction. 

It follows that the linear part of the  isometry in question is trivial. Thus it is a parallel translation. But $A$ is a fixed point, hence it is the identity.
\end{proof}

The next statement is a version of the  ``Bianchi permutability" \cite{RS}, proved in \cite{TT} for a polygonal version of the bicycle correspondence. 

\begin{prop}\label{prop:bp}
Let  $\ell_1, \ell_2 >0$ with $\ell_1\neq \ell_2$ and let $A(t),B(t),C(t)$ be  three parameterized curves in $\R^n$ such that  $(A, B)$ and  $(B, C)$ are in $\ell_1$- and $\ell_2$-bicycle correspondences, respectively. Complete $A(t)B(t)C(t)$ to a Darboux Butterfly $A(t)B(t)C(t)D(t)$ (the non-degeneracy assumption $\ell_1\neq \ell_2$ assures that $A(t)\neq C(t)$ so Lemma \ref{lemma:immed}(iii) applies). Then $(A,D)$
 and $(C,D)$ are in $\ell_2$- and $\ell_1$-bicycle correspondence, respectively.
\end{prop}

\begin{figure}[h!]
\centerline{\includegraphics[width=.5\textwidth]{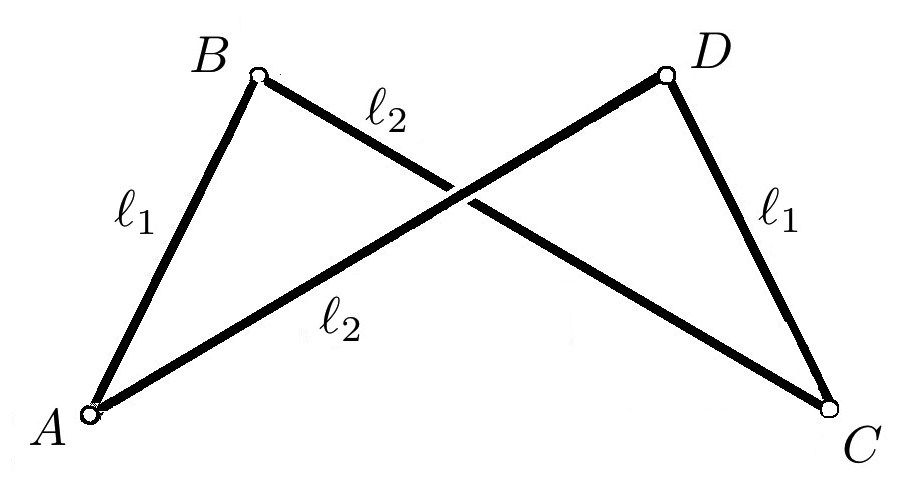}}
\caption{Bianchi permutability}
\end{figure}

\begin{proof}
%
%\note{Edited the proof, giving (slightly) more details, since referee 2 seemed to have a hard time with this proof. Maybe we should give even more details. Please check. -G.}
By Lemma \ref{lemma:reflection}, $\dot A=G_{BA}\dot B, \dot C=G_{BC}\dot B$, and we need to show that $\dot D=G_{AD}\dot A= G_{CD}\dot C$.

Now $\|A-D\|=\ell_2$ and $ \|C-D\|=\ell_1$ imply the ``non-stretching condition": the  orthogonal projections of $\dot D$ onto $AD$ and $CD$  coincide with the orthogonal projections of $\dot A$ and $\dot C$ onto $AD$ and $CD$, respectively. If  the butterfly is non-collinear then $AD,CD$ are linearly independent, hence $\dot D$ is determined uniquely by the non-stretching condition. 
On the other hand, using Lemma \ref{lemma:bf}, we have
$$
G_{AD}\dot A = G_{AD} G_{BA}\dot B = G_{CD} G_{BC}\dot B = G_{CD}\dot C,
$$
 and this vector  clearly satisfies the ``non-stretching condition".  Hence $\dot D=G_{AD}\dot A= G_{CD}\dot C.$
For a collinear butterfly, the result follows by continuity from the non-colinear case.
\end{proof}

The next result shows that the  flows of the bicycle equation along curves  in bicycle correspondence are conjugated.

%\note{Answering a comment of referee 2, added the statement that $\Phi^\lambda(t)$ is Mobius, and a reference to its proof in \cite{TT}.  -G.}
\begin{lemma} 
Let $\Gamma_1, \Gamma_2$ be two parameterized curves in $\R^n$ in $2\ell$-bicycle correspondence. For each $\lambda\neq \ell$, let $\Phi^\lambda(t):S^{n-1}\to S^{n-1}$ be the map $\r_1\mapsto \r_2$ defined by completing $\Gamma_1(t)+2\lambda\r_1, \Gamma_1(t), \Gamma_2(t)$ to a Darboux butterfly  
$$\Gamma_1(t)+2\lambda\r_1, \Gamma_1(t), \Gamma_2(t), \Gamma_2(t)+2\lambda\r_2
$$ 
(see Figure \ref{fig:bcl}). Then $\Phi^\lambda(t)$ is a M\"obius transformation (possibly orientation reversing), conjugating the $\lambda$-bicycle flows along $\Gamma_1, \Gamma_2.$ That is, if $\r_1(t)$ satisfies $\lambda\dot\r_1=- \v_1+(\v_1\cdot \r_1)\r_1$ then  $\r_2(t):=\Phi^\lambda(t)\r_1(t)$ satisfies $\lambda\dot\r_2=- \v_2+( \v_2\cdot \r_2)\r_2$, where $\v_1=\dot\G_1, \v_2=\dot\G_2.$
\end{lemma}

\begin{figure}[h!]
\centerline{\includegraphics[width=.4\textwidth]{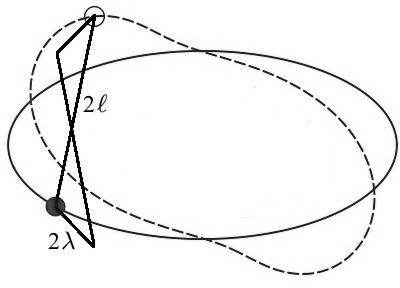}}
\caption{Two closed curves in $2\ell$-bicycle correspondence, with the Darboux butterfly conjugating their $\lambda$-bicycle flows}\label{fig:bcl}
\end{figure}

\begin{proof} If $\r_1(t)$ solves the $\lambda$-bicycle equation along $\Gamma_1$  then $\Gamma_1+2\lambda\r_1$ is in $2\lambda$-bicycle correspondence with $\Gamma_1$. By Bianchi permutability (Proposition \ref{prop:bp}), $\Gamma_2+2\lambda\r_2$ is in $2\lambda$-bicycle correspondence with $\Gamma_2$, hence $\r_2(t)$ is a solution to the $\lambda$-bicycle equation along $\Gamma_2$.

%\note{Edited a bit Mark's proof. -G.} 
The  proof that $\Phi^\lambda(t)$ is a M\"obius transformation was given in the proof of Theorem 1 of \cite{TT}. Here we present an alternative proof for $n=2$, i.e., $\Phi^\lambda(t):S^1\to S^1$, where $S^1\subset\C$. We denote the image of $z$ by $w$, where $ | z |=|w|=1 ,$  see Figure~\ref{fig:mobbutt}. Expanding $|   2\ell + 2\lambda w- 2\lambda z| ^2 = (2\ell) ^2 $, we obtain
\begin{equation} 
	\Re\bigg(w(\lambda\bar z-\ell)+(\ell\bar z-\lambda)\bigg)=0.
	\label{eq:dist}
\end{equation}  
Geometrically, it is clear that given any $z\not=\pm1$ on the unit circle, there are precisely two solutions $w$ to  \eqref{eq:dist}, and that  one of them is $ w= z $ (corresponding to the parallelogram, rather than the butterfly). The other  (algebraically)  obvious  solution  is given simply by 
\[
	w(\lambda\bar z-\ell)+(\ell\bar z-\lambda)=0,    
\]  
or
 \[
	w= -\frac{\ell\bar z-\lambda}{\lambda\bar z-\ell}.
\]  
The last formula shows that  $ z\mapsto w $ is a reflection  $z\mapsto \bar z$, followed by a  M\"obius transformation, the projectivizaton of 
$\left(
\begin{array}{rr}\ell&-\lambda\\ -\lambda&\ell\end{array}
\right)\in\GL_2(\R)$,  as claimed. 
\begin{figure}[h!]
\centerline{\includegraphics[width=.6\textwidth]{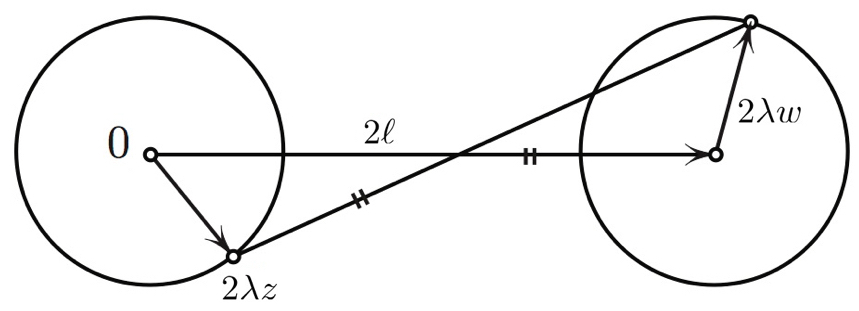}}
\caption{Proving that $\Phi^\lambda$ is a M\"obius transformation}\label{fig:mobbutt}
\end{figure}
\end{proof}

\begin{theorem}\label{thm:bc}If $\Gamma_1, \Gamma_2$ are two closed curves in $\R^n$ in $2\ell$-bicycle correspondence then, for all $\lambda>0$, their $\lambda$-bicycle monodromies are conjugate elements of $\SO^+_{n,1}$. 
\end{theorem}

\begin{proof}
%\note{changed `last theorem' to `last lemma', following referee 2 comment.-G. }
For $\lambda\neq \ell$ this follows from the last lemma. By continuity, it then follows also for $\lambda=\ell$. 
\end{proof}

This theorem provides integrals of the bicycle correspondence: a conjugacy-invariant function on the M\"obius group, considered as a function of $\lambda$, is such an integral. Individual integrals can be obtained by expanding such a function in a series in $\lambda$. We call such integrals of the bicycle correspondence the monodromy integrals.

For a discussion of symplectic properties and complete integrability of the bicycle correspondence, see \cite{T2}.

\subsection{The bicycle equation and the filament equation} \label{bikefil}

In this subsection we describe a relation between the bicycle equation \bode\ in $\R^3$ and the {\em  filament  equation} (also called  the {\em localized induction equation}, among several other names). The later is an evolution equation  on arc length parameterized curves $\Gamma(t)$ in $\R^3$, 
$$
 \Gamma' = \dot \Gamma \times \ddot \Gamma,
$$ 
where prime $'$ denotes time derivative and dot stands for  derivative with respect to arclength $t$ along $\Gamma$ (this unconventional choice is forced by the prior role of $t$ in this paper). In other words,  the point $\Gamma(t)$ moves in the  binormal direction  with velocity equal to the curvature $\kappa(t)$. 
%
%\note{Added the vortex line motion interpretation. -ML}
% 
This equation provides a simplified model of the the motion of a vortex line in ideal fluid.  Here we are concerned  with closed curves. 

This  infinite-dimensional system is completely integrable in the following sense (see \cite{LP,La}). It is a Hamiltonian system with respect to the 
so-called Marsden-Weinstein symplectic structure on the space of arc length parameterized curves, the Hamiltonian function being the perimeter of the curve. We do not use this symplectic structure in the present paper, so we simply refer to \cite{MW} and \cite{AK}, p.~326 and p.~332, for its definition and main properties. 
%\note{added page numbers. -ML}

The filament equation  has a hierarchy of Poisson commuting integrals $F_1,F_2,\dots$ that starts with
\begin{equation} \label{integrals}
\int 1\ \d t,\ \int \tau\ \d t,\ \int \kappa^2\ \d t,\ \int \kappa^2\tau\ \d t, \int \left(\dot\kappa^2+\kappa^2 \tau^2 -\frac{1}{4} \kappa^4\right)\d t, \dots,
\end{equation}
where, as before, $\tau$ is the torsion and $\kappa$ is the curvature of $\Gamma$. 
%\note{edited the phrasing here a little, after a comment of referee 2. -G.} 
One also has a hierarchy of  vector fields $\x_0, \x_1, \x_2,\dots$ along $\Gamma$, that starts with
\be \label{fields}
-\v,\  \kappa \b, \frac{\kappa^2}{2} \v + \dot\kappa \nn + \kappa \tau \b, \ \kappa^2\tau \v +(2\dot\kappa\tau + \kappa\dot\tau)\nn+\left(\kappa\tau^2-\ddot\kappa-\frac{\kappa^3}{2}\right)\b, \dots,
\ee
where, as before, $\v,\nn,\b$ is the Frenet frame along $\Gamma$. For each $i\geq 1$, $\x_i$ defines a Hamiltonian vector field on the space of arc length parameterized curves in $\R^3$, whose 
Hamiltonian with respect to the Marsden-Weinstein structure is $F_i$. 

The vector fields $\x_i$ satisfy the   relations 
%\note{deleted "recurrence" from "recurrence relation". -ML}
\begin{equation} \label{recur}
 \dot \x_i=\v\times \x_{i+1}, \quad i=0,1,2,\ldots.
\end{equation}
Following \cite{La}, \cite{FM}, consider the generating function
$
\x:=\sum_{j\ge 0} \eps^j \x_j,
$
where $\eps$ is a formal parameter. Then the  relations (\ref{recur}) can be compactly encoded in the  equation
\begin{equation} \label{genf}
\eps \dot \x=\v\times \x, \quad \v=\dot\Gamma.
\end{equation}
%\note{changed per suggestion (8) of referee 3. -ML}
We impose an additional normalization condition $\x\cdot\x=1$;   the vector fields $ \x_i $ are then uniquely defined by equation \eqref{genf} and $\x_0=-\v$ (see \cite{La}). 
\begin{rmrk}
%\note{changed ``every $\eps \neq 0$" to ``any $\eps \neq 0$". Maybe this is what confused referee 2. -G.}
The series $\x=\sum_{j\ge 0} \eps^j \x_j$ is a formal periodic solution of the differential equation (\ref{genf}) on the sphere. We do not claim that it converges and represents a genuine periodic solution  for any $\eps \neq 0$. 
\end{rmrk}

Now let us compare equation \eqref{genf} with the bicycle equation \eqref{eq:bode3d} in $\R^3$:
\be\label{eq:bode3d1} 
\ell\dot\r=(\v\times \r)\times\r, \quad \r\cdot \r =1,\quad \v=\dot\Gamma.
\ee
Each of the right hand sides of the last two displayed equations defines a  time-dependent vector field  on $S^2$, determined by $\v(t)=\dot\Gamma(t)$. For equation 
\eqref{genf}, it is the velocity  field of rotations about the axis $\R\v$ with angular velocity $\|\v\|$. 
\begin{prop}\label{prop:vf}
The vector field on  the right hand side of  equation \eqref{genf}  is obtained from that on the  right-hand side of equation  \eqref{eq:bode3d1} by an anti-clockwise rotation by 90 degrees.

\end{prop}
The proof  is straightforward  from the equations. 

\begin{figure}[h!]
\centerline{\includegraphics[width=.6\textwidth]{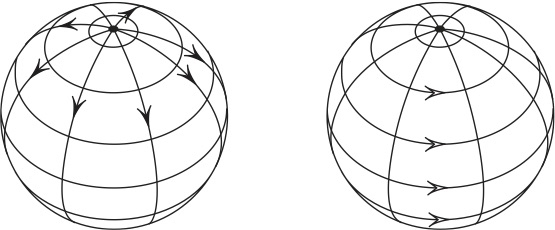}}
\caption{The vector field of the filament equation \eqref{genf} (right)  is obtained from that of the bicycle equation   \eqref{eq:bode3d1}   (left) by a $90^0$ anticlockwise rotation; the fixed points (``north" and ``south poles")  of both vector fields  are  in the direction of $\pm\v$.}\label{fig:spheres}
\end{figure}

Consequently, if we use a stereographic projection to put a complex coordinate on $S^2$, the resulting Riccati equations (in an inertial frame) differ by multiplication of their right-hand-sides  by $i$:
\begin{align*}
\ell\dot z&={1\over 2}\left(-q+2v_1z+\bar q z^2\right), &\mbox{(Bicycle)}\\
\eps\dot w&={i\over 2}\left(-q+2v_1w+\bar q w^2\right), &\mbox{(Filament)}
\end{align*}
where 
$$ \dot \Gamma=\v=(v_1, v_2, v_3), \quad q=v_2+iv_3.
$$
In other words, the filament equation \eqref{genf} is the equation of a bicycle with  ``imaginary length" $\ell=-i\eps$. 

 We also have a filament analog  of Proposition \ref{comp3DR}: 
we project the curve $\x(t)$ stereographically  from $-\v$ (the south pole in Figure \ref{fig:spheres}) on  the $\nn,\b$ plane, equipped with  complex coordinate $W=(X_2+iX_3)/(1+X_1),$ where $\x=X_1\v+X_2\nn+X_3\b,$ and express  (\ref{genf}) as a differential equation on ${W}$.

\begin{prop} Equation \eqref{genf} is equivalent to the Riccati equation
\begin{equation} \label{RicW}
\dot {W} =i\left({1\over \eps}-\tau\right)W-{\kappa\over 2}(1+W^2).
\end{equation}
\end{prop}

\n The proof is  a  direct calculation that we do not reproduce here.

\begin{rmrk}
Equation (\ref{genf}) is a particular case of the equation 
$\dot \x = \Omega (t) \x$, where $\Omega$ is a time-dependent skew-symmetric $3\times 3$ matrix,
studied in \cite{Le}. Its monodromy is the orthogonal transformation that relates the initial and terminal unit tangent vectors of the spherical curve $\b(t)$ whose geodesic curvature, in our case, equals
$$
\frac{\tau(t)+\frac{1}{\eps}}{\kappa(t)}.
$$
\end{rmrk}

Comparing the filament Riccati equation (\ref{RicW}) to the bicycle Ricatti equation \eqref{eq:ric3}
\be\label{eq:ric7}
\dot {Z} =\left({1\over \ell}-i\tau\right)Z-{\kappa\over 2}(1+Z^2),
\ee
we see that they are almost the same.

%
%\note{Rephrased a bit this corollary. -G

%\mn 

%Rewrote the  part between Cor. \ref{same} and Prop. \ref{intint}, responding to referee 2 confusion; also,  this part  has become somewhat confused after all the small changes we introduced in this section. -G.}
%

\begin{cor} \label{same}
The filament Riccati equation (\ref{RicW}) is obtained from the bicycle Ricatti equation  \eqref{eq:ric7} by replacing the variable $Z$ with $W$ and  $\ell$ with   $-i\eps$.
\end{cor}

%\note{}

 We now come to the main point  of this section. Let us describe first the idea before proceeding to the technical details.  Let us fix a closed smooth  curve $\Gamma$ in $\R^3$. 
 The last corollary suggests a relation between the filament  integrals $F_i$   of equation \eqref{integrals} and the bicycle equation \eqref{eq:bode3d1}. It is natural to seek this relation by looking at the conjugacy class of the   bicycle monodromy $M_\ell\in\PSLtc$.  The later has (generically) two fixed points in $S^2$, one stable and one unstable,  and the  derivative at one of them determines the conjugacy class of   $M_\ell$ (the derivatives are the reciprocal of each other and each   is the  square of the corresponding eigenvalue of a matrix in  $\SL_2(\C)$ representing $M_\ell$). We denote by $\lambda(\ell)$, for $\ell$ small enough,  the derivative of $M_\ell$ at the unstable fixed point. We  will show that $\ell\ln\lambda(\ell)$ has a Taylor series at $\ell=0$, i.e., is infinitely differentiable  (we do not claim analyticity). The Taylor coefficients of $\ell\ln\lambda(\ell)$  at $0$ are the   {\em monodromy integrals} of $\Gamma$. They share with $F_i$ the property of being invariant under bicycle correspondence. By linearizing the Riccati bicycle equation  \eqref{eq:ric7} around the unstable periodic solution, we are able to express the monodromy integrals, like the filament invariants $F_i$, as integrals along $\Gamma$ of certain differential polynomials  in $\kappa, \tau$. We can then check that the first few monodromy  integrals coincide with the filament  invariants $F_i$, up to index shift and  multiplicative constants.  This suggests  a  conjectured relation between    the filament integrals and the  monodromy invariants.

The following proposition is the main technical tool for implementing the above plan.  

\begin{prop} \label{intint}
Consider a closed smooth curve $\Gamma$ in $\R^3$. 
Then there exists $\ell_0>0$ such that for all $\ell\in(0,\ell_0)$ the associated Ricatti equation \eqref{eq:ric7}  has a unique  unstable periodic solution  $Z(t,\ell)$, tending uniformly, with all its derivatives,  to the zero function, as $\ell\to 0$. Furthermore, extended to $\ell=0$ via $Z(t,0)=0$, $Z(t,\ell)$  is infinitely differentiable in   $\R\times [0,\ell_0)$ \end{prop}
%
%\note{I added in Prop \ref{intint} the 1st statement  (existence and uniqueness  of an  unstable fixed point).  Uniqueness is trivial, as there are at most 2 fixed points and they cannot be both unstable. The existence has been already proved, without saying so explicitly. See my notes in Appendix \ref{app:pf}. 

%\mn Prop. \ref{intint} was previously a theorem. I changed it because the following result was previously a corollary and now it is a theorem, as  suggested by referee 2.

%\mn
Changed Cor. \ref{compI} to Thm. \ref{compI}, and expanded its statement, at the request of referee 2. 

%\mn Tried to give Thm. \ref{compI} a ``crisp formulation", as requested by referee 2. 

%\mn Changed indexing of $I_n$ (shifted by 1)... previously $I_n$ were the Taylor coefficients of $\ln\lambda$, now it is of $\ell\ln\lambda$. -G. 

%\mn I think it is  the unstable one -- M.

%\mn Agree. -G. } 

We defer the  proof of this proposition to Appendix \ref{app:pf}. 

\begin{theorem} \label{compI}
Let $\Gamma$ be a closed smooth curve in $\R^3$ and denote by  $\lambda(\ell)\in\C$ the derivative of the $\ell$-bicycle monodromy at its unstable fixed point, for small enough $\ell$ (as per Proposition \ref{intint}). Then  $\ell \ln \lambda(\ell)$ extends to an  infinitely differentiable function in  $[0,\ell_0)$ for some $\ell_0>0$. The Taylor coefficients at $0$,  i.e., the  numbers 
$$I_n:={1\over n!}(\partial_\ell)^n\big|_{\ell=0}\left[\ell\ln \lambda(\ell)\right],\quad n\geq 0,$$
are invariants of the bicycle correspondence and can be determined  recursively as integrals along $\Gamma$ of  polynomials in $\kappa, \tau$ and their derivatives. \end{theorem}

\begin{proof}We  linearize equation (\ref{eq:ric7}) at the unstable periodic solution $Z(t,\ell)$, writing this linearization in the form
$$
\frac{\dot U}{U} = \frac{1}{\ell} -i \tau  - \kappa  Z.
$$
Integrating both sides  over a  period and multiplying  the result by $\ell$ gives 
\begin{equation} \label{eq:lambda}
\ell\ln \lambda (\ell) = \int \left(1 -i \ell\tau  - \ell\kappa  Z\right) \d t. 
\end{equation} 
By Proposition \ref{intint}, this expression is $C^\infty$ in $[0,\ell_0)$ for some $\ell_0>0.$ To compute the derivatives of the last equation with respect to $\ell$ at $\ell=0$, we need to calculate the derivatives
$$Z_n:={1\over n!}\left.(\partial_\ell)^n\right|_{\ell=0}Z, \quad n\geq 0.$$
To find   $Z_n$, we first multiply  the bicycle Riccati equation   \eqref{eq:ric7} by $\ell$, obtaining 
\begin{equation} \label{Ricmod}
\ell\dot Z= Z-\ell f, \quad \mbox{where } f=i\tau Z+{\kappa\over 2}(1+Z^2),
\end{equation}
then  differentiate $n$ times at $\ell=0$ (an operation justified by Proposition \ref{intint}), obtaining, after some manipulation, 
the  recursion relation $$Z_n=\dot Z_{n-1}+{1\over (n-1)!}\left.\partial^{n-1}_\ell\right|_{\ell=0}f.$$

%\note{Added to satisfy referee 2 request. -S}
Equivalently, and perhaps easier for calculations, one expands $Z$ as a formal power series in $\ell$
$$
Z=Z_0 + \ell Z_1 + \ell^2 Z_2 + \ldots
$$
substitutes in (\ref{Ricmod}), and equates the terms having the same degree in $\ell$. 

In this way, one consecutively finds 
$$
Z_0 =0, \ Z_1 ={ \kappa\over 2},\ Z_2 ={ \dot\kappa\over 2} + i {\tau \kappa\over 2},\ 
Z_3 = \left({ \ddot\kappa\over 2} +{\kappa^3\over 8} - { \tau^2 \kappa\over 2}\right) + i \left({ \dot\tau \kappa\over 2} + \tau \dot\kappa\right),\ldots
$$
and so on, and then
\begin{align}\label{eq:Iexpand}
\begin{split}
I_0 &= \int 1\ \d t, \ I_1 = -i\int \tau\ \d t,\ I_2 = -\frac{1}{2} \int \kappa^2\ \d t, \ I_3 = -\frac{i}{2} \int \kappa^2\tau\ \d t,\\ 
I_4 &= -\int \left(\frac{\kappa\ddot\kappa}{2} - \frac{\kappa^2\tau^2}{2} + \frac{\kappa^4}{8}\right)\d t, \ldots
\end{split}
\end{align}
and so on. 
\end{proof}

We make two observations: 
\begin{enumerate}
\item The integrals $I_n$ are real for even $n$ and imaginary for odd $n$. 

\item Up to multiplicative constants and index shift, the integrals (\ref{eq:Iexpand}) coincide with the integrals  \eqref{integrals} of the filament equation.
\end{enumerate}

We will not attempt to justify these observations formally here and leave their validity as conjectures, to be studied in future work. 

%\note{Calculated explicitly several  $I_n$, as referee 2 asked. -G.}

\begin{rmrk} One can easily verify these observations  by explicit calculation for the first several cases. For example, we can see immediately from formulas \eqref{integrals} and \eqref{eq:Iexpand} that 
$$I_0=F_1, \ I_1=-iF_2, \ I_2=-{1\over 2}F_3, \ I_3=-{i\over 2}F_4.$$ As for $I_4$, we make use of the presence of total derivatives, 
$$I_4-{1\over 2}F_5=-{1\over 2}\int \left(\kappa\ddot\kappa+\dot\kappa^2\right)\d t=
-{1\over 2}\int {\d\left(\kappa\dot\kappa\right)\over \d t}\d t=0,$$
concluding that  $I_4={1\over 2}F_5.$ We have tested in a similar fashion observations 1 and 2 for at least ten additional terms.

A heuristic argument for observation 1 is as follows: Corollary \ref{same} suggests  a Laurent  expansion $\ell\ln\lambda(\ell)=\sum I_n\ell^n$ in the {\em complex} $\ell$-plane. For imaginary $\ell$, i.e., $\ell=-i\varepsilon$, the monodromy is an orthogonal transformation (since the linear system \eqref{genf} has antisymmetric coefficient matrix), hence $\ln\lambda$ is imaginary and $\ell\ln\lambda(\ell)=\sum (-i)^nI_n\varepsilon^k$ is real. Thus $I_n$ is real for even $n$, imaginary for odd $n$. 
\end{rmrk}\begin{rmrk} 
That the bicycle correspondence preserves the integrals (\ref{integrals}) of the filament equation is proved in \cite{T2}.
\end{rmrk}

\subsection{Wegner's curves, buckled rings, and solitons of the planar filament equation} \label{elastica}

In this section we show that the Zindler curves constructed by Wegner in \cite{We1}--\cite{We6} are buckled rings. The latter are also  solitons of the  planar filament equation (specified later in this section), see \cite{La}.

As we mentioned in the introduction, a (planar) Bernoulli elastica is an extremum of the total squared curvature (bending energy) among curves of  fixed length. That is, if $\Gamma(t)$ is an arc length parameterized curve with curvature $\kappa(t)$, one is looking for extrema of $\int_{\Gamma} \kappa^2(t)\, \d t$, subject to the constraint that the perimeter  is fixed. The extremal curves satisfy the Euler-Lagrange equation
$$
\ddot \kappa + \frac{1}{2} \kappa^3 + \lambda \kappa = 0,
$$   
where $\lambda$ is a Lagrange multiplier; see, e.g., \cite{Si}. 

A  buckled ring is an extremum of the  total squared curvature functional, subject to  two constraints:  both the perimeter  and the  area are being fixed.   
%\note{changed from "subject to the perimeter and area constraints". -S.} 
Buckled rings have been extensively studied, starting with L\'evy \cite{Lev}, Halphen \cite{Hal} and Greenhill \cite{Gr} in the 19th century; see \cite{AC,DVM} for recent works. 
The area constraint gives rise to a  second Lagrange multiplier, $\mu$, in the Euler-Lagrange equation of a buckled ring:
\begin{equation} \label{EL}
\ddot \kappa + \frac{1}{2} \kappa^3 + \lambda \kappa = \mu.
\end{equation}

Let us turn attention to Wegner's curves, which  come in two flavors, the linear (non-closed) and the circular (closed) ones. 
%\note{I propose to add: The former are the curves whose curvature is proportional to the distance to the $x$--axis, and are given by the graphs ..."-M} 
The linear  curves are the curves whose curvature is proportional to the distance to the $x$--axis. They are the graphs  $y=f(x)$, in  Cartesian coordinates, where $f$ satisfies the differential  equation 
\begin{equation} \label{main1}
\frac{1}{\sqrt{1+f_x^2}} = a f^2 +b
\end{equation}
with  parameters $a,b$. 
%\note{Maybe add here: 
The circular  curves are the curves whose curvature is proportional to the distance to the origin, and are given by  the graphs $r = r(\psi)$, in  polar coordinates, satisfying the differential equation 
\begin{equation} \label{main3}
\frac{1}{\sqrt{r^2+r_\psi^2}} = ar^2+b + \frac{c}{r^2}
\end{equation} 
with parameters $a,b,c$.

\begin{theorem} \label{WegEl}
Wegner's curves are buckled rings:  equation (\ref{EL}) holds for the linear Wegner curves with 
$\lambda=2ab, \mu=0$, and for the circular Wegner curves with $\lambda=8ac-2b^2, \mu=8a$.
\end{theorem}  

\begin{proof}
Consider the linear case first. Let $x(t),y(t)$ be an arc length parameterization of this curve, $\theta(t)$ be its direction, and $\kappa(t)$ its curvature. Then
%\note{changed $k$ to $\kappa$. -G.}
\begin{equation} \label{Fr1}
\dot x=\cos\theta,\ \dot y=\sin\theta,\ \dot\theta = \kappa.
\end{equation}
The left hand side of (\ref{main1}) is $\dot x$, and we rewrite this equation as
\begin{equation} \label{main2}
\dot x=ay^2+b.
\end{equation}

We claim that $\kappa=-2ay$. Indeed, differentiate the first equation of  (\ref{Fr1}) and equation (\ref{main2}),
$$
-\dot\theta \sin\theta = \ddot x = 2a y \dot y,
$$
and use the second and third equations of (\ref{Fr1}),  
$$
- \kappa \sin\theta = 2a y \sin\theta,
$$
implying the claim.

Now using $\kappa=-2ay$, equation \eqref {Fr1} and (\ref{main2}), we find that 
$$
\ddot \kappa + \frac{1}{2}  \kappa^3 = 4a^2 y (ay^2+b)-4a^3y^3 = 4a^2by=-2ab\kappa,
$$
as needed.

The argument in the circular case is similar. Let $r(t),\psi(t)$ be the polar coordinates of the curve in an arc length parameterization and  $\alpha=\theta-\psi$  the angle between the tangent to the curve and the radial direction. Then 
\be\label{eq:dr}
\dot r=\cos\alpha,\quad  r\dot\psi=\sin\alpha, \quad  \dot\theta=\kappa=\dot\psi+\dot\alpha.
\ee 
The left hand side  of \eqref{main3} is $\dot\psi$, hence  we can rewrite this equation 
as 
\be\label{eq:dpsi}\dot\psi=ar^2+b+\frac{c}{r^2}.
\ee
We claim that  $\kappa = 4ar^2+2b.$ Indeed, from  \eqref{eq:dr} and \eqref{eq:dpsi}, we get
$$\sin\alpha=r\dot\psi=ar^3+br+\frac{c}{r},$$ 
hence 
$$(\kappa-\dot\psi)\dot r=\dot\alpha \cos\alpha=(\sin\alpha)\dot{}=\dot r \left(3ar^2+b-\frac{c}{r^2}\right).$$
  Thus   
  $$\kappa=\dot\psi+3ar^2+b-\frac{c}{r^2}=
\left(ar^2+b+\frac{c}{r^2}\right)+3ar^2+b-\frac{c}{r^2}=4ar^2+2b,$$
 as claimed.

Now using $\kappa=4ar^2+2b,$ equation \eqref{eq:dr}, and \eqref{eq:dpsi}, we find
\begin{align*}
\ddot \kappa + \frac{1}{2}  \kappa^3 &=8a[1-\kappa(ar^4+br^2+c)]+\frac{1}{2}(4ar^2+2b)^3=\\
&=4 (2 a + b^3 - 4 a b c) + 8a(  b^2 - 4 a c) r^2=2 (b^2 - 4 a c)\kappa+8a,
\end{align*}
as needed. 
\end{proof}

It is known that buckled rings are solitons of the planar filament equation. Let us review this material. 

Let $\Gamma(t)$ be a smooth  planar arc length parameterized curve,   $(\v,\nn)$  its Frenet frame and $\kappa$ its curvature.  The planar filament equation is the evolution equation
$$
\Gamma' = \frac{\kappa^2}{2} \v + \dot\kappa \nn,
$$
obtained from the vector field $\x_2$ in the filament hierarchy \eqref{fields} by restricting to planar curves with $\tau=0$.
The planar filament equation has  infinitely many integrals of motion, namely,  the odd-numbered ones in the sequence (\ref{integrals}), restricting again to  $\tau=0$, see \cite{LP1}. 
The planar filament equation is equivalent to the modified Korteweg-de Vries equation, in the same way as the filament equation is equivalent to the non-linear Schr\"odinger equation \cite{LP1}.

By solitons of the planar filament equation we mean the curves that evolve under this flow by isometries and a parameter shift. The next proposition is not new (see, e.g., \cite{La}); we include its proof here for completeness. 

\begin{prop} \label{Sol}
The buckled rings are solitons of the planar filament equation.
\end{prop}

\begin{proof}

%\note{Referee 2 complains about  this proposition... Sergei?. -G.

%I don't mind removing it, just saying that it is known and refer ro \cite{La}. -S.

%I think yr proof is  really nice and simple, with no need to refer to anything (I also had a hard time finding the proof in \cite{La}). OK, maybe it was a bit too brief. So I edited  yr proof, adding 2  lines of calculation and an introductory paragraph. See what you think. -G. }
%
A planar curve evolves by isometries if and only if its curvature remains unchanged. Let $\Gamma(t)$ be an arc length parameterized planar curve and $\db(t)$ a vector field  along $\Gamma$,  defining its variation. Denote by $\kappa'$ the directional derivative of $\kappa$ with respect to $\db$. The change in $\kappa$ due to parameter shift is $\dot\kappa:={\d\kappa\over \d t}$. It follows that the soliton condition, requiring   $\Gamma$ to evolve under $\db$ by isometries and parameter shift, is equivalent to the condition that $\kappa$ satisfies the equation  $\kappa'+\lambda \dot\kappa=0$ for some real constant $\lambda$. 

Now a  straightforward calculation  shows that, for a  general $\Gamma$ and $\db$,  
\be\label{eq:dk}
\kappa' = \ddot \db \cdot \nn - 2\kappa\, \dot \db\cdot \v.
\ee
(Sketch: calculate  $\ddot\db=\ddot\Gamma'\equiv (2u'\kappa + \kappa')\nn\ (\mod \v)$, where $u=|\dot\Gamma|$, so that $\dot\Gamma=u\v.$ It follows that $\ddot\db\cdot \nn=2u'\kappa + \kappa'.$ Next, calculate
$\dot\db=\dot\Gamma'\equiv u'\v \ (\mod \nn)$,  hence $u'=\dot\db\cdot\v$, from which equation \eqref{eq:dk} follows.) 

In our case, $\db = \frac{\kappa^2}{2} \v + \dot\kappa \nn$, which implies, again by a straightforward    calculation, that 
$$ \ddot \db \cdot \nn=\frac{\d}{\d t} \left(\ddot \kappa +\frac{1}{2} \kappa^3\right), \  \dot \db\cdot \v=0.
$$
(The last equation  means that the flow defined by  $\db$ on the space of parametrized curves is arc length preserving.) 

It follows from the last two displayed equations  that  the soliton condition on  $\Gamma$ for the planar filament equation is 
$$
\kappa'+\lambda\dot\kappa=\frac{\d}{\d t} \left(\ddot \kappa + \frac{1}{2} \kappa^3 + \lambda \kappa \right) =0,
$$
or
$$
\ddot \kappa + \frac{1}{2} \kappa^3 + \lambda \kappa = \mu,
$$
that is, the Euler-Lagrange equation (\ref{EL}).
\end{proof}

We conclude that Wegner's curves are solitons of the planar filament equation.

%%%%

\section{Case study: multiple circles }\label{sec:mc}
%
%\note{rephrased a bit this sentence, after a comment of referee 2.-G.}
%
In this section we study some interesting and non-trivial curves in  bicycle correspondence
 with a circle. 
\subsection{Definition of the curves $\Gamma_{k,n}$}

Denote by $nS^1$ the $n$-fold circle, parameterized by $t\mapsto e^{it}\in\C,$ $0\leq t\leq 2\pi n$. 
Recall that the  monodromy of a closed parametrized planar curve is a conjugacy class of
an element of $\PSLt$; these elements are divided into hyperbolic, parabolic, elliptic and trivial, 
according to the number of  fixed points in $S^1\simeq \RP^1$ ($2,1,0, \infty$, respectively).

\begin{prop}\label{prop:bmono}   Let $\ell>0$. Then the  $\ell$-bicycle monodromy of $nS^1$, $n\geq 1$,  is 
\begin{itemize}
\item hyperbolic for $\ell<1$;

\item parabolic for $\ell=1$;

\item elliptic for $\ell>1$, except for the $n-1$ values  $$\ell_{k,n}=1/\sqrt{1-(k/n)^2},\quad  k=1,2,\ldots,n-1,$$   for  which the monodromy is trivial.

\end{itemize}
 \end{prop}

\begin{proof}The $\ell$-monodromy of $nS^1$ is the $n$-th power of the $\ell$-monodromy of the (simple) circle $S^1$. The latter is easily found by direct calculation to be: hyperbolic for $0<\ell<1$, 
parabolic for $\ell=1$,
and elliptic for $\ell>1$. The $n$-th power of hyperbolic or parabolic element is hyperbolic or parabolic, respectively. 

An elliptic element is conjugate to a rotation by some angle $\alpha$, hence its $n$-th power is trivial if and only if $\alpha$ is a multiple of $2\pi/n$.  On the other hand, by Corollary \ref{Berry}, $\alpha$ is the solid angle at the vertex of a  right cone over $S^1$ with generator of length $\ell$. A simple calculation shows that this solid angle is $2\pi(1-\sqrt{1-1/\ell^2}).$ Thus the $\ell$-monodromy of $nS^1$ is trivial if and only if  $2\pi(1-\sqrt{1-1/\ell^2})=2\pi m/n$ for some $m\in\Z$, or $\ell=1/\sqrt{1-(k/n)^2},$ $  k=1,2,\ldots,n-1.$
 \end{proof}

\begin{definition} For each $n\geq 2$ and $1\leq k\leq n-1$, let
$\Gamma_{k,n}(t)$ be the  unique closed plane curve   in $2\ell_{k,n}$-bicycle correspondence with $nS^1$, such that  $\Gamma_{k,n}(0)=1+2\ell_{k,n}\in\C$, where $ \ell_{k,n}=1/\sqrt{1-(k/n)^2}$.  \end{definition}

See Figure \ref{fig:zind}. This class of curves is mentioned in Section 8.4 of \cite{We5}. 
%
%\note{this figure was modified a bit, after a comment of referee 2. -G.}
%

\begin{figure}\centering\label{fig:Gam}
\includegraphics[width=\textwidth]{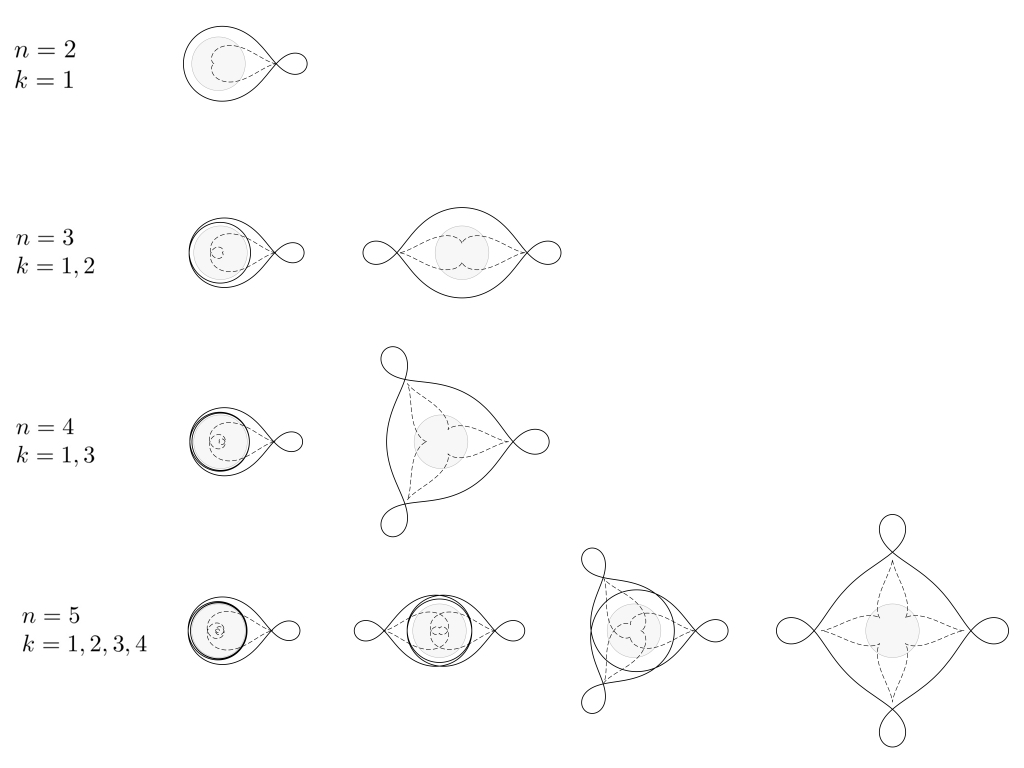}
\caption{ 
The curves $\Gamma_{k,n}$  (the solid curves) for $1\leq k<n\leq 5$  and  $(k,n)$ relatively prime. Each one is in bicycle correspondence with the unit circle (the boundary of the grey disk); their common back track is the dotted curve.
}\label{fig:zind}
\end{figure}

\begin{rmrk}\label{rmrk:sym}
%
%\note{Referee 2 want to write  here ``{\bf the}  $2\ell_{k,n}$-bicycle correspondence". I disagree. -G.}
%
Since the $\ell_{k,n}$-monodromy of $nS^1$ is trivial, we have in fact a whole circle worth of planar closed curves in $2\ell_{k,n}$-bicycle correspondence with $nS^1$. 
By the obvious  rotational symmetry of the bicycle equation for $nS^1$, they are all obtained from $\Gamma_{k,n}$ by rotation  about the origin and shift reparametrization.
\end{rmrk}

We next find an explicit arclength parametrization  of $\Gamma_{k,n}$  by  solving  the bicycle equation for $S^1$. 

\begin{prop}\label{prop:kn}
One has
\be\label{eq:gkn}
\Gamma_{k,n}(t)=e^{it}\left[1+2\ell e^{i\phi(t)}\right], 
\ee
where $\phi(t)$ is defined (as a continuous function) by
%
%\note{Added $\phi(0)=0$ to eqn \eqref{eq:ab} and some light editing to the text of the proof. -G.} 
%
\be\label{eq:ab}
\tan{\phi\over 2}=-a\tan{bt\over 2}, \quad a={n\over k}+\sqrt{\left({n\over k}\right)^2-1}, \quad b={k\over n}, \quad \phi(0)=0.
\ee
\end{prop}
\begin{proof} Let $\Gamma$ be the unit circle in $\R^2=\C$, parameterized by $\Gamma(t)=e^{it}.$ 
Let $\Gamma_\ell$ be the (not necessarily closed) parameterized plane  curve in $2\ell$-bicycle correspondence~with $\Gamma$, satisfying $\Gamma_\ell(0)=1+2\ell$. By definition, $\Gamma_\ell(t)=\Gamma(t)+2\ell\r(t)$, where $\r(t)$ is the solution to equation \bode\ with $\r(0)=1\in\C.$ Taking $\Gamma(t)=e^{it}$, $\r=e^{i\theta}$  in equation \bode\  gives
$\ell\dot\theta=-\cos(\theta-t).$ Changing to $\phi:=\theta-t$ gives $\Gamma_\ell(t)=e^{it}\left[1+2\ell e^{i\phi(t)}\right],
$ where $\phi(t)$ satisfies
${\dot\phi}=-1-(\cos\phi)/\ell$, $ \phi(0)=0.$ Changing  again to  $p= \tan(\phi/ 2),$ gives 
$$ \dot p=-{1\over 2\ell}\left[p^2(\ell-1)+\ell+1\right], \quad p(0)=0,$$ a constant coefficient Riccati equation, whose solution, for $\ell>1$,  is  
\be\label{eq:kn} 
 p=
 -a\tan{bt\over 2}, \quad a=\sqrt{{ \ell+1\over\ell-1}}, \quad b= \sqrt{1-{1\over \ell^2}}, \quad \ell>1.\ee
For  $\ell=\ell_{k,n}=1/\sqrt{1-(k/n)^2}$, we obtain the stated formulas. 
\end{proof}
\begin{rmrk}There are expressions similar  to \eqref{eq:kn}  for $\Gamma_\ell$ with $0<\ell\leq1$ (which we do not really use).  For $0<\ell<1$,
$$
p= -a\tanh{bt\over 2},\;  a=\sqrt{{ 1+\ell\over1-\ell}},  \;  b= \sqrt{{1\over \ell^2}-1}.$$
 For $\ell=1$, we have $\dot p=-1$ so that $p(t)=-t$. 
 \end{rmrk}

If $(k,n)$ have a common divisor $d>1$, then $\Gamma_{k,n}$ is a $d$-fold cover of $\Gamma_{\bar k, \bar n}$, where $\bar k=k/d, \bar n=n/d$, so all properties of  $\Gamma_{k,n}$  can be easily deduced from those of $\Gamma_{\bar k, \bar n}$. We will thus restrict attention henceforth  to $\Gamma_{k,n}$ with relatively prime  $(k,n)$.

\begin{cor}\label{cor:sym} For all $n\geq 2$ and $1\leq k\leq n-1$,
\begin{enumerate}
\item $\Gamma_{k,n}$ and $nS^1$ have the same length $2\pi n$ and the same $\lambda$-bicycle monodromy for all  $\lambda>0$ (as in   Proposition \ref{prop:bmono}). 
\item $\Gamma_{k,n}$ admits a $D_k$-symmetry (the symmetries of a regular $k$-gon); that is, 
\begin{align}\label{eq:sym}
\begin{split}
&\Gamma_{k,n}(t+2\pi n/k)=e^{2\pi in/k}\Gamma_{k,n}(t), \\
& \Gamma_{k,n}(-t)=\overline{\Gamma}_{k,n}(t),
\end{split}
\end{align}
for all $t\in\R$.
\end{enumerate}
\end{cor}
\begin{proof} The first statement  follows from Corollary \ref{cor:arc} and  Theorem \ref{thm:bc} and the second 
from equations \eqref{eq:gkn} and \eqref{eq:ab}.
\end{proof}

\subsection{$\Gamma_{k,n}$ as  Zindler curves}
Let $\Gamma(t)$ be an arc length parameterized closed immersed Zindler curve of length $L$. That is,  there is a number $\rho\in(0,1)$, called a  {\it rotation number} of $\Gamma$, such that the chord length $\|\Gamma(t+\rho L)-\Gamma(t)\|$, corresponding to a fixed length $\rho L$ of an arc,  is a positive constant and the velocity of the midpoint $(\Gamma(t+\rho L)+\Gamma(t))/2$  is  parallel to 
$\Gamma(t+\rho L)-\Gamma(t)$. 
Note that $\rho$ is a rotation number if and only if so is $1-\rho$, and that a circle  is Zindler for all rotation numbers $\rho\in(0,1)$. 

In this subsection we show that most  $\Gamma_{k,n}$  are Zindler and determine their associated  rotation numbers (Theorem \ref{thm:rot}). 
The proof, although elementary, is somewhat technical, so we give here the main idea.

%
%\note{edited here lightly, as suggest by referee 2. -G}
%
By construction, each $\Gamma_{k,n}$ is in bicycle correspondence with $nS^1$, 
hence its $\lambda$-bicycle monodromy is hyperbolic for $0<\lambda<1$ and parabolic for  $\lambda= 1$. 
It follows that for all $\lambda\leq 1$ there is a planar  closed curve $\Gamma^\lambda_{k,n}$ in 
$2\lambda$-bicycle correspondence with $\Gamma_{k,n}$ (unique up to reflection about one of the symmetry axes of $\Gamma_{k,n}$). By Bianchi permutability and 
 circular symmetry, each  $\Gamma^\lambda_{k,n}$ is a rigid    
 rotation of $\Gamma_{k,n}$ about the origin by some $\lambda$-dependent angle $\chi$. By the $D_k$-symmetry of 
 $\Gamma_{k,n}$ (Corollary \ref{cor:sym}),  if  $\chi$ is  a multiple of $2\pi/k$ then 
$\Gamma^\lambda_{k,n}$ coincides with $\Gamma_{k,n}$ (up to  shift reparametrization), 
so that $\Gamma_{k,n}$ is Zindler.  

Thus the proof of the Zindler property of $\Gamma_{k,n}$ and the calculation of the associated rotation numbers reduces to the calculation of $\chi$ as a function of $\lambda\in(0,1]$. This is done in Lemma \ref{lemma:rot} and is the main ingredient in the proof of Theorem \ref {thm:rot}. (In fact, it is  more convenient to write $\lambda=\sin(\gamma/2)$ and calculate $\chi$ as a function of $\gamma\in(0,\pi]$).

\begin{theorem} \label{thm:rot} 
Let $(k,n)$ be a relatively prime pair  of positive integers. Then   $\Gamma_{k,n}$  is a  Zindler curve if and only if $1\leq k\leq n-2$, with    $n-k-1$ rotation numbers 
$\rho\in(0,1)$, given by the equation
\be\label{eq:rot3}
n \tan(k\pi\rho)=k \tan(n\pi\rho).
\ee
(including $\rho =1/2$ when both $n$ and $k$ are odd). 
\end{theorem}

\begin{proof}
For each $\ell>0$, $\gamma\in(0,\pi]$  and $t\in\R,$ define   
\begin{itemize}
\item $\Gamma(t)=e^{it}$, 
\item  $\Gamma^\lambda(t)=e^{i(t+\gamma)}$, where $\lambda=\sin(\gamma/2)\in(0,1]$, 
\item $\Gamma_\ell(t)$ -- the (not necessarily closed)  
$2\ell$-bicycle transform of $\Gamma(t)$ with $\Gamma_\ell(0)=1+2\ell\in \C$,    
\item $\Gamma^\lambda_\ell(t)$ -- the completion of $\Gamma^\lambda(t)\Gamma(t)\Gamma_\ell(t)$  to the Darboux Butterfly\\
$\Gamma^\lambda(t)\Gamma(t)\Gamma_\ell(t)\Gamma^\lambda_\ell(t)$. 
\end{itemize}

\begin{figure} \centering
{\includegraphics[width=0.5\textwidth]{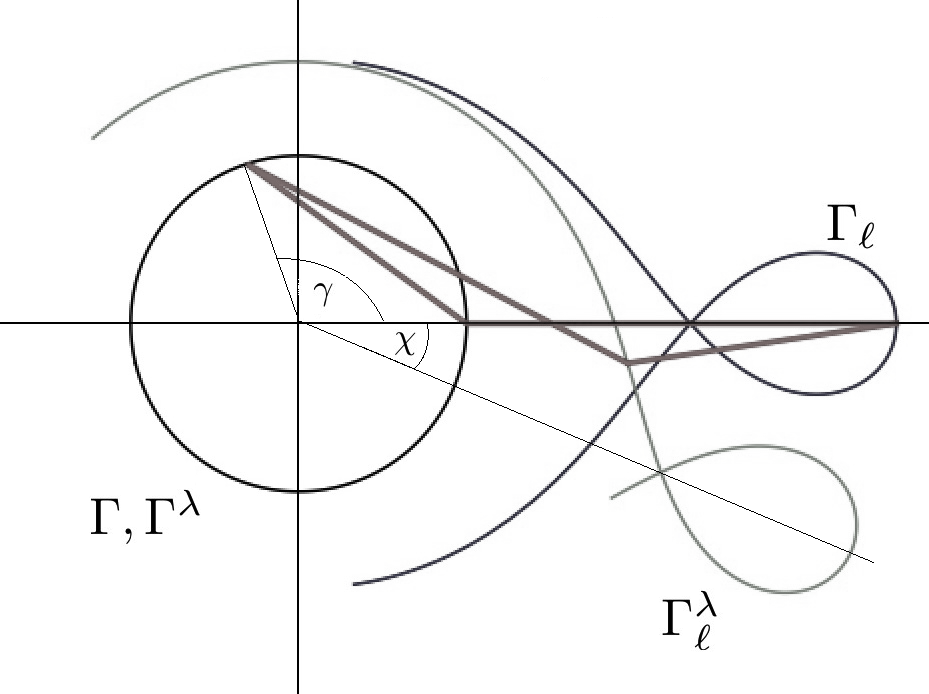}}
\caption{The notation  of the proof of Theorem \ref{thm:rot}}\label{fig:not}
\end{figure}

Note that $\Gamma^\lambda$ is in $2\lambda$-bicycle correspondence with $\Gamma$ and so, by Bianchi permutability  (Proposition \ref{prop:bp}),  $\Gamma^\lambda_\ell$ is in $2\lambda$-bicycle correspondence with $\Gamma_\ell$, as well as $2\ell$-bicycle correspondence with $\Gamma^\lambda.$
It follows (see Remark \ref{rmrk:sym}) that there exist $\chi,\tau\geq 0 $ such that
\be\label{eq:rot}\Gamma^\lambda_\ell(t)=e^{-i\chi}\Gamma_\ell(t+\tau)
\ee
for all $t$. See Figure \ref{fig:not}. 
%
%\note{Referee 2 thinks that if we just want to show that $\Gamma_{k,n}$ are Zindler, without figuring out the rotation numbers, then this lemma can be weakened, showing only that $\xi$ is monotonous in $\gamma$, making the proof of the theorem easier. I doubt it. -G.}
%
\begin{lemma}\label{lemma:rot}
For $\ell>1$,  $\chi(\gamma)$ and $\tau(\gamma)$ in equation \eqref{eq:rot} satisfy 
$$\tan \left({b\tau\over 2}\right)=b\tan{\gamma\over 2},\quad \chi= \tau-\gamma,\quad b= \sqrt{1-{1\over\ell^2}}.$$
It follows that  $\chi$ is a monotonically increasing function of $\gamma$,  varying from $0$ to $\pi(1/ b-1)$, as $\gamma$ varies from $0$ to $\pi$. 
\end{lemma}

\begin{proof} By definition, we have 
$$\Gamma_\ell(t)=e^{it}(1+2\ell e^{i\phi(t)}),$$
where $\phi(t)$ is given by  equation \eqref{eq:kn}, and 
$$\Gamma_\ell^\lambda(t)=e^{it}(e^{i\gamma}+2\ell e^{i\psi(t)}),$$
for some function $\psi(t)$.  Substituting  the last two equations  in \eqref{eq:rot}, 
we  get
$$e^{i(\gamma+\chi)}+2\ell e^{i[\psi(t)+\chi]}=e^{i\tau}+2\ell e^{i[\phi(t+\tau)+\tau]}.$$
Both sides of the last equation are parameterized arcs of circles of radius $2\ell$, hence their centers and angular parameter must coincide, giving
$$\tau= \chi+\gamma, \quad\psi(t)= \phi(t+\tau)+\gamma. $$ 
Note that, strictly speaking, these equations hold only modulo $2\pi$, but by the continuous dependence of $\tau, \chi$, and $\psi$ on $\gamma$, and the initial conditions $\tau(0)=\chi(0)=\psi(0)=0$, they hold as stated.

\mn

This gives the stated second formula and  $\psi(0)= \phi(\tau)+\gamma$, hence
$$\tan{\psi(0)\over 2}=\tan\left({\phi(\tau)\over 2}+{\gamma\over 2}\right)={\tan{\phi(\tau)\over 2}+\tan{\gamma\over 2}\over 1-\tan{\phi(\tau)\over 2}\tan{\gamma\over 2}}.$$
Using equation  \eqref{eq:kn}, we get from the last equation
\be\label{eq:rot1}\tan{\psi(0)\over 2}={-a\tan{b\tau\over 2}+\tan{\gamma\over 2}\over 1+a\tan{b\tau\over 2}\tan{\gamma\over 2}}, \qquad  a=\sqrt{{ 1+\ell\over1-\ell}}.
\ee

Next, we look  at the Darboux Butterfly $ABCD$, where $A=\Gamma(0)=1$, 
$B=\Gamma^\lambda(0)=e^{i\gamma}$, $C=\Gamma^\lambda_\ell(0)=e^{i\gamma}+2\ell e^{i\psi(0)}$
and $D=\Gamma_\ell(0)=1+2\ell$.

\begin{figure}[h!]
\centering
{\includegraphics[width=0.6\textwidth]{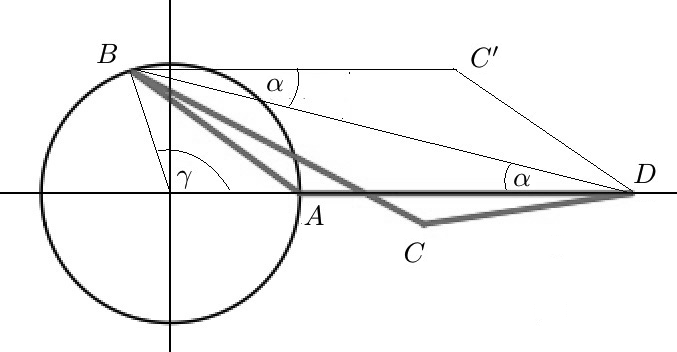}}
\caption{The Darboux Butterfly $ABCD$}
\end{figure}

 Let $C'$ be the  reflection of $C$ about $BD$. Then  $ABC'D$ is a parallelogram, in which 
$\alpha=\angle DBC'=\angle BDA=-\psi(0)/2$, 
$AD=2\ell$, $AB=2\lambda$  and $\angle ABD=\pi/ 2-\alpha-\gamma/2$. 
Applying  the sine law to $\triangle DAB$, we get
$${\sin(\pi/ 2-\alpha-\gamma/ 2)\over 2\ell}={\sin\alpha\over 2\lambda}.$$
Using $\lambda=\sin(\gamma/2)$ and solving the last equation for $\tan\alpha$, we get
$$-\tan{\psi(0)\over 2}=\tan\alpha={\tan(\gamma/2)\over \ell+(\ell+1)\tan^2(\gamma/ 2)}. $$
Substituting this into  the left hand side of \eqref{eq:rot1} and simplifying, we obtain the stated formula $\tan(b\tau/2)=b\tan(\gamma/2). $
It follows immediately from this formula that 
$$\chi'(\gamma)=\tau'(\gamma)-1={u^2\over \ell^2(1+ b^2u^2)}, 
$$
where $u=\tan^2(\gamma/2),$ hence  $\chi$ (as well as $\tau$) is monotonically increasing in $\gamma$, as stated.
\end{proof}

We continue with the proof of Theorem \ref{thm:rot}.
Setting   $\ell=\ell_{k,n}=1/\sqrt{1-(k/n)^2}$ in Lemma \ref{lemma:rot},  we have   $b=k/n$, so  that
\be \label{eq:tan}n\tan{k\over n}{\tau\over 2}=k\tan{\gamma\over 2}, 
\ee
and $\chi=\tau-\gamma$ increases monotonically from $0$ to 
$\pi(n-k)/k$,  as $\gamma$ varies from $0$ to $\pi$. 

Now the Zindler condition on $\Gamma_{k,n}$ is  that it is in the bicycle correspondence with itself (up to shift reparametrization). Due to the $D_k$-symmetry  (see Corollary \ref{cor:sym}), this means that $\chi$ in equation \eqref{eq:rot} should be  an integer  multiple of $2\pi/ k$.  As $\gamma$ varies from  $0$ to $\pi$, $\chi$ increases monotonically from $0$ to 
$\pi(n-k)/k$, so  there are exactly $[(n-k)/2]$ values  $\gamma\in(0,\pi]$ for which  $\chi$ is an integer  multiple of  $2\pi/ k$. For each  such $\gamma$,   $\Gamma_{k,n}$ is Zindler with chord length $2\lambda=2\sin(\gamma/2)$, giving rise to a pair of rotation numbers $\{\rho,1-\rho\}\subset(0,1)$, except if $\rho=1/2$, which occurs if and only if  $\gamma=\pi$ and  $n-k$ is even (i.e., $k,n$ are both   odd). It follows that $\Gamma_{k,n}$ is Zindler if and only if $1\leq k\leq n-2$, with a total of $n-k-1$ rotation numbers $\rho\in(0,1),$ as stated.

To determine the associated rotation numbers, let $\gamma$ be a   value for which 
\be\label{eq:chi}
\chi=2\pi m/k, \quad m=1,2,\ldots,[(n-k)/2].
\ee  
An associated rotation number $\rho\in(0,1)$ satisfies 
\be\label{eq:rho}
\Gamma^\lambda_{k,n}(t)=\Gamma_{k,n}(t+2\pi n\rho)
\ee 
for all $t$. Since $k,n$ are relatively prime, there exists an integer $\bar n$ such that $n\bar n\equiv 1$ $(\mod k)$. Let $\bar m:=m\bar n$, then 
\be\label{eq:mbar}
{2\pi m\over k}\equiv {2\pi \bar mn\over k}\quad (\mod 2\pi).
\ee 

Now we  calculate:
\begin{align}
\label{eq:sys}
\Gamma^\lambda_{k,n}(t)&=e^{-i{2\pi m\over k}}\Gamma_{k,n}(t+\tau) 
& &\mbox{by  \eqref{eq:rot},\eqref{eq:chi}}\nonumber \\
&=e^{-i{2\pi \bar m n\over k}}\Gamma_{k,n}(t+\tau) 
& &\mbox{by  \eqref{eq:mbar}} \nonumber \\
&=\Gamma_{k,n}(t+\tau- {2\pi \bar m n\over k})
& &\mbox{by \eqref{eq:sym}} \\
&=\Gamma_{k,n}(t+2\pi n\rho)
& &\mbox{by  \eqref{eq:rho}}\nonumber
\end{align}
for all $t$. The last equality  is equivalent to 
$\tau-2\pi \bar mn/k\equiv 2\pi n\rho  \; (\mod 2\pi n),$ implying 
\be \label{eq:tau} {k\over n}{\tau\over 2}\equiv\pi k\rho \quad (\mod \pi ).
\ee

Next, applying $2\ell_{k,n}$-bicycling correspondence to equation \eqref{eq:rho}, gives  
$\Gamma^\lambda(t)=\Gamma(t+2\pi n\rho)$, or 
$e^{i(t+\gamma)}=e^{i(t+2\pi n\rho)}.$ It follows that  $\gamma\equiv 2\pi n\rho \;(\mod 2\pi ),$ or 
\be\label{eq:gam1}
{\gamma\over 2}\equiv  \pi n\rho \quad (\mod \pi ).
\ee
Substituting  equations  \eqref{eq:tau} and \eqref{eq:gam1} in equation \eqref{eq:tan}, we see that $\rho$ satisfies equation \eqref{eq:rot3}.

%%%%%%%%%%%%%%%%%%%
%%%%%%%%%%%%%%%%%%%

We have shown so far that $\Gamma_{k,n}$ has $n-k-1$ rotation numbers and that they all satisfy equation 
\eqref{eq:rot3}. To complete the proof of Theorem \ref{thm:rot}   it is thus sufficient to show that equation 
\eqref{eq:rot3} has exactly $n-k-1$ solutions $\rho\in(0,1)$ (including $\rho=1/2$, when $k,n$ are both odd). 

\begin{lemma}\label{lemma:rot1}
%
%\note{Added the phrase ``including $\rho=1/2$ when\ldots etc".-G.}
For any pair of integers $k,n$ with $1\leq k<n$, equation \eqref{eq:rot3} 
$$n \tan(k\pi\rho)=k \tan(n\pi\rho)
$$
has $n-k-1$ solutions in $(0,1)$, including $\rho =1/2$ when both $n$ and $k$ are odd. 
\end{lemma}
\begin{proof}Clearly,  $\rho$ is   a solution if and only if 
$$\left({1\over n\pi}\right)\arctan\left({n\over k}\tan(k\pi\rho)\right)\equiv \rho\quad \left(\mod {1\over n}\right).$$
For each $j\in\Z$, let 
$I_j=(j/k-1/2k, j/k+1/2k)\subset\R$ (the open interval of length  $1/k$ centered at $j/k$) and define  $f_j: I_j\to\R$ by 
$$f_j(\rho):=\left({1\over n\pi}\right)\arctan\left({n\over k}\tan(k\pi\rho)\right)+{j\over n}, \quad \rho\in I_j, \quad j\in \Z.$$
One can  verify that $f_j$ extends smoothly to $\overline I_j$ and that the extensions at adjacent intervals coincide at the shared endpoints, combining to define a smooth, strictly increasing function $f:\R\to\R$, satisfying
\begin{enumerate}[(i)]\setlength\itemsep{-4pt}
\item  $f(0)=0$, 
\item  $f(1)=k/n,$ and
\item  $0<f'(\rho)\leq 1$ for all $\rho\in\R,$  with $f'(\rho)=1$ only at isolated points. (In fact, $f'(\rho)=1$ precisely  at $\rho\equiv 0 \; (\mod 1/k)$.)
\end{enumerate}
See   Figure \ref{fig:f}.
\begin{figure}[h!] 
\centering
{\includegraphics[width=0.7\textwidth]{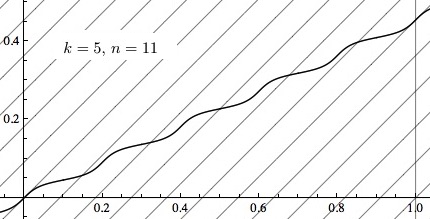}}
\caption{The function $f(\rho)$ of the proof of Lemma \ref{lemma:rot1}}\label{fig:f}
\end{figure}
By construction,  the solutions of equation 
\eqref{eq:rot3} are  given by 
$$f(\rho)\equiv \rho\quad\left(\mod {1\over n}\right).$$  It now follows easily from the above 3 properties of $f$ that this equation has exactly $n-k-1$ solutions in the interval $0<\rho<1$. 
%%%%%%%%%%%%%%%%%%%
%%%%%%%%%%%%%%%%%%%
\end{proof}
This concludes the proof of Theorem \ref{thm:rot}. \end{proof}
\begin{rmrk}  
\begin{enumerate}[(i)]
 \item  Here is a table of the (approximate) rotation numbers $\rho\in(0,1/2]$ of $\Gamma_{k,n}$, for relatively prime pairs $(k,n)$, with  $1\leq k\leq n-2$ and $n\leq 7$.

$$
\begin{array}{|l|l|l|}
\hline
k&n&\rho\\
\hline
1&3&0.5\\
1&4& 0.37  \\
1&5& 0.29,  0.5\\
2&5& 0.31\\
3&5&   0.5\\
1&6& 0.24, 0.41\\
\hline
\end{array}
\qquad
\begin{array}{|l|l|l|}
\hline
k&n&\rho\\
\hline
1&7& 0.21,0.35,0.5\\
2&7& 0.21,0.37\\
3&7& 0.23,0.5\\
4&7& 0.35\\
5&7& 0.5\\
\hline
\end{array}
$$

\item Here is also a plot of all rotation  numbers for $n=11$. 

\centerline{\includegraphics[width=0.6\textwidth]{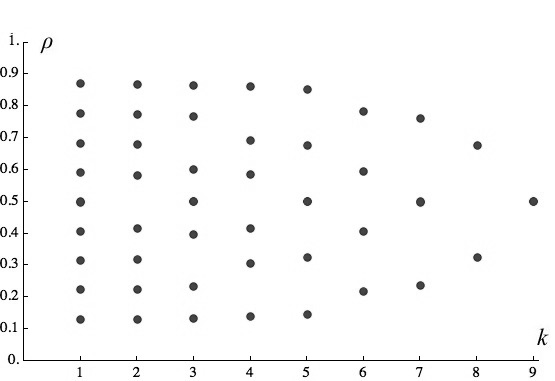}}

\item   For  small values of $(k,n)$, the numbers  $\tan^2(\pi\rho)$ are roots  of linear or quadratic polynomials, obtained from multiple angle trigonometric  identities.   Here are all such  cases with $\rho\neq 1/2$: 
$$\begin{array}{|l|l|l|}
\hline
k&n&\tan^2(\pi\rho)\\
\hline
1&4&5  \\
1&5& 5/3\\
2&5& 2
   \sqrt{21}-7\\
1&6& \left(21\pm 4
   \sqrt{21}\right)/3\\
1&7&  \left(7\pm 2
   \sqrt{7}\right)/3\\
\hline
\end{array}
$$

For higher values of $(k,n)$, the numbers $\tan^2(\pi\rho)$ are roots of higher degree polynomials.

\item For $k=1$, equation \eqref{eq:rot3} is 
$ n\tan(\pi \rho)=\tan(n\pi\rho)
$. This is  equation (14) of Theorem 7 of \cite{T1}, describing the rotation numbers $\rho$ for which the (simple) unit  circle can be infinitesimally deformed  in the class of planar Zindler curves with rotation number $\rho$. 

The same equation appeared in a study of billiards and of a flotation problem \cite{Gu1,Gu2}. See \cite{Cy} for number theoretic properties of its solutions. A discrete version of the equation is proposed in \cite{T1}; see \cite{Cs,CC} for its solutions.

\end{enumerate}
\end{rmrk}

\subsection{Spherical curves in bicycle correspondence with $\Gamma_{k,n}$}

As we have seen in the previous subsection, all curves in $2\lambda$-bicycle correspondence with $\Gamma_{k,n}$, for $\lambda\leq 1$,  are rotations of $\Gamma_{k,n}$ about the origin. For a generic value of $\lambda>1$ (i.e., for $\lambda\neq \ell_{k',n}$, $1\leq k'<n$), the $\lambda$-monodromy of $\Gamma_{k,n}$ is elliptic, and thus there are two space curves in $2\lambda$-bicycle correspondence with $\Gamma_{k,n}$,  related by reflection about the $xy$-plane. 

\begin{figure}[h!]
\centerline{\includegraphics[height=120pt]{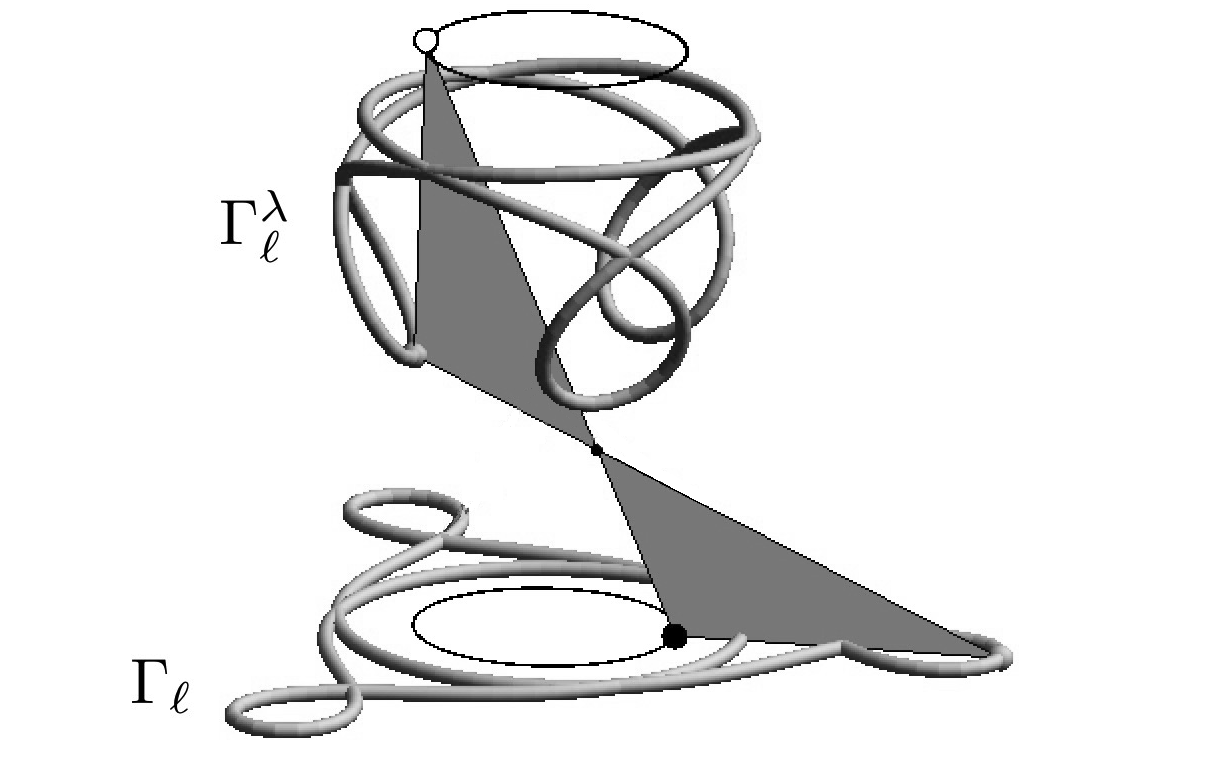}}
\caption{The spherical curve $\Gamma^\lambda_\ell$} \label{3DD}
\end{figure}

\begin{prop} \label{spherical}
Let  $\lambda> 1$ and let $\Gamma^\lambda_{k,n}$ be a curve in $\R^3$ in $2\lambda$-bicycle correspondence with $\Gamma_{k,n}$. Then $\Gamma^\lambda_{k,n}$ is either planar, contained in the $xy$ plane, in which case $\lambda=\ell_{k',n}$ for some $1\leq k'<n$, or it is a spherical curve, with the center of the sphere on the $z$ axis. 
\end{prop}

\begin{proof}  
We define for every $\ell>0$, $\lambda>1$ and $t\in\R$:
\begin{itemize}
\item $\Gamma(t)=(e^{it},0)\in\C\oplus\R=\R^3$ (lower thin circle in Figure \ref{3DD}), 
\item $\Gamma^\lambda(t)=(e^{-it},2\sqrt{\lambda^2-1})$ (upper thin circle),
\item $\Gamma_\ell(t)$ -- the (not necessarily closed) curve in the $xy$-plane in $2\ell$-bicycle correspondence with $\Gamma$, such that $\Gamma_\ell(0)=(1+2\ell,0,0)\in\R^3$  (lower thick planar curve),
\item $\Gamma^\lambda_\ell(t)$ --  the completion of $\Gamma^\lambda(t)\Gamma(t)\Gamma_\ell(t)$  to the Darboux Butterfly\\
$\Gamma^\lambda(t)\Gamma(t)\Gamma_\ell(t)\Gamma^\lambda_\ell(t)$ (upper thick space curve).  

\end{itemize} 

Note that $\Gamma$ is in $2\lambda$-bicycle correspondence with $\Gamma^\lambda$ and in $2\ell$-bicycle correspondence with $\Gamma_\ell$, and hence, by Bianchi permutability, $\Gamma^\lambda_\ell$ is in $2\lambda$-bicycle correspondence with $\Gamma_\ell$ and  in $2\ell$-bicycle correspondence with $\Gamma^\lambda$. Since 
$\Gamma^\lambda$ is related to $\Gamma$ by Euclidean translation along the $z$-axis and reparametrization, it is enough to  show that   any  curve (not necessarily closed) 
in $2\ell$-bicycle correspondence  with $\Gamma$ lies either in the $xy$ plane or on some sphere centered on the $z$-axis.

At this junction, we can explicitly solve the bicycle equation, as described in Section \ref{B3D}: this is an equation with constant coefficients.  We  present a more geometrical argument here. 

The desired property is invariant with respect to  rotations about the $z$-axis, so it can be shown  in  a  frame rotating about the $z$-axis with angular velocity 1. 

\begin{figure}[h!]
\centerline{\includegraphics[width=0.7\textwidth]{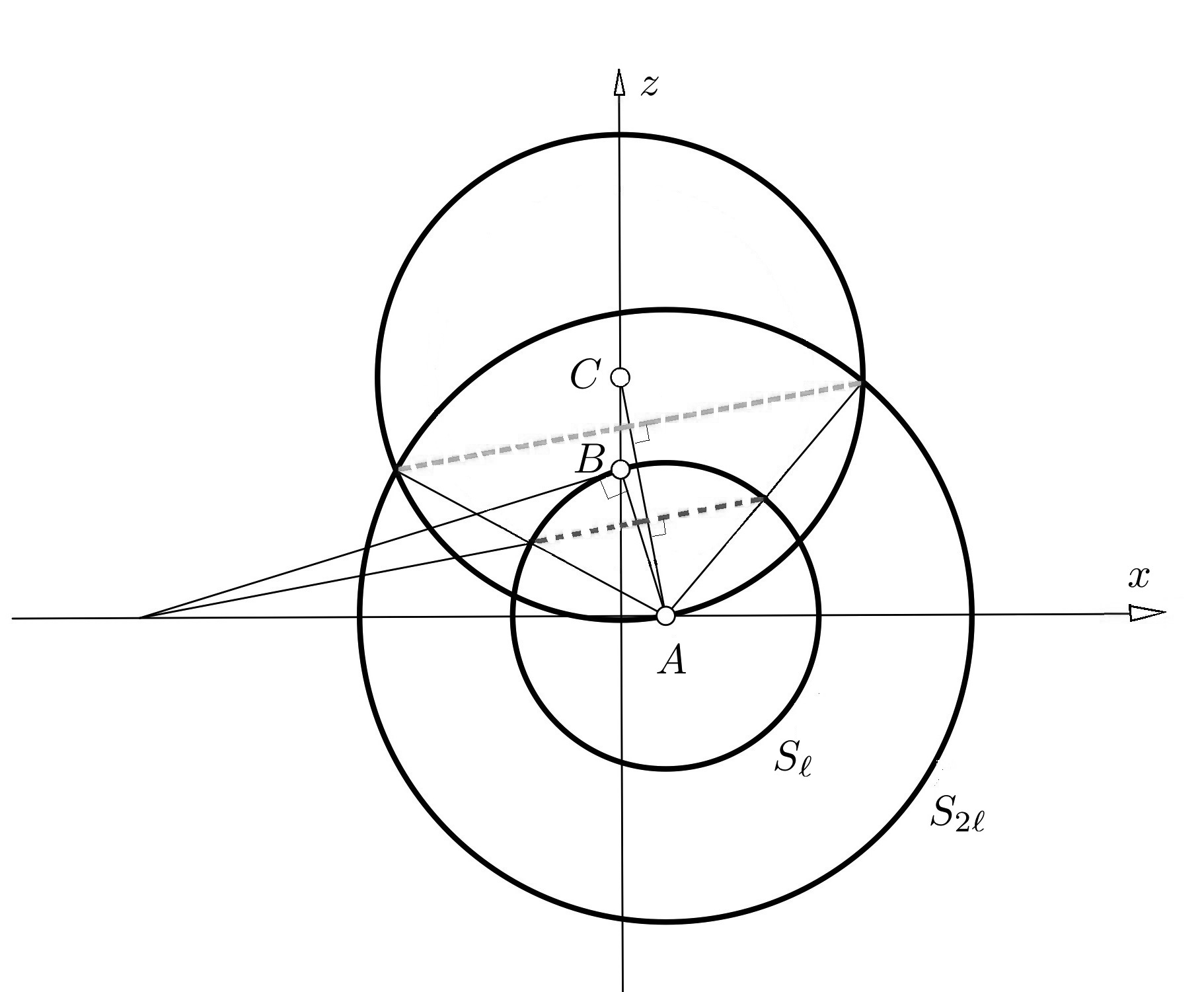}}
\caption{The proof of Proposition \ref{spherical}}\label{fig:3dz}
\end{figure}

In this frame, the front wheel $\Gamma(t)$ is stationary, say at $A=(1,0,0)$,  
so the rear wheel traces some curve on  
the 2-sphere $S_\ell$ of radius $\ell$ centered at $A$.
We claim that this curve is a circle  (shown as a dotted chord of  $S_\ell$ in Figure \ref{fig:3dz}), whose axis (the line through its center, perpendicular 
 to the plane of the circle) intersects the $z$ axis at some point $C$, or else is parallel to the $z$-axis, in which case  the curve is the equatorial circle (the intersection of $S_\ell$ with the $xy$-plane). 
 
Indeed, the bicycle equation \bode\ in a rotating frame (the Frenet-Serret frame)  is autonomous, i.e.,~defines a time-independent conformal vector field on $S_\ell$, whose flow is a 1-parameter elliptic subgroup of the M\"obius group of  $S_\ell$  (here we use the fact that $\ell>1$). Its trajectories are planar circles,  with two fixed points,  the  vertices of the  two right  circular cones over $\Gamma$ 
 with generator of length $\ell$; one of them, $B$, is shown in Figure \ref{fig:3dz}. 
 
 It follows that the  curve in $2\ell$-bicycle correspondence with $\Gamma$,  
 generated by this rear track (shown as a dotted chord of $S_{2\ell}$ in Figure \ref{fig:3dz}), is a circle on the sphere $S_{2\ell}$ of radius $2\ell$, centered at $A$, as well as a circle on the sphere centered at $C$ and passing through $A$.
  \end{proof}

\appendix

%\note{deleted "images". -ML}

\section{Bicycle correspondence as a Darboux transformation}\label{sec:STP}

In this appendix  we relate the  bicycle correspondence for curves in $\R^3$  with    Darboux transformations of a certain spectral problem. 
%\note{deleted ", called the AKNS system (after Ablowitz,  Kaup,  Newell, and Segur)."} 

We use  the STP  construction (after Sym, Tafel, and Pohlmeyer), associating  with each solution of the AKNS system a family of  curves in $\sutwo\cong\R^3$. We note that only Theorem \ref{thm:new} of this section is new;  the preliminary material, as well as related motivation and details, can be found  in \cite{RS}.  
The curves that we are dealing with in this subsection are not necessarily closed.

We begin with a description of   the {\it AKNS system}, following  \cite{AKNS}. Given a complex-valued function $q(t)$ of a real variable $t$,  we define  the  linear system %
\be\label{eq:AKNS}
\Phi_t = (Q + i \lambda A) \Phi, 
\ee
where  $\Phi= \Phi(t,\lambda)$ is a  complex $2\times 2$  matrix-valued  function of two real variables,    $\Phi_t=\partial\Phi/\partial t$ and 
$$
Q =
\left(
\begin{array}{cc}
 0& q \\
-\bar q& 0
\end{array}
\right),
\quad q=q(t), \quad
A=
\left(
\begin{array}{cc}
 \frac{1}{2}& 0 \\
0& -\frac{1}{2}
\end{array}
\right). 
 $$
The variable $\lambda$ is called the {\em spectral parameter}   and $Q$ is the {\em potential}. Observe that $Q+i\lambda A$ is ${\mathfrak{su}}_2$-valued, hence  if we assume, as we shall  do henceforth, that
$$\Phi(0, \lambda)\in\SU_2 \;\mbox{  \em for all }\lambda\in\R,$$
 then 
$\Phi(t,\lambda)\in\SU_2$ for all $(t,\lambda)$ as well. 

%\note{Ron: Can you insert here the  reference? -G.} 
%
Given a solution $\Phi(t,\lambda)$ to  \eqref{eq:AKNS}, we define,  following \cite{Sym1}, the associated STP curves  $\Gamma(t,\lambda)$ in $\sutwo$ by   
\be\label{eq:gam}\Gamma = 
\Phi^* \Phi_\lambda.
\ee

In what follows, we use the (slightly modified) standard Killing form on $\sutwo$, $\|X\|^2=-2\mathrm{tr}(X^2).$
\begin{prop}\label{prop:STP}
For each  $\lambda$, the map $t\mapsto \Gamma(t, \lambda)$ defines an arclength parameterized curve in $\sutwo$, 
i.e., $\|\Gamma_t\|=1$, with the curvature and torsion functions given in terms of $q(t)$ by 
\be\label{eq:STP}\kappa=2|q|, \quad \tau=\Im\left(q_t/ q\right)-\lambda.\ee

Conversely, given a curve $C$ in $\sutwo$ with curvature $\kappa$ and torsion $\tau$, the AKNS system associated with  
$$q={\kappa\over 2}e^{i\int\tau}$$
has a family of STP curves $\Gamma(\cdot, \lambda)$ with curvature $\kappa$ and torsion $\tau-\lambda$, so that $\Gamma(\cdot, 0)$ is  congruent to $C$. 
In fact, by adjusting the initial condition $\Phi(0,\lambda)$ in equation \eqref{eq:AKNS}, one can have $\Gamma(\cdot, 0)$ actually  coincide with $C$. 
\end{prop}
\begin{proof}
A simple calculation shows that $\Gamma=\Phi^*\Phi_\lambda$ implies $\Gamma_t = \Phi^* (i A) \Phi$.  Since the  Killing form is conjugation-invariant, $\| \Gamma_t\|^2=\|iA\|^2=1.$ 
Similarly, one finds that 
$\Gamma_{tt}= \Phi^* 
[iA,Q]\Phi$ and  
$$ \Gamma_{ttt}= \Phi^* 
\left([iA,Q_t]+[[iA,Q],Q+i\lambda A]\right)\Phi,$$
from which follows
$$\kappa=\|\Gamma_{tt}\|=\|[iA,Q]\|=2|q|$$ 
and 
$$\tau={\<[\Gamma_t,\Gamma_{tt}],\Gamma_{ttt}\>\over \kappa^2}=\Im\left( q_t/ q\right)-\lambda.
$$

Conversely, given a curve $C$ in $\sutwo$ with curvature and torsion functions 
$\kappa, \tau$, one can verify easily  that  $q=(\kappa/2)e^{i\int\tau}$ satisfies  equations \eqref{eq:STP} for $\lambda=0$, so that  the STP curve $\Gamma(\cdot, \lambda)$, associated with the AKNS system defined by $q$,  has curvature $\kappa$ and torsion $\tau- \lambda$. 

Finally, if we take a solution $\Phi(t, \lambda)$ to the AKNS system \eqref{eq:AKNS} and right-multiply it by $G(\lambda)\in\SU_2$, then we obtain another solution of \eqref{eq:AKNS}, whose STP curve is
 $G^*\Gamma G+ G^*G_\lambda.$ The first term is a rotation of $\Gamma$  and the second gives a translation, so that by choosing $G$ appropriately we can move $\Gamma(\cdot,0)$  onto $C$. 
\end{proof}
Next, we define the {\em Darboux transformations} of  the AKNS system \eqref{eq:AKNS}. To this end, we fix  a {\em non-real} complex  number $\mu$  and  a non-zero element  $v\in \C^2$ and use the data  $(\mu, v)$ to transform the  AKNS system \eqref{eq:AKNS} to a new system $\tilde\Phi_t = (\tilde Q + i \lambda A) \tilde\Phi, $ where $\tilde Q$ and $\tilde\Phi$ are given in terms of $Q, \Phi$ and  $(\mu, v)$,  as follows. 
Let  
$$\phio(t)=\left(\begin{matrix}\phi_1(t)\\ \phi_2(t)\end{matrix}\right)$$ be  the $\C^2$-valued function defined by 
$$\phio_t = (Q + i \mu A)\phio,   \quad  \phio(0) = v, $$
and the associated  projection operator
$$\pi =
{\phio \phio^*\over \|\phi\|^2}={1\over |\phi_1|^2+|\phi_2|^2}
\left(\begin{matrix}
\phi_1\bar\phi_1&\phi_1\bar\phi_2\\
\phi_2\bar\phi_1&\phi_2\bar\phi_2
\end{matrix}
\right). 
$$
Note that $\pi(t)$ is the orthogonal projection on the (complex) 1-dimensional subspace  of $\C^2$ spanned by $\phi(t)$,  hence it is unchanged if  $v$ is multiplied by a non-zero complex scalar.

Next, define the complex numbers
\begin{align*}
\alpha &= \sqrt{ \frac{\lambda -\bar \mu}{\lambda -\mu}}, \quad  \beta = \frac{\mu - \bar \mu}{\lambda - \bar \mu},
\end{align*}
and the  linear operators 
\be\label{eq:darb}
U=\alpha(I - \beta \pi),  \quad \tilde Q = Q + i(\mu-\bar \mu) [A, \pi], 
\ee
where    $[A,\pi]=A\pi-\pi A.$ 

%Finally, we introduce the transform of $\Phi$ and its associated AKNS system:
\begin{lemma}    

\mn 

\begin{enumerate}

\item $U(t, \lambda)\in\SU_2$. 

\item $\tilde Q(t)\in\sutwo$ is the potential associated with the complex function $$\tilde q=q+i(\mu-\bar \mu){\bar\phi_1\phi_2\over\|\phi\|^2}.$$

\end{enumerate}

\end{lemma}
\begin{proof} 1. $\pi$ is an orthogonal  projection operator, hence is conjugate, by some element in $\SU_2$,  to $diag(1,0)$. It follows that $U$ is conjugate, by the same element,  to $diag(\alpha(1-\beta),\alpha)$, from which it follows that $U(t,\lambda)\in\SU_2$.

\mn 2.  
%\note{Checked. -G. }
One calculates  that 
$$[A,\pi]={1\over \|\phi\|^2}\left(\begin{matrix}0&\phi_1\bar\phi_2 \\ -\bar\phi_1\phi_2 &0\end{matrix}\right),$$
 from which the stated formula follows easily. 
\end{proof}

The following theorem shows how the  Darboux transformation  $Q\mapsto \tilde Q$  is matched by a transformation $\Phi\mapsto \tilde\Phi$ of the solutions to the associated AKNS systems.   This is followed by a description of the effect of the transformation on the associated curves in $\R^3$. 

\begin{theorem}[\cite{RS}]$\Phi$ is a solution to the AKNS system  \eqref{eq:AKNS} if and only if     $\tilde \Phi:=U\Phi$ is a solution to the AKNS system 
\be\label{eq:AKNS2}
{\tilde \Phi}_t = (\tilde Q + i \lambda A) \tilde \Phi,
\ee
where $U,\tilde\Phi$ are given by equation \eqref{eq:darb} above. 
\end{theorem}

%\note{Not checked -G. }
The proof is by a straighforward calculation using the formulas above and is omitted. 

Now we look at the effect of the Darboux transformation $\Phi\mapsto \tilde \Phi$ on the associated STP curves. 
\begin{prop}\label{prop:dif} Let  $\tilde \Gamma =\tilde \Phi^*\tilde\Phi_\lambda$. Then 
%\note{Checked. G.} 
\be\label{eq:W}
\tilde \Gamma -\Gamma= {\mu-\bar\mu\over |\lambda -\mu|^2}\Phi^*( \pi    -\frac{1}{2}I       )\Phi,
\qquad \|\tilde \Gamma -\Gamma\|={|\mu-\bar\mu|\over \:|\lambda -\mu|^2}.
\ee
In particular, the distance $\|\tilde \Gamma(t,\lambda) -\Gamma(t,\lambda)\|$ is independent of  $t$. 
\end{prop}
\begin{proof} The formula for $\tilde \Gamma -\Gamma$ is a direct calculation using the definitions above and is omitted. To calculate 
 $\|\tilde \Gamma -\Gamma\|, $ note that  $\pi$, being an orthogonal projection, is conjugate by an element in $\SU_2$ to $diag(1,0)$, thus $ i (\pi -\frac{1}{2}I) $ is  conjugate by the same element to $iA$, which is of  unit norm.  This implies the stated formula. 
\end{proof}

The last proposition states   that the pair of STP curves $\Gamma(\cdot, \lambda), \tilde \Gamma(\cdot, \lambda)$  satisfy one of the  conditions needed  to be in   bicycle correspondence. To get the other condition, one needs to restrict $\mu$ and $\lambda$. 

\begin{theorem}\label{thm:new} If  we choose a  purely imaginary
$\mu = i\varepsilon$ in the Darboux trasnformation  described above and $\lambda=0$, then the pair of  
$\sutwo$-valued curves $\Gamma$ and $\tilde \Gamma$ are in $2/|\varepsilon|$-bicycle correspondence.
%\note{Not checked. G.}

Conversely, let  $\Gamma_1, \Gamma_2$ be  two parametrized curves in $\R^3$ in $2\ell$-bicycle correspondence. Then, there exists an AKNS system \eqref{eq:AKNS} with  initial conditions $\Phi(0,\lambda)\in\SU_2$ whose corresponding STP curve at $\lambda=0$  is $\Gamma_1$, and a   Darboux transform with $\mu=i/\ell$ mapping  $\Gamma_1$ to  $\Gamma_2$. 
\end{theorem}

\begin{proof}Consider a Darboux transformation with $\mu=i\varepsilon$ and associated STP curves $\Gamma, \tilde\Gamma.$ Let $W=\tilde \Gamma -\Gamma$. By Proposition \ref{prop:dif}, $\|W\|=2/|\varepsilon|.$ What is left to show then is that $W$  and $ (\Gamma_t + \tilde \Gamma_t)/2$ are parallel for $\lambda=0$, that is,  $[ W, \Gamma_t + \frac{1}{2}W_t]=0$. The proof is
by a simple but  lengthy computation. We omit the details.

Conversely, given two curves $\Gamma_1, \Gamma_2$ in $\sutwo$ in $2\ell$-bicycling correspondence, let $\kappa_1$ and $\tau_1$ be the curvature and torsion functions of $\Gamma_1$. According to Proposition \ref{prop:STP},  the AKNS system associated with $q=(\kappa_1/2)e^{i\int\tau_1}$, with appropriate initial conditions, realizes $\Gamma_1$ as the associated STP curve at $\lambda=0$.  We now show that  an appropriate  Darboux transformation maps $\Gamma_1$ to $\Gamma_2$. 

From the first part of the theorem, we know that   Darboux transformations with  $\mu=i/\ell$ produce curves in $2\ell$-bicycle correspondence with $\Gamma_1$. We use the expression for $\tilde \Gamma-\Gamma$ in formula \eqref{eq:W} of Proposition \ref{prop:STP} to show that,  by varying $v\in\C^2\setminus\{0\}$, we obtain {\em all} curves in $2\ell$-bicycling correspondence with $\Gamma_1$ and, in particular, $\Gamma_2$. By this formula, the direction of $\tilde \Gamma-\Gamma$ at $t=0$ is the unit vector  
$$B(v):=i\left( \frac{v v^*}{\|v\|^2}    -\frac{1}{2}I       \right)\in\sutwo,$$
 rotated by conjugation with $\Phi$. Now the map $v\mapsto B(v)$ is clearly $\SU_2$-equivariant, $B(gv)=gB(v)g^{-1}$. Its image is therefore the whole unit sphere in $\sutwo$ (the orbit of $iA$ under $\SU_2$). 
 
 It follows that every  initial direction of $\tilde \Gamma-\Gamma$ at $t=0$ can be  obtained by choosing $v$ appropriately and, consequently we obtain all curves $\Gamma_2$ in $2\ell$-bicycle correspondence with $\Gamma_1$. \end{proof}

\section{Proof of Proposition \ref{intint} }\label{app:pf}

\newcommand{\DD}{\mathrm{D}}
We will prove the statement of Proposition \ref{intint} for a class of ODEs that includes the bicycle Riccati equation (\ref{eq:ric7}).  For each   $ C^ \infty $ function  $f: \C \times{\mathbb R} \to \C $, complex analytic in the first variable and  $T$-periodic in the second, consider the ODE
\begin{equation} 
	 \dot Z= - \frac{1}{\ell} Z+f(Z,t),
	\label{eq:gode}
\end{equation}
where $\dot Z= \partial_t Z$. (Note that equation \eqref{eq:ric7} converts to  this form upon the change of  variable $t\mapsto -t$, which interchanges stable and unstable fixed points.)  We will show the following. 

\begin{prop}
For every function $f:\C\times\R\to\R$ as above, there exists an  $\ell_0>0$ 
such that for every $0<\ell<\ell_0$\\ 
(1) there is a unique periodic solution $Z(t, \ell)$ to equation \eqref{eq:gode} with  $|Z(t, \ell)|<1$ for all $t$.\\
 (2)  $Z(t,\ell)$ 
is a stable periodic solution.\\
 (3) Extended to $\ell=0$ as $Z(t,0)\equiv 0,$  $Z(t,\ell)$ is infinitely differentiable in    $\R\times [0,\ell_0)$.\\
  In particular, 
$\lim_{\ell\downarrow 0}\partial^n_t Z(t,\ell)=0$  and $\partial^n_\ell Z(t,0)$ exists for all $t\in\R$ and integer $n\geq 0$. 
\end{prop}

%\note{Added this sentence about uniqueness. -G.}
Note that the above existence result implies immediately the uniqueness statement in Proposition \ref{intint}, since a M\"obius transformation in $\PSLtc$ has at most one unstable  fixed point. 
We divide the  proof into the following steps.

\begin{enumerate}
\item There exists an $ \ell_0>0 $ (depending on $f$ alone) such that for all $ 0<\ell< \ell_0 $ the period map $\varphi:\C\to \C$ of equation \eqref{eq:gode} has  a unique fixed point in $\DD:=\{|Z|\leq 1\}.$ This fixed point is stable and the associated periodic solution $Z(t, \ell) $ is $ C^ \infty $ in $\R\times (0,\ell_0)$. 

\item  For all integers $ n \geq 0 $: 
\be
	 \lim_{\ell\downarrow 0} \partial_t^nZ(t,\ell) = 0.   
	\label{eq:lt}
\ee 

\item  For all integers $ n \geq 0 $, the limit 
\be
	 a_n(t):=\lim_{\ell\downarrow 0} \partial_\ell^nZ(t,\ell) 
	 \label{eq:ll}
 \ee
 exists. 
 
 \item  $Z(t,\ell)$, extended to $\ell=0$ by $Z(t,0)=0$, is $C^\infty$ in $\R\times [0,\ell_0)$ with  $\partial_\ell^nZ(t,0)=a_n(t).$

 \end{enumerate}

\mn {\bf Step 1.} Let $M=\max |f(Z,t)|$  over $|Z|\leq 1$ and $t\in\R$, and let $\ell_1>0$ be such that $1/\ell_1>M$. Then, for all $\ell<\ell_1,$ if 
$Z(t)$ is a solution to equation \eqref{eq:gode} 
and $|Z(t)|=1$ for some $t$, one has
\be \label{eq:zsq}
	\frac{1}{2} \frac{d}{dt} |Z|^2  = Z\cdot \dot Z \buildrel{ (\ref{eq:gode})}\over{=} -\frac{1}{\ell} | Z | ^2 + f\cdot Z\leq -\frac{1}{\ell}+{1\over \ell_1}<0. 
\ee 
It follows  that for $\ell<\ell_1$  the period map $ \varphi$  maps $\DD$  into itself, and thus has a fixed point. 

%\note{I added this phrase in this paragraph about stability.  -G. }
To prove uniqueness we show that, for $\ell$ small enough, $\varphi$ resricted to $\DD$ is a contraction. For this, it is enough to show that $|\varphi'(Z_0)|\leq \rho$ for some $\rho<1$ and all $Z_0\in\DD$.  This will show also that the fixed point is stable. 

 Choose an arbitrary solution $ Z(t) $ of equation \eqref{eq:gode} with $Z(0)\in \DD$ and consider a nontrivial solution $U(t)$ of the linearized equation
$  \dot U = - \frac{1}{\ell} U +  f _Z U $ around $Z(t) $; here $ f_Z=\partial_Z f(Z,t) $. The derivative of the period map $\varphi$ at $Z(0)$ is given by 
\[
	\varphi ^\prime (Z(0)) = \frac{U(T)}{U(0)} = \exp\int_{0}^{T} \left(- \frac{1}{\ell} + f_Z\right) \,dt.   
\]   
Now, given  $\eps>0$, there exists  $\ell_0<\ell_1$ such that  $\Re\left(- \frac{1}{\ell} + f_Z\right)<-\eps$  for all  $\ell\leq \ell_0$, hence $| \varphi ^\prime ( Z(0)) | \leq  e^{-T\eps}<1$, proving that $\varphi$ is a contraction in $\DD  $. 
Since $ \varphi $  is analytic, and the contraction is by a factor bounded away from $1$, we conclude that the fixed point is an analytic function of $ \ell $ for all $ 0<\ell< \ell_0$. 

\mn {\bf Step 2.} We prove   (\ref{eq:lt}) by induction on $n$. For $ n = 0 $, it suffices to show that for $\ell<\ell_1$, our periodic solution satisfies $ | Z(t,\ell) | \leq  \ell/\ell_1$. If $ | Z(t, \ell) |> \ell/\ell_1$ for some $t$, then (for this $t$), by  (\ref{eq:zsq}), 
\begin{equation} 
	\frac{1}{2} \frac{d}{dt} |Z|^2 = - \frac{1}{\ell} | Z | ^2 + f\cdot Z\leq  \left(- \frac{1}{\ell} | Z |   + {1\over \ell_1}\right) | Z | < 0.  
	\label{eq:dzsq}
\end{equation}  
It follows that   for all $\ell<\ell_1$ the periodic solution $Z(t, \ell)$ is confined to the disk 
$|Z|\leq\ell/\ell_1$ and thus $\lim_{\ell\downarrow 0}Z(t,\ell) = 0.$ This completes the step $n=0$ of the induction.    
 
Assume now that   (\ref{eq:lt})  holds up to  order $ n -1 $ for some $ n > 0$. 
 Let us denote $Y_k:=\partial_t^k Z $.  
 Differentiating   (\ref{eq:gode})   $n$ times by $t$, we obtain 
\[
	\dot Y_n=- \frac{1}{\ell} Y_n+ \partial_t ^n f. 
\]  
The last term is   of the form 
\[
	\partial_t ^n f = A  Y_n + B_n, 
\]
where $ A  =  f_Z(Z(t,\ell),t) $ and where $ B_n$   is a polynomial in the variables $ Z, Y_1, \ldots , Y_{n-1} $ (containing no $ Y_n$) with  coefficients of the form $ \partial_Z^i\partial_t^j f(Z,t) $ with $ i+j\leq n $. By the assumption on $f$, these coefficients are bounded. This  and    the inductive assumption imply that 
\begin{equation} 
	\lim_{\ell\downarrow 0 }A  =f_Z(0,t), \  \  \lim_{\ell\downarrow 0 }B_n = B_n^0, 
	\label{eq:AB}
\end{equation}   
where $ B_n^0=B_n^0(t) $ is a smooth function. 
 Summarizing, $ Y_n $ satisfies the ODE
 \begin{equation} 
	\dot Y_n= \left(-\frac{1}{\ell}+A\right) Y_n+ B_n 
	\label{eq:Zn}
\end{equation}   
with the coefficients satisfying   (\ref{eq:AB}). 
The same argument used in proving that $ \lim_{\ell\downarrow 0}Z=0 $ applies here, and it shows that $  \lim_{\ell\downarrow 0}Y_n=0 $. This completes the proof of   (\ref{eq:lt}). 

\mn {\bf Step 3.} We prove (\ref{eq:ll}) by induction on $n$. For  $ n=0 $, we already proved it in Step 2. For any fixed $ n> 0 $, assume that  (\ref{eq:ll}) holds for all  orders $ < n $. 
Let $ Z_k = \partial _\ell^k Z $.  
Multiplying both sides of    (\ref{eq:gode}) by $ \ell $ and differentiating $ n $ times with respect to $\ell$, we get 
\begin{equation} 
	\ell\dot Z_n+ n \dot Z_{n-1} = - Z_n +  \partial_\ell ^n ( \ell f ).
	\label{eq:Yn} 
\end{equation} 
Now  
\[
 	\partial_\ell^n (\ell f (Z, t)) = \ell \partial_\ell^nf+ n\partial_\ell^{n-1}f =  \ell A Z_n+\ell C_n+ n\partial_\ell^{n-1}f, 
\]  
where, similarly to Step 3, we have  $ A = f_Z$ and $ C_n $   is a polynomial  in $ Z, Z_1, \ldots Z_{n-1} $ with  coefficients of the form $ \partial_Z^j f(Z, t) $ with $ j\leq n $. By Step 3, 
$$  \lim_{\ell\downarrow 0}A \buildrel{def}\over{=}f_Z(0,t)\buildrel{def}\over{=} A^0, $$
 and  by the inductive assumption, 
\[
	  \lim_{\ell\downarrow 0}C_n \buildrel{def}\over{=} C_n^ 0 
\]  
exists. 
Substituting all this into   (\ref{eq:Yn})
yields, after some manipulation: 
\begin{equation} 
	\dot Z_n= -\frac{\alpha }{\ell} \left( Z_n- \beta  \right), 
	\label{eq:Yn1}
\end{equation} 
where 
\[
	\alpha = 1-\ell A, \  \  \beta =  \frac{ n\partial_\ell^{n-1}f -n \dot Z_{n-1}+\ell C_n}{\alpha }  . 
\]  
We have  $ \lim_{\ell\downarrow 0} \alpha = 1 $; to show that $ \lim_{\ell\downarrow 0} \beta$ exists, we   need to know that $\dot Z_{n-1} $ has a limit as $ \ell\downarrow 0 $. 
We showed that this limit  exists for $ \dot Z_0=\dot Z $; let us add the inductive assumption to the one already made  that     
$\lim_{\ell\downarrow 0} \dot Z_{n-1}$ exists; we will show in a moment that then $ \lim_{\ell\downarrow 0}\dot Z_n $ exists as well. 
With this assumption,  
\[
	     \lim_{\ell\downarrow 0} \beta = 
		n\lim_{\ell\downarrow 0}(\partial_\ell^{n-1}f - \dot Z_{n-1} ) 
\]   
exists. 
To complete the induction we must show that the limits of $ Z_n $ and of $ \dot Z_n $ exist. The existence of these limits follows from   (\ref{eq:Yn1}) and from the existence of the limits of $\alpha$ and $\beta$; we will in fact show that
 \[
	\lim_{\ell\downarrow 0} Z_n = \lim_{\ell\downarrow 0}\beta, 
\]    
and 
\begin{equation} 
	\lim_{\ell\downarrow 0}\dot Z_n = \lim_{\ell\downarrow 0} \dot \beta.
	\label{eq:ynd}
\end{equation}   
	
Indeed, 
\[
	\frac{1}{2} \frac{d}{dt} (Z_n- \beta )^2 \buildrel{  (\ref{eq:Yn1}) }\over{=} - \frac{\alpha }{\ell} (Z_n- \beta )^2 - \dot \beta \cdot  (Z_n- \beta ), 
\]  
which shows that $ | Z_n- \beta | $ is monotonically decreasing whenever   $  | Z_n- \beta | \geq 2 \max | \dot  \beta | \ell $. This shows that our periodic $ Z_n $ is confined to $ | Z_n -\beta| <     2 \max | \dot  \beta | \ell $ (by the argument used in Step 3), and thus converges to $\beta$ as $ \ell \downarrow 0 $, as claimed. 

Finally, differentiating   (\ref{eq:Yn1})  by $t$, one can apply a similar Lyapunov--type argument to prove the existence of the   limit in  (\ref{eq:ynd}).
This completes the induction step and thus the proof of   (\ref{eq:ll}).

\mn{\bf Step 4.} Consider the difference quotient
\[
	\frac{Z_n(t, \ell)- Z_n(t, 0)}{\ell}  = \frac{1}{\ell} \int_{0}^{\ell} Z_{n+1}(t,s)\,ds, 
\]  
where $ Z_n(t, 0) $ is the limit which exists by Step 3.
Since also $ \lim_{\ell\downarrow 0}Z_{n+1}(t, \ell) $ exists, so does the limit of the above integral, showing that 
\[
	\lim_{\ell\downarrow 0}\frac{Z_n(t, \ell)- Z_n(t, 0)}{\ell}= Z_{n+1}(t, 0), 
\]
and thus proving the claim.\qed

\end{document}